\documentclass[10pt,leqno,twoside]{amsart}

\usepackage{enumerate}

\usepackage{amssymb}

\usepackage[english]{babel}
\usepackage{amssymb,amsthm,amsmath,eucal,mathrsfs}
\usepackage{bm}

\usepackage{amsmath}
\usepackage{amsfonts}
\usepackage{amsmath}
\usepackage{amssymb}
\usepackage{graphicx}
\usepackage{lscape}
\usepackage{amstext}
\usepackage{amsthm}
\usepackage{color}
\usepackage{float}
\usepackage{mathrsfs}
\usepackage{epsfig}
\usepackage{url}
\usepackage{fancyhdr}
\usepackage{pspicture}
\usepackage{graphicx}

\usepackage{cite}

\usepackage{amsmath,verbatim}

\usepackage{amsthm}

\usepackage{amssymb}

\usepackage{amsfonts}

\usepackage{dsfont}

\usepackage{hyperref}

\newtheorem{theorem}{Theorem}[section]

\newtheorem{lemma}[theorem]{Lemma}

\newtheorem{proposition}[theorem]{Proposition}

\newtheorem{Hypotheses}{Hypotheses}

\theoremstyle{definition}

\newtheorem{definition}[theorem]{Definition}

\newtheorem{remark}[theorem]{Remark}

\numberwithin{equation}{section}

\newcommand{\diam}{\mathrm{diam}}      

\renewcommand{\Im}{{\ensuremath{\mathrm{Im\,}}}} 

\renewcommand{\Re}{{\ensuremath{\mathrm{Re\,}}}} 

\newcommand\restr[2]{{
  \left.\kern-\nulldelimiterspace 
  #1 
  \vphantom{\big|} 
  \right|_{#2} 
  }}

\usepackage{eucal}

\title[PA-imaging using resonating nano-particles]{Mathematical Analysis of the Photo-acoustic imaging modality using resonating dielectric nano-particles: The $2D$ TM-model}

\author[Ghandriche and Sini]{Ahcene Ghandriche  $^*$ and Mourad Sini$^{\ddag}$}
\thanks{$^*$ Altenbergerstrasse 69, A-4040, Linz, Austria. Email: ahcene.ghandriche@ricam.oeaw.ac.at. This author is supported by the Austrian Science Fund (FWF): P 30756-NBL.}
\thanks{$^{\ddag}$ RICAM, Austrian Academy of Sciences, Altenbergerstrasse 69, A-4040, Linz, Austria. Email: mourad.sini@oeaw.ac.at. This author is partially supported by the Austrian Science Fund (FWF): P 30756-NBL}

\begin{document}

\date{\today}

\allowdisplaybreaks

\begin{abstract} 
We deal with the photoacoustic imaging modality using dielectric nanoparticles as contrast agents. Exciting the heterogeneous tissue, localized in a bounded domain $\Omega$, with an electromagnetic wave, 
at a given incident frequency, creates heat in its surrounding which in turn generates an acoustic pressure wave (or fluctuations). The acoustic pressure can be measured in the accessible region $\partial \Omega$ surrounding 
the tissue of interest. The goal is then to extract information about the optical properties (i.e. the permittivity and conductivity) of this tissue from these measurements. We describe two scenarios. 
In the first one, we inject single nanoparticles while in the second one we inject couples of closely spaced nanoparticles (i.e. dimers). 
From the acoustic pressure measured, before and after injecting the nanaparticles (for each scenario), at two single points $x_{1}$ and $x_{2}$ of $\partial \Omega$ and two single times $t_{1} \neq t_{2}$ such that $t_{1,2} >  \diam (\Omega)$,
\bigskip

\begin{enumerate}

\item	we locatize the center point $z$ of the single  nanoparticles and reconstruct the phaseless total field $\vert u_0\vert$ on 
that point $z$ (where $u_0$ is the total field in the absence of
the nanoparticles).  Hence, we transform the photoacoustic problem into the 
inversion of phaseless internal electric fields. 

\item	we localize the centers  $z_1$ and $z_2$  of the injected dimers 
and reconstruct both the permittivity and the conductivity of the tissue on those points. 
\end{enumerate}

This can be done using {\it{dielectric}} nanoparticles enjoying high contrasts of both its electric permittivity and conductivity.

\bigskip

These results are possible using frequencies of incidence close to the resonances of the used dielectric nanoparticles. These particular frequencies are computable. This allows us to solve the photoacoustic
inverse problem with direct approximation formulas linking the measured pressure to the optical properties of the tissue.  The error of approximations are given in terms of the scales and the contrasts 
of the dielectric nanoparticles. The results are justified in the $2$D TM-model.

\end{abstract}


\maketitle

\section{Introduction and statement of the results}
\subsection{Motivation and the mathematical models}
Imaging using small scaled contrast agents has known in the recent years a considerable attention, see for instance \cite{B-B:2011,Chen-Craddock-Kosmas, Shea-Kosmas-VanVeen-Hagness}. 
To motivate it, let us recall that conventional 
imaging techniques, as microwave imaging, are known to be potentially capable of extracting features in breast cancer, for instance, in case of the relatively high contrast of
the permittivity, and conductivity, between healthy tissues and malignant ones, \cite{F-M-S:2003}. However, it is observed that in case of benign tissue, the variation of the permittivity is quite low so that such conventional
imaging modalities are limited to be used for early detection of such diseases. 
In these cases, creating such missing contrast is highly desirable. One way to do it is to use micro or nano scaled particles as contrast agents, \cite{B-B:2011,Chen-Craddock-Kosmas}. 
There are several imaging modalities using contrast agents as acoustic imaging using gas microbubbles, optical imaging and photoacoustic using dielectric or magnetic nanoparticles \cite{B-B:2011, F-M-S:2003, Q-C-F-2009}. 
The first two modalities are single wave based methods. In this work, we deal with the last imaging modality. 
\bigskip

Photoacoustic imaging is a hybrid imaging method which is based on coupling electromagnetic waves with acoustic waves to achieve 
high-resolution imaging of optical properties of biological tissues, \cite{P-P-B:2015, L-C:2015}. Precisely, exciting the heterogeneous tissue with an electromagnetic wave, at a certain
frequency related to the used small scale particles, creates heat in its surrounding which in turn generates an acoustic pressure wave (or fluctuations). The acoustic wave can be measured in a region surrounding the tissue of interest.
The goal is then to extract information about the optical properties of this tissue from these measurements, \cite{P-P-B:2015, L-C:2015}.
\bigskip

A main reason why such a modality is promising is that injecting nanoparticles, see \cite{B-B:2011, Chen-Craddock-Kosmas} for information on its feasibility, with appropriate scales between their sizes and optical properties, in the targeted tissue will create localized contrasts in the tissue and hence 
amplify the local electromagnetic energy around its location. This amplification can be more pronounced if the used incident electromagnetic wave is sent at frequencies close to resonances. 
In particular, dielectric or magnetic nanoparticles (as gold nanoparticles \cite{L-C:2015}) can exhibit such resonances when its inner electric permittivity or magnetic permeability is tuned appropriately, see below. 
 Our target here is to mathematically analyze this imaging technique when injecting such nanoparticles. 
 \bigskip
 
 To give more insight to this, let us briefly recall the photoacoustic model, 
see \cite{C-A-B:2007, Triki-Vauthrin:2017 , K-K:2010, N-S:2014, S:2010, S-U:2009} for extensive studies and different motivations of this model and related topics. 
We assume the time harmonic (TM) approximation for the electromagnetic model \footnote{Here, we describe the photoacoustic model assuming the TM-approximation of the electromagnetic field. 
The more realistic model is of course the full Maxwell system.}, then the third component of the electric field, that we denote by $u$, satisfies
\begin{equation}\label{helmotzequa}
\Delta u  +\omega^2 \varepsilon \mu_0 u=0,\; \mbox{ in } \mathbb{R}^2
\end{equation}
with $u:=u^i+u^s$ where $u^i:=u^i(x,d, \omega):=e^{i \omega d\cdot x}$, is the incident plane wave, sent at a frequency $\omega$ and direction $d,\; \vert d\vert=1$, and $u^s:=u^s(x, \omega)$ is the corresponding scattered wave
selected according to the outgoing Sommerfeld radiation conditions (S.R.C) at infinity. Here, $\mu_0$ is the magnetic permeability of the vacuum, which we assume to be a positive real constant, and $\epsilon:=\epsilon (x)$ is defined as
\begin{equation}\label{defvareps}
 \varepsilon(x):=\left \{
\begin{array}{llrr}
\epsilon_0, & in \quad   \mathbb{R}^2 \setminus \Omega,\\
\epsilon(x) & in \quad \Omega \setminus  \overset{M}{\underset{m=1}{\cup}} D_{m}, \\
\epsilon_m , & in \quad  D_m,
\end{array} \right.
\end{equation}
where $\epsilon_0$ is the positive constant permittivity of the vacuum and $\epsilon:=\epsilon_r +i\frac{\sigma_{\Omega}}{\omega}$ with $\epsilon_r$ as the permittivity  and $\sigma_{\Omega}$ the condutivity of 
the heterogeneous tissue (i.e. variable functions). The quantity $\epsilon_m$ is the permittivity constant of the particle $D_m$, of radius $a<<1$, which is taken to be complex valued, 
i.e.\;\, $\epsilon_m:=\epsilon_{m,r} +i\frac{\sigma_m}{\omega}$ where $\epsilon_{m,r}$ is its actual electric permittivity and $\sigma_m$ its conductivity. The bounded domain $\Omega$ models the region of the tissue of interest. We take the nanoparticle of dielectric type, meaning that  
$\frac{\epsilon_m}{\epsilon_0}>>1$ when $a <<1$, and hence its relative speed of propagation
is very large as well. Under particular rates of the ratio $\frac{\epsilon_m}{\epsilon_0}>>1$, resonances can occur, as the dielectric (or Mie-electric) resonances. These regimes will be of particular interest to us. 
Here, we take the $D_{m}$'s of the form $D_{m}:= z_{m} + a \, B_{m}$ where $z_{m}$ models its location, $a$ its radius and $B_{m}$ as a smooth domain of radius $1$ containing the origin. 

\bigskip

As said above, exciting the tissue with such electromagnetic waves will generate a heat $T$ which in turn generates acoustic pressure. Under some appropriate conditions, see \cite{Habib-lectures:2017, Triki-Vauthrin:2017} for instance, 
this process is modeled by the following system:

\begin{equation*}\label{Photoacoustic-general-model}
\left \{
\begin{array}{llrr}
\rho_0 c_p\frac{\partial T}{\partial t}-\nabla \cdot \kappa \nabla T=\omega \Im(\epsilon)\vert u \vert^2 \delta_{0}(t),\\
\\
\frac{1}{c^2}\frac{\partial^2 p}{\partial t^2}-\Delta p= \rho_0 \beta_0\frac{\partial^2 T}{\partial t^2} 
\end{array} \right.
\end{equation*}
 where $\rho_0$ is the mass density, $c_p$ the heat capacity, $\kappa$ is the heat conductivity, $c$ is the wave speed 
and $\beta_0$ the thermal expansion coefficient. 
To these two equations, we supplement the homogeneous initial conditions: 
\begin{equation*}
 T=p=\frac{\partial p}{\partial t}=0, \mbox{ at } t=0.
\end{equation*}
Under additional assumptions on the smallness of the heat conductivity $\kappa$, one can neglect the term $\nabla \cdot \kappa \nabla T$ and hence, we end up with the photoacoustic model linking the electromagnetic field to the acoustic pressure \footnote{We stated the model in the whole plan $\mathbb{R}^2$. However, we could also state it in a bounded domain supplemented with Dirichlet 
or Neumann boundary conditions.}:
\begin{equation}\label{Waveequa}
 \left \{
\begin{array}{llrr}
\frac{\partial^2 p}{\partial t^2} - c_{s}^2 \Delta p=0, & in \quad   \mathbb{R}^2 \times \mathbb{R}_+,\\
p(x, 0)= \frac{\omega\beta_0}{c_p} \Im(\varepsilon)\vert u \vert^2& in \quad \mathbb{R}^2, \\
\frac{\partial p}{\partial t}(x, 0)=0, & in \quad  \mathbb{R}^2.
\end{array} \right. 
\end{equation}
\bigskip

The imaging problem we wish to focus on is stated in the following terms:
\bigskip

{\bf{Problem}}. Reconstruct the coefficient $\epsilon$ from the given pressure $p(x, t)$ measured for $(x, t) \in \partial \Omega \times (0, T)$, with some positive time length $T$,
\bigskip

\begin{enumerate}
 \item after injecting single nanoparticles located in a sample of points in $\Omega$,
 \bigskip
 
 or/and
 \bigskip
 
 \item after injecting couples of nanoparticles two by two closely spaced (i.e. dimers) and located in a sample of points in  $\Omega$.
\end{enumerate}

\bigskip

It is natural to split this problem into two steps. The first step concerns the acoustic inversion, namely to reconstruct the source term $\Im(\varepsilon)\vert u\vert^2, \;~ x \in \Omega,$
from the pressure  $p(x, t)$ for $(x, t) \in \partial \Omega \times (0, L)$. The second step concerns the electromagnetic inversion, namely to reconstruct $\epsilon$ from the internal data  
$\Im(\varepsilon)\vert u\vert^2$.

\subsection{The acoustic inversion}\label{photo-acoustic-inversion-known-results}

We start by recalling the main results related to the model $(\ref{Waveequa})$. More informations about this part can be found in \cite{!PATTAT} and \cite{KuchmentKunyansky}.

For this inversion, there are two cases to distinguish: 
\begin{enumerate}
\item[\underline{Case 1:}] The speed of propagation $c_{s}$ is constant everywhere in $\mathbb{R}^{2}$ and $\Omega$ is a disc.\\ 
The solution of the problem $(\ref{Waveequa})$ is given by the Poisson formula
\begin{equation}\label{pressursol}
p(x,t) = \frac{\omega \, \beta_{0}}{2\pi c_{s} c_{p}} \partial_{t} \Bigg( \int_{\vert x-y \vert < c_{s} t} \frac{\Im(\varepsilon)(y) \, \vert u \vert^{2}(y)}{\sqrt{c^{2}_{s}t^{2}-\vert x-y \vert^{2}}} dy\Bigg).
\end{equation}
We denote by $M(f)$ the circular means of $f$ 
\begin{equation*}
 M(f)(x,r) := \frac{1}{2\pi} \int_{\vert \xi \vert=1} f(x+r\xi) \, d\sigma(\xi). 
\end{equation*}
The equation $(\ref{pressursol})$ takes the following form 
\begin{equation*}
p(x,t) = \frac{\omega \, \beta_{0}}{c_{s} c_{p}} \partial_{t} \Bigg( \int_{0}^{c_{s}t} \frac{r}{\sqrt{c^{2}_{s}t^{2}-r^{2}}} M(   \Im(\varepsilon) \, \vert u \vert^{2})(x,r) dr \Bigg).
\end{equation*}
The recovery of $Im(\varepsilon) \, \vert u \vert^{2}$ from $p(x,t), (x,t) \in \partial \Omega \times [0,T]$, is done in two steps. First,  
as $\partial \Omega$ is a circle, the circular means can be recovered from the pressure as follows 
\begin{equation}\label{Abelequa}
M(Im(\varepsilon) \, \vert u \vert^{2})(x,r) = \frac{2 \omega \beta_{0}}{c_{p} \pi} \int_{0}^{c_{s}r} \frac{p(x,t)}{\sqrt{r^{2}-t^{2}}} dt, \quad x \in \partial \Omega.
\end{equation}
Second, if $ Im(\varepsilon) \, \vert u \vert^{2} \in C^{\infty}(\mathbb{R}^{2})$  with $supp(Im(\varepsilon) \, \vert u \vert^{2})\subset \overline{\Omega}$, then, for $x \in \Omega$, 
\begin{equation}\label{N}
Im(\varepsilon)(x) \, \vert u \vert^{2}(x) = \frac{1}{2 \pi R_{0}}  \int_{\partial \Omega} \int_{0}^{2R_{0}} (\partial_{r} \, r \, \partial_{r} M(Im(\varepsilon) \, \vert u \vert^{2}))(p,r) \, \log(\vert r^{2} - \vert x-p \vert^{2} \vert) \, dr \, d\sigma(p). 
\end{equation} 

We can find in \cite{Natterer} and \cite{FHR}  the justification of $(\ref{Abelequa})$ and $(\ref{N})$ respectively.

\bigskip

\item[\underline{Case 2:}] The speed of propagation is variable in $\Omega$ and constant in $\mathbb{R}^{2}\setminus\Omega$, with $\Omega$ not necessarily a disc. However, the following assumptions are needed, namely 
(1). $Supp(\Im(\varepsilon) \, \vert u \vert^{2})$ is compact in $\Omega$, (2). $c(x) > c > 0$ and $Supp(c(x)-1)$ is compact in $\Omega$ and (3). the non trapping condition is verified. In $\mathbb{L}^{2}(\Omega; c^{-2}_{s}(x)dx)$, we consider the operator given by the differential expression $A = -c^{-2}_{s}(x) \Delta$ and the Dirichlet boundary condition on $\partial \Omega$. 
This operator is positive self-adjoint operator, and has discrete spectrum $\lbrace s^{2}_{k} \rbrace_{k \geq 1}$ with a basis set of eigenfunctions 
$ \lbrace \psi_{k} \rbrace_{k \geq 1}$ in $\mathbb{L}^{2}(\Omega; c^{-2}_{s}(x)dx)$. Then, the function $\Im(\varepsilon)(x) \, \vert u \vert^{2}(x)$ can be reconstructed inside $\Omega$ from the data $p$, 
as the following $\mathbb{L}^{2}(\Omega)$ convergent series 
\begin{equation*}
\Im(\varepsilon)(x) \, \vert u \vert^{2}(x) = \frac{c_{p}}{\omega \, \beta_{0}} \, \sum_{k} (\Im(\varepsilon)(x) \, \vert u \vert^{2})_{k} \psi_{k}(x),
\end{equation*}
where the Fourier coefficients $(\Im(\varepsilon)(x) \, \vert u \vert^{2})_{k}$ can be recovered as: 
\begin{equation*}
(\Im(\varepsilon)(x) \, \vert u \vert^{2})_{k} = s^{-2}_{k} p_{k}(0) - s^{-3}_{k} \int_{0}^{\infty} \sin(s_{k}t) \, p^{\prime \prime}_{k}(t) dt,
\end{equation*}
with  
\begin{equation*}
p_{k}(t) := \int_{\partial \Omega} p(x,t)  \frac{\partial \overline{\psi_{k}}}{\partial \nu}(x) dx.
\end{equation*}
\end{enumerate} 
More details can be found in \cite{!PATTAT}.
\bigskip

In our work, we address the following two situations regarding the types of the used dielectric nanoparticles.
\bigskip

\begin{enumerate}

\item {\it{Only the permittivity $\epsilon_{m,r}$ of the nanoparticle is contrasting.}} For this case, we use the results above on the acoustic inversion to obtain $\Im(\varepsilon)(x) \, \vert u \vert^{2}(x), x\in \Omega$ 
and hence $\vert u \vert, x\in D_m$, as $\Im \varepsilon=\frac{\sigma_{m}}{\omega}$ on $D_m$ which is known. With this information, we perform the electromagnetic inversion to reconstruct $\epsilon_{r}$ and $\sigma_{\Omega}$. 

\bigskip

\item {\it{Both the permittivity $\epsilon_{m,r}$ and the conductivity $\sigma_m$ of the nanoparticle are contrasting.}} In this case, we do not rely on the acoustic inversion results above.
Instead, we propose direct approximating formulas to link the measured data $p(x, t)$ for $x\in \partial \Omega$ and $t\in (0, T)$, to $\vert u \vert(x)$, $x \in D_{m}$. 
Actually, we need only to measure $p(x, t)$ on two single points on $\partial \Omega$ for two distinct times $t_{1}$ and $t_{2}$. Then, we perform the electromagnetic inversion.

\end{enumerate}

\bigskip
 
\subsection{The electromagnetic inversion and motivation of using nearly resonant incident frequencies}\label{electromagnetic-inversion}

We start from the model
\begin{equation}\label{prblm}
\left\{%
\begin{array}{lll}
    (\Delta + \omega^2 n^2) u = 0  & in & \mathbb{R}^{2}\\
    u := u^{i} + u^{s} & and  & u^{s} \quad S.R.C 
    \end{array}%
\right.
\end{equation}
where, taking $M=1$ in ($\ref{defvareps}$), 
\begin{equation*}
n := \left\{%
\begin{array}{llll}
    n_{p}& = \sqrt{\epsilon_{p} \mu_{0}} & in & D\\
    n_{0}& = \sqrt{\epsilon_{0} \mu_{0}} & in & \mathbb{R}^{2}  \setminus D.
\end{array}%
\right.
\end{equation*}
We set \,
$\varepsilon_{p}-\varepsilon_{0} = \tau,~~ \tau >> 1.$ 
Then, we obtain
\begin{equation*}
n^{2}-n^{2}_{0} = \left\{%
\begin{array}{llll}
    \mu_{0} \, \tau & in & D\\
    0 & in & \mathbb{R}^{2}  \setminus D.
    \end{array}%
\right.
\end{equation*}
We call the dielectric (or \emph{Mie-electric}) resonances the possible eigenvalues of $(\ref{prblm})$, 
i.e. the possible solutions $(\omega,u^{s})$ of $(\ref{prblm})$ when $u^{i}=0$. 
It is known from the scattering theory, precisely Rellich's lemma, that those eigenvalues belong to the lower complex plane $\mathbb{C}_{-}$.
However, as $\tau >>1$, and $a<<1$, their imaginary parts tend to zero, see \cite{A-D-F-M-S:2019} for instance.
Using the Lippmann-Schwinger equation (L.S.E), such eigenvalues are also characterized by the equation
\begin{equation}\label{defLSE}
u(x) = - \omega^{2} \int_{D} (n_{0}^{2}-n^{2}) G_{k}(x,y) \, u(y) dy, \quad x \in \mathbb{R}^{2},
\end{equation}
where $G_k$ is the Green's function satisfying $(\Delta + \omega^2 n^2) G_k =-\delta$ with the S.R.C, and $k:=\omega n$ is the wave number.
As $\epsilon_{p}$ is constant in $D$, and assuming $\epsilon$ to be constant in $\Omega$ for simplicity of the exposition here, we get from $(\ref{defLSE})$ 
\begin{equation}\label{spe}
u(x) \frac{1}{\omega^{2} \mu_{0} \tau} =  \int_{D} G_{k}(x,y) u(y) dy, \quad x \in \mathbb{R}^{2}.
\end{equation}
To solve $(\ref{spe})$, it is enough to find and compute eigenvalues $w_{n}(k)$ of the volumetric potential operator $A_{k}$ defined as 
\begin{equation}\label{New}
A_{k}(u)(x) :=  \int_{D} \Phi_{k}(x,y) u(y) dy,~ u \in \mathbb{L}^2(D). 
\end{equation} 
Then combining ($\ref{spe}$) and ($\ref{New}$), we can write $
A_{k}(u) = \frac{1}{\omega^{2} \mu_{0} \tau} \; u$
and then solve in $\omega$, and recalling that $k=\omega\; n$, the dispersion equation 
\begin{equation}\label{L}
w_{n}(k) = \frac{1}{\omega^{2}\mu_{0}\tau}.
\end{equation}
Let us now recall that the operator $LP$, called the Logarithmic Potential operator, defined by
\begin{equation*}
LP(u)(\eta) :=  \int_{B} -\frac{1}{2 \pi} \log\vert \eta - \xi \vert \; u(\xi) \; d\xi, u \in \mathbb{L}^2(B), \eta \in B, 
\end{equation*}
has a countable sequence of eigenvalues with the corresponding eigenfunctions as a basis of $\mathbb{L}^2(B)$. For more details see \cite{RG} and \cite{AKL}.
Correspondingly, we define $A_{0}$ to be 
\begin{equation}\label{U}
A_{0}(u)(x) :=  \int_{D} -\frac{1}{2 \pi} \log\vert x - y \vert \; u(y) \; dy,  u \in \mathbb{L}^2(D), x \in D.
\end{equation}

Rescaling, we have
$ A_{0}(u)(x)=a^2 LP \tilde{u} (\xi)-\frac{a^2\log(a)}{2\pi}\int_B \tilde{u}(\xi)d\xi,~~ \xi:=\frac{x-z}{a}.$
Hence the eigenvalue problem $A_0(u)=\lambda_n u$, on $D$, becomes
\begin{equation*}
 LP \tilde{u} -\frac{\log(a)}{2\pi}\int_B \tilde{u}(\eta)d\eta=\frac{\lambda_n}{a^2}\tilde{u}, ~ \mbox{ on } B.
\end{equation*}

We observe that the spectrum $Spect(A_0\lvert_{\mathbb{L}_{0}^2(D)})$ of $A_0$, restricted to $\mathbb{L}_0^2(D):=\{v \in \mathbb{L}^2(D),~ \int_{D}v(x)~dx=0\}$, is characterized by 
$Spect(A_0\lvert_{\mathbb{L}_0^2(D)})
=a^{-2} Spect(LP\lvert_{\mathbb{L}_0^2(B)})$. However, as we see it later, the important eigenvalues are those for which the corresponding eigenfunctions are not average-zero.
Therefore, we need to handle the other part of the spectrum of $A_0$ as well. As $\mathbb{L}_0^2(D)$ is not invariant under the action of $A_0$, the natural decomposition $\mathbb{L}^2(D)= \mathbb{L}_0^2(D) \oplus 1$ does not decompose it.
\bigskip

The following properties are needed in the sequel and we state them as hypotheses to keep a higher generality.

\begin{Hypotheses}\label{hyp}
The particles $D$, of radius $a,\; a<<1$, are taken such that the spectral problems $A_0 u =\lambda\; u, \mbox{ in } D$, have eigenvalues $\lambda_n$ and corresponding eigenfunctions, $e_n$, satisfying the following properties:
\newline

\begin{enumerate}
 \item $\int_D e_n(x) dx \neq 0, \; \forall a <<1.$
 \bigskip
 
 \item $\lambda_n \sim a^2 \vert \log(a)\vert, \; \forall a <<1.$ 
\end{enumerate}
\end{Hypotheses}

In the appendix, see section \ref{the hypotheses-justification}, we show that for particles of general shapes, the first eigenvalue and the corresponding eigenfunction satisfy {\bf{Hypotheses}} \ref{hyp}. 
In addition, we characterize the properties of the eigenvalues for the case when $D$ is a disc.
\bigskip

Since, the dominant part of the operator $A_{k}$ defined in $(\ref{New})$ is $A_{0}$,   
we can write\footnote{More exactly, using the expansion and the scales of the fundamental solution, we show that an eigenvalue of $A_{k}$ can be written as
\begin{equation*}
a^{2} \, \Big( {\tilde{\lambda}_{n}} + \frac{1}{2} \vert \log(a) \vert (\int_{B}\bar{e_n}(\xi) d\xi)^2 - \frac{1}{2} \log(k)(z) + \pi \, \Gamma \Big) + \mathcal{O}(a^{3}).
\end{equation*}
}
\begin{equation}\label{B1}
w_{n}(k) = \lambda_{n} + \mathcal{O}(a^{2}).
\end{equation}

Combining $(\ref{B1})$, $(\ref{T})$ and $(\ref{L})$, we get
${\lambda_{n}} = \displaystyle\frac{1}{\omega^{2}\, \mu_{0} \, \tau \, a^{2}} + \mathcal{O}(1)$
or 
$\omega^{2} = \displaystyle\frac{1}{\mu_{0} \, \tau \, \lambda_{n}} + \mathcal{O}(\vert \log(a) \vert^{-1}).$
This means that $(\ref{prblm})$ has a sequence of eigenvalues that can be approximated by 
\begin{equation*}
\frac{1}{\mu_{0} \, \tau \, \lambda_{n}} + \mathcal{O}(\vert \log(a) \vert^{-1}).
\end{equation*}
The dominating term is finite if the contrast of the used nanoparticle's permittivity behaves as $\tau \sim \lambda^{-1}_n \sim a^{-2} \vert \log(a)\vert^{-1} $ for $a<<1$.

\bigskip

We distinguish two cases as related to our imaging problem.

\begin{enumerate}

\item Injecting one nanoparticle and then sending incident plane waves at real frequencies $\omega$ close to the real values
\begin{equation}\label{Number}
\omega_{n} := \left( \mu_{0} \, \tau \, \lambda_{n} \right)^{-1/2} ,
\end{equation}
we can excite, approximately, the sequence of eigenvalues described above. As a consequence, see the justification later, if we excite with incident frequencies near 
$w_{n}, n \in \mathbb{N}$, the total field $u$ solution of $(\ref{prblm})$, restricted to $D$ will be dominated by 
$\int_{D} u_0(x) \, e_{n}(x) \, dx \,~ e_{n}(x),$
which is, in turn, dominated by 
$u_0(z) \, e_{n}(x) \, \int_{D}  e_{n}(x) \, dx$
where $u_0$ is the wave field in the absence of the nanoparticles, i.e.
$(\Delta + \omega^2 n_0^2) u_0 =0, \;~ u_0=u^{i}+u^s_0$
and $u_0^s$ satisfies the S.R.C. Hence from the acoustic inversion, i.e. from the knowledge of $\Im(\epsilon)(x)\vert u \vert(x), x\in \Omega$, 
and hence $\vert u\vert(x), x\in D$, as, for $x\in D$, $\Im(\epsilon)=\Im(\epsilon_p)=\frac{\sigma_p}{\omega}$ is known, 
we can reconstruct
\begin{equation*}
\big\vert u_0(z) \big\vert \, \big\vert e_{n}(z) \big\vert \, \Bigg\vert  \int_{D}  e_{n}(x) \, dx \Bigg\vert.
\end{equation*}
\bigskip

As $e_{n}$ and $D$ are in principle known, then we can recover the total field $\vert u_0(z) \vert$.  
Taking a sampling of points $z$ in $\Omega$, we get at hand the phaseless internal total field $\vert u_0(z) \vert$, $z \in \Omega$.
\bigskip

\item Now, we inject a dimer, meaning a couple of close nanoparticles, instead of only single particles, with prescribed high contrasts of the relative permittivity or/and conductivity. 
Sending incident plane wave at frequencies close to the dielectric resonances, 
we recover also the amplitude of the field generated by the first interaction of 
the two nano-particles. 
Indeed, based on point-approximation expansions, this field can be approximated by the Foldy-Lax field. This field describes the one due to multiple interactions between the nanoparticles. 
We show that the acoustic inversion approximately reconstruct the first multiple interaction field (i.e. the Neumann series cut at the first, and not the zero, order term). 
From this last field, we recover the value of $\vert \varepsilon_{0}(z) \vert, z \in \Omega$.\\

\end{enumerate}

Both steps are justified using the incident frequencies close to the dielectric resonance of the nanoparticles. This wouldn't be possible using incident frequencies away from these resonances.
\bigskip

Recall that $\displaystyle\epsilon_0=\epsilon_{r} +i \frac{\sigma_{\Omega}}{\omega}$, then $\displaystyle\vert \epsilon_{0} \vert^2=\epsilon_{r}^2 + \frac{\sigma_{\Omega}^2}{\omega^2}$. Hence using two different dielectric resonances, we can reconstruct both the permittivity $\epsilon_r$ and the conductivity $\sigma_{\Omega}$.

\subsection{Statement of the results}
We recall that the mathematical model of the photoacoustic imaging modality is $(\ref{helmotzequa}), (\ref{defvareps})$ and $(\ref{Waveequa})$. \\

Next, we set $u:=u_j, j=0, 1, 2$, the solution of $(\ref{helmotzequa})$ and $(\ref{defvareps})$ when there is no nanoparticle injected, there is one or two nanoparticles, respectively (i.e take $M=0,1$ or $2$ in $(\ref{defvareps})$).\\
To keep the technicality at the minimum, we deal only with the case when the electromagnetic properties of the injected particles are the same i.e, $\epsilon_{1} = \cdots = \epsilon_{M}$. 

\subsubsection{Imaging using dielectric nanoparticles with permittivity contrast only}\label{only-contrasted-permittivity}
Let the permittivity $\epsilon$, of the medium, be $W^{1, \infty}-$smooth in $\Omega$ and the permeability $\mu_0$ to be constant and positive.
Let also the injected nanoparticles $D$ satisfy {\bf{Hypotheses}} \ref{hyp}. We assume these nanoparticles to be characterized with moderate magnetic permeability and their permittiviy and conductivity are 
such that $\epsilon_{m,r} \sim a^{-2} \, \vert \log(a) \vert^{-1}$ while $\sigma_{m} \sim 1$ as $a << 1$.  The frequency of the incidence $\omega$ is chosen close to the dielectric resonance $\omega_{n_0}$
\begin{equation*}\label{resoance-n-0}
\omega^2_{n_0}:= \left( \mu_0 \, \tau \, \lambda_{n_0} \right)^{-1},
\end{equation*}
as follows
\begin{equation*}\label{resoance-n-0}
\omega^2=\omega^2_{n_0}(1\pm \vert \log(a)\vert^{-h}),\;~~ 0<h<1.\;\, \footnote{Choosing + or - does not make a difference for the results in Theorem \ref{Using-only-permittivity-contrast}.}
\end{equation*}

\begin{theorem}\label{Using-only-permittivity-contrast}
We assume that the acoustic inversion is already performed using one of the methods given in section \ref{photo-acoustic-inversion-known-results}. Hence, we have at hands
$$
\vert u_j\vert(x), x \in D, j=1, 2.
$$

\begin{enumerate}
\item {\it{Injecting one nanoparticle.}} In this case, we use the data $\vert u_1\vert(x), x \in D$. We have the following approximation
\begin{equation}\label{one-particle-reconst-permittivity-only}
\int_D\vert u_1\vert^2(x) dx =\frac{\vert u_0\vert^2(z) (\int_D e_{n_0}(x) dx)^2}{\vert 1-\omega^2 \mu_0 \tau \lambda_{n_0}\vert^2} + \mathcal{O}\big( a^{2} \big).
\end{equation}
\bigskip

\item {\it{Injecting two closely spaced nanopartilces.}} These two particles are distant to each other as 
$$
\vert z_1-z_2\vert \geq exp(-\vert \log(a)\vert^{1-h}), a<<1,
$$
where $z_1$ and $z_2$ are the location points of the particles. In this case, we use as data $\vert u_j\vert(x), x \in D$ $j=1, 2$, where $D$ is any one of the two particles. The following expansion is valid

\begin{equation}\label{reconstruction-k-using-contras-permittivity-only}
\log (\vert k\vert)(z)=2\pi \gamma -
\frac{A_1-(1-C\Phi_0)^2}{A_1-2(1-C \Phi_0)}+O(\vert \log(a)\vert^{\max\{h-1, 1-2h\}}),
\end{equation}
where $\gamma$ is the Euler constant, 
\begin{equation*}
 A_1:=\frac{\int_D\vert u_2\vert^2(x) dx}{\int_D\vert u_1\vert^2(x) dx},\;~~ \Phi_0:=\frac{-1}{2\pi} \log \vert z_1-z_2\vert 
\end{equation*}
 and 
\begin{equation*}
\textbf{C}:=\int_D[\frac{1}{\omega^2 \mu_0 \tau}I-A_0]^{-1}(1)(x)dx=\frac{\omega^2 \mu_0 \tau}{1-\omega^2 \mu_0 \tau \lambda_{n_0}}\Bigg(\int_D e_{n_0}(x) dx\Bigg)^2 +O(\vert \log(a) \vert^{-1}).
\end{equation*}

\end{enumerate}

\end{theorem}

From the formula (\ref{one-particle-reconst-permittivity-only}), we can derive an estimate of the total field in the absence of the nanoparticles, i.e. $\vert u_0\vert(x), x \in \Omega$, 
by repeating the same experiment scanning the targeted tissue located in $\Omega$ by injecting single nanoparticles. Hence, we transform the photoacoustic inverse problem to the reconstruction of
$\epsilon_0$ in the equation $(\Delta +\omega^2 \mu_0 \epsilon_0)u_0=0,$ in $\mathbb{R}^2$, from the phaseless internal data $\vert u_0\vert(x), x\in \Omega$.
\bigskip

From the formula (\ref{reconstruction-k-using-contras-permittivity-only}), we can reconstruct $\vert k\vert (z)$ using the data $\vert u_1\vert(x)$ and $\vert u_2\vert(x)$, with $x \in D$. 
Indeed,
\begin{equation*}
\vert k\vert (z)=\omega^2 \vert \epsilon_0\vert \mu_0=\omega^2 \Big(\vert \epsilon_{r}\vert^2 +\frac{\vert \sigma_{\Omega} \vert^2}{\omega^2}\Big)^{1/2} \, \mu_{0},
\end{equation*} 
then using two different resonances $\omega_{n_0}$ and $\omega_{n_1}$, we can reconstruct both the permittivity $\epsilon_0(z)$ and the conductivity $\sigma_{\Omega}(z)$.

\subsubsection{Imaging using dielectric nanoparticles with both permittivity and conductivity contrasts}\label{both-contrasted-permittivity-conductivity}
As in section \ref{only-contrasted-permittivity}, let the permittivity $\epsilon$, of the medium, be $W^{1, \infty}-$smooth in $\Omega$ and the permeability $\mu_0$ to be constant and positive.
Let also the injected nanoparticles $D$ satisfy {\bf{Hypotheses}} \ref{hyp}.
Here, we assume that $\epsilon_{m,r} \sim a^{-2} \vert \log(a) \vert^{-1}$
and $\sigma_m\sim a^{-2} \vert \log(a) \vert^{-1-h-s}$, $s \geq 0$.
\, The frequency of the incidence $\omega$ is chosen close to the dielectric resonance $\omega_{n_0}$
\begin{equation*}\label{resoance-n-0}
\omega^2_{n_0} := \left( \mu_0 \, \tau \, \lambda_{n_0} \right)^{-1}
\end{equation*}
as follows
\begin{equation}\label{resoance-n-0-2}
\omega^2=\omega^2_{n_0}(1\pm \vert \log(a)\vert^{-h}),\;~~ 0<h<1.
\end{equation}

\begin{theorem}\label{Using-both-permittivity-and-conductivity-contrasts}

Let $x \in \partial \Omega$ and $t \geq  diam(\Omega)$. We have the following expansions of the pressure:

\begin{enumerate}
\item {\it{Injecting one nanoparticle}}. In this case, we have the expansion

\begin{equation}\label{pressure-to-v_0}
(p^{+} + p^{-} - 2p_{0})(t,x)  =  \frac{-t \, \omega \, \beta_{0}}{c_{p} \, (t^{2}-\vert x-z \vert^{2})^{3/2}}  \frac{2\Im(\tau) \vert u_0(z) \vert^{2}}{\vert 1- \omega^{2} \mu_{0} \lambda_{n_{0}} \tau \vert^{2}}  \Big(\int_{D} e_{n_{0}} dx \Big)^{2} + \mathcal{O}\big(\vert \log(a) \vert^{\max(-1,2h-2)}\big),
\end{equation}
under the condition $0\leq s < \max\{h,\; 1-h\}$, where $p^+$ and $p^-$ correspond to the pressure after injecting the nanoparticles and exciting with frequencies of incidence (\ref{resoance-n-0-2}) while $p_0$ is the pressure in the absence of the nanoparticles.
\bigskip

\item {\it{Injecting two close dielectric nanoparticles.}}
These two particles are distant to each other as 
\begin{equation*}
\vert z_1-z_2\vert \geq exp(-\vert \log(a)\vert^{1-h}), a<<1,
\end{equation*}
where $z_1$ and $z_2$ are the location points of the particles.
We set 
\begin{equation*}\label{pressure-tilde}
\tilde{p}(t, x):=(p^+-p_{0})(t,x) +\frac{1-\omega_{n_0}^2}{1+\omega_{n_0}^2}(p^--p_0)(t,x)
\end{equation*}
then we have the following expansion\footnote{Since $z_{1}$ and $z_{2}$ are sufficiently close, we make in $(\ref{pressure-tilde-expansion})$ an arbitrary choice of one of 
them, i.e. (\ref{pressure-tilde-expansion}) does not distinguish between $z_{1}$ and $z_{2}$.}
\begin{equation}\label{pressure-tilde-expansion}
\tilde{p}(t, x) = \frac{\omega \, \beta_{0}}{c_{p}} \, \frac{-t \, }{(t^{2}-\vert x-z_{2} \vert^{2})^{\frac{3}{2}}} \, \, \frac{4 \; \Im(\tau)}{1+ \omega_{n_0}^{2}} \,   \, \Big(\int_{D} u_2(x)e_{n_{0}}(x) dx \Big)^{2} + \mathcal{O}\big(\vert \log(a) \vert^{\max(-1,2h-2)}\big),
\end{equation}
where $D$ is any one of the two nanoparticles. 
\end{enumerate}
\end{theorem}
\
\\
\newline
The formula (\ref{pressure-to-v_0}) means that if we measure before and after injecting one nanoparticle, then we can reconstruct the phaseless data $\vert u_0 \vert(x), x \in \Omega$. Hence, we transform the photoacoustic inverse problem to the inverse scattering using phaseless internal data.

\bigskip

The formula (\ref{pressure-tilde-expansion}) can be expressed using $u_0$ instead of $u_2$ under the condition $0\leq s < \max\{h,\; 1-h\}$ as for (\ref{pressure-to-v_0}).
The formula (\ref{pressure-tilde-expansion}) means that if we measure before and after injecting two closely spaced nanoparticles, 
then we can reconstruct $\int_{D} u_2(x)e_{n_{0}}(x) dx $ and hence $\int_{D} \vert u_2(x) \vert^2 dx$. In addition, a slightly different form of formula (\ref{pressure-to-v_0}), see $(\ref{abxyz})$,
\begin{equation*}
(p^{+} + p^{-} - 2p_{0})(t,x)  =  \frac{-2 \, t \; \Im(\tau) \; \vert \int_{D} u_{1}(x) e_{n_{0}}(x) \, dx \vert^{2}}{(t^{2}-\vert x-z \vert^{2})^{3/2}}   + \mathcal{O}\big(\vert \log(a) \vert^{\max(-1,2h-2)}\big),
\end{equation*}
shows that if we measure before and after injecting one nanoparticle, we can reconstruct $\int_{D} \vert u_1(x) \vert^2 dx$.
Using these two last data, i.e. $\int_{D} \vert u_1(x) \vert^2 dx$ and $\int_{D} \vert u_2(x) \vert^2 dx$, we can reconstruct, via (\ref{reconstruction-k-using-contras-permittivity-only}), $\vert \epsilon_0\vert$. 
Hence, using two different resonances, we reconstruct both the permittivity $\epsilon$ and the conductivity $\sigma_{\Omega}$.
\bigskip

Let us show how we can use (\ref{pressure-to-v_0}) to localize the position $z$ of the injected nanoparticles and estimate $\vert u_{0}(z) \vert$. The corresponding results can also be shown using $(\ref{pressure-tilde-expansion})$.  For this, we use the notations
\begin{eqnarray*}
\tilde{p} &:=& (p^{+} + p^{-} - 2p_{0}), \; A := \frac{- 2 \; \, \Im(\tau) \; \; \vert u_0(z) \vert^{2} }{\vert 1 - \omega^{2} \, \mu_{0} \, \lambda_{n_{0}} \, \tau \vert^{2}}  \, \Bigg( \int_{D} e_{n_{0}} \, dx \Bigg)^{2} \; and \; 
Err  :=   \mathcal{O}\big(\vert \log(a) \vert^{\max(-1,2h-2)}\big).
\end{eqnarray*}
Let $t_{1} \neq t_{2}$ then we have 
\begin{equation}\label{p/p}
\frac{\tilde{p}(t_{1},x)}{\tilde{p}(t_{2},x)}  =  \frac{ A \displaystyle\frac{t_{1}}{(t_{1}^{2}-\vert x-z \vert^{2})^{3/2}}+Err}{ A \displaystyle\frac{t_{2}}{(t_{2}^{2}-\vert x-z \vert^{2})^{3/2}}+Err} =  \frac{t_{1}}{t_{2}} \Bigg( \frac{t_{2}^{2}-\vert x-z \vert^{2}}{t_{1}^{2}-\vert x-z \vert^{2}} \Bigg)^{3/2} + \mathcal{O}\Big(\vert \log(a) \vert^{s+\max(-h,h-1)}\Big), 
\end{equation}
where 
\begin{equation}\label{cdtons}
0 \leq s < \min(h;1-h).
\end{equation}
From $(\ref{p/p})$ we derive the formula
\begin{equation}\label{positionza}
\vert x - z \vert = \Bigg[ \frac{t_{1}^{2} \, (t_{2} \tilde{p}(t_{1},x))^{2/3} - t_{2}^{2} \, (t_{1} \tilde{p}(t_{2},x))^{2/3}}{(t_{2} \tilde{p}(t_{1},x))^{2/3} - (t_{1} \tilde{p}(t_{2},x))^{2/3}} \Bigg]^{\frac{1}{2}} + \mathcal{O}\Big(\vert \log(a) \vert^{s+\max(-h,h-1)}\Big). 
\end{equation}
The expression $(\ref{positionza})$ tells that, for $x \in \partial \Omega$, the point $z$ is in the arc given by the intersection of $\Omega$ and the circle $S$ with center $x$ and radius computable as
\begin{equation}\label{radius}
\Bigg[ \frac{t_{1}^{2} \, (t_{2} \tilde{p}(t_{1},x))^{2/3} - t_{2}^{2} \, (t_{1} \tilde{p}(t_{2},x))^{2/3}}{(t_{2} \tilde{p}(t_{1},x))^{2/3} - (t_{1} \tilde{p}(t_{2},x))^{2/3}} \Bigg]^{\frac{1}{2}}. 
\end{equation}
Then in order to localise $z$, we repeat the same experience with another point $x_{\star} \neq x$, and take the intersection of two arcs, see Figure $\ref{fig1}$. 
\begin{figure}[h!]
\centering
\begin{center}
 \includegraphics[scale=0.4]{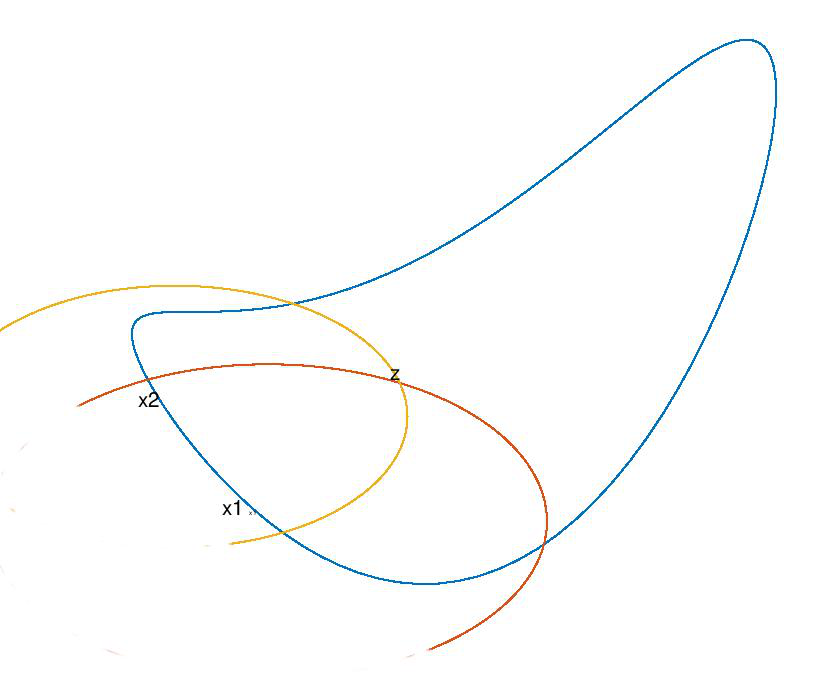}
\end{center}

\caption[Le titre]{Localization of the particles. The blue curve represents $\partial \Omega$ while the red and yellow ones the circles of center $x_{1}:=x$ and $x_{2}:=x^*$ 
and radius $(\ref{radius})$, with $x$ and $x^{\star}$ respectively.}
\label{fig1}
\end{figure}

Assume that $z$ is obtained, then from the equation $(\ref{pressure-to-v_0})$, we get 
\begin{equation*}
\vert u_{0}(z) \vert^{2} = - \, \displaystyle \frac{\vert 1- \omega^{2} \, \mu_{0} \, \lambda_{n_{0}} \, \tau \vert^{2} (t^{2}-\vert x-z \vert^{2})^{3/2} \; \tilde{p}(t,x)}{2 \,t \, \Im(\tau) \,  \Big(\int_{D} e_{n_{0}}\Big)^{2}} + \mathcal{O}\Big(\vert \log(a) \vert^{s+\max(-h,h-1)}\Big). 
\end{equation*}
with 
\begin{equation}\label{cdt2s}
0\leq s < \min \{h ; 1-h \}.
\end{equation}
\bigskip

Let us finish this introduction by comparing our findings with the previous results. To our knowledge, the only work published to analyse the photo-acoustic imaging modality using contrast 
agents is the recent work \cite{Triki-Vauthrin:2017}. The authors propose to use plasmonic resonances instead of dielectric ones. Assuming the acoustic inversion 
to be known and done, as described in section  \ref{photo-acoustic-inversion-known-results}, they perform the electromagnetic inversion. They state the 2D-electromagnetic 
model where the magnetic fields satisfie a divergence form equation. Performing asymptotic expansions, close to these resonances they derive the dominant part of the magnetic field and reconstruct 
the permittivity by an optimization step applied on this dominating term. This result could be compared to Theorem $\ref{Using-only-permittivity-contrast}$, i.e formula 
$(\ref{one-particle-reconst-permittivity-only})$. 

\bigskip

The rest of the paper is organized as follows. In section \ref{Theo-Using-only-permittivity-contrast} and section \ref{Theo-using-both-contrasted-permit-conduct},
we prove Theorem \ref{Using-only-permittivity-contrast} and Theorem \ref{Using-both-permittivity-and-conductivity-contrasts} respectively. 
In section \ref{appendixlemma}, we derive the needed estimates on the electric fields, used in section \ref{Theo-Using-only-permittivity-contrast} and section \ref{Theo-using-both-contrasted-permit-conduct} in 
terms of the contrast of the permittivity, conductivity and for frequencies close to the dielectric resonances. 
Finally, in section \ref{the hypotheses-justification}, we discuss the validity of the conditions in {\bf{Hypotheses}\ref{hyp}}. \\
\bigskip

\textbf{Notations:} Only $\mathbb{L}^{2}$-norms on domains are involved in the text. Therefore, unless indicated, we use $\Vert \cdot \Vert$ without specifying the domain. In addition, we use $<\cdot,\cdot>$ for the corresponding scalar product.      
For a given function $f$ defined on $\overset{M}{\underset{j=1}{\cup}} D_{j}$, we denote by $f_{j} := f_{|D_{j}}$, $j = 1, \cdots,M$. \\ 
The eigenfunctions $\left( e^{(i)}_{n} \right)_{n \in \mathbb{N}}$ of the Newtonian operator stated on $D_{i}$ depend, of course, on $D_{i}$. Nevertheless, unless specified, we use the  notation $\left( e_{n} \right)$ even when dealing with multiple particles located in different positions.       

\bigskip

\section{Proof of Theorem \ref{Using-only-permittivity-contrast}}
\label{Theo-Using-only-permittivity-contrast} 

We split the proof into two subsections. In the first one, we derive the Foldy-Lax algebraic system, see $(\ref{as})$ in proposition $(\ref{propositionas})$, as an approximation of the continuous L.S.E satisfied by the electric field.
In the second subsection, we invert the algebraic system and extract the needed formulas, see $(\ref{Kds})$.

\bigskip

\subsection{Approximation of the L.S.E}
\label{motref}
\bigskip

In the following, we notice by $G_{k}$ the Green kernel for Helmholtz equation in dimension two. This means that $G_{k}$ is a solution of: 
\begin{equation*}
(\Delta + \omega^{2} \, n_{0}^{2}(\cdot) ) G_{k(\cdot)} (\cdot,\cdot) = - \delta_{\cdot}(\cdot), \qquad in \qquad \mathbb{R}^{2}
\end{equation*}
satisfying the S.R.C.

\begin{lemma}\label{GreenKernel}
The Green kernel $G_{k}$ admits the following asymptotic expansion
\begin{eqnarray}\label{Gkexpansion}
\nonumber
G_{k}(\vert x - y \vert) & = &  \frac{-1}{2\pi} \log(\vert x-y \vert) - \frac{1}{2\pi} \log(k(y)) + \frac{i}{4}+\frac{1}{2 \pi}\Bigg[\log(2)-\displaystyle\lim_{p \rightarrow +\infty}\Big( \sum_{m=1}^{p} \frac{1}{m} - \log(p) \Big)\Bigg] \\ 
& + &  \mathcal{O}\big(\vert x-y\vert \; \log(\vert x-y \vert)\big), \quad x \; \text{near} \; y. 
\end{eqnarray} 
\end{lemma}
\begin{proof}
Following the same steps as in \cite{Sini&arabe}, pages 10-12, and taking into account the logarithmic singularity of the fundamental solution of 2D Helmholtz equation we can deduce the expansion $(\ref{Gkexpansion})$.
\end{proof}
\begin{definition}
We define 
\begin{equation*}
a := \frac{1}{2} \underset{1 \leq m \leq M}{diam}(D_{m}),\quad d_{mj} := dist(D_{m},D_{j}),\quad d := \underset{m \neq j \atop 1\leq m ,j \leq M}{ \min} d_{mj},
\end{equation*}
where  $D_{m} = z_{m} + a\, B$ with $B$ is a bounded domain containing the origin. 
\end{definition}
The unique solution of the problem $(\ref{prblm})$, with $D :=  
\overset{M}{\underset{j=1}{\cup}} D_{j}$, satisfies the L.S.E
\begin{equation}\label{O}
u(x) - \omega^{2} \, \mu_{0} \, \int_{D} G_{k}(x,y) \, (\varepsilon_{p}-\varepsilon_{0})(y) \, u(y) \, dy = u_{0}(x) \quad in \; D.
\end{equation}
We set\footnote{We use the notation $v_{m} := u_{|_{D_{m}}}$ instead of $u_{m} := u_{|_{D_{m}}}$ to avoid confusion with $u_{0}, u_{1}$ and $u_{2}$ we defined before concerning the electric fields in the absence or the presence of one or two particles.} $v_{m} := u_{|_{D_{m}}}, m=1,\cdots,M$. Then $(\ref{O})$, for $x \in D_{m}$, rewrites as 
\begin{equation*}
v_{m}(x) - \omega^{2} \mu_{0}  \int_{D_{m}} G_{k}(x,y)  (\varepsilon_{p}-\varepsilon_{0}(y))  v_{m}(y)  dy - \omega^{2} \mu_{0}  \sum_{j \neq m} \int_{D_{j}} G_{k}(x,y) (\varepsilon_{p}-\varepsilon_{0}(y))  v_{j}(y) dy = u_{0}(x).
\end{equation*}
We set: $\tau_{j} = (\varepsilon_{p}-\varepsilon_{0}(z_{j}))$. Assuming $\varepsilon_{0_{|_{\Omega}}}$ to be $\mathcal{W}^{1,\infty}(\Omega)$, the solution $u$ of the scattering problem 
\begin{eqnarray*}
\Big( \Delta + \omega^{2} \, n_{0}(x) \Big) u &=& 0 \quad \text{in} \; \mathbb{R}^{2} \\
u &:=&  u^{i} + u^{s} \; \mbox{and} \; u^{s} \; S.R.C,
\end{eqnarray*}
has a $\mathcal{C}^{1}$-regularity in $\Omega$. Set\footnote{The constant $\Gamma$ will be written as $\Gamma := \frac{i}{4}+\gamma$ where $\gamma$ is the Euler constant.} 
\begin{eqnarray}\label{phigamma}
\Phi_{0}(x,y) := \frac{-1}{2\pi} \log(\vert x-y \vert) \;\; \text{and} \;\;    \Gamma := \frac{i}{4}+\frac{1}{2 \pi}\Bigg[\log(2)-\displaystyle\lim_{p \rightarrow +\infty}\Big( \sum_{m=1}^{p} \frac{1}{m} - \log(p) \Big)\Bigg]. 
\end{eqnarray} 
Expanding  $(\varepsilon_{p}-\varepsilon_{0}(.))$, $u_{0}$ and $G_{k}(\cdot,\cdot)$ near the center $z_{m}$, we obtain 
\begin{eqnarray*}
v_{m}(x) & - & \omega^{2} \, \mu_{0} \, \tau_{m} \int_{D_{m}} \Big(\Phi_{0}(x,y) - \frac{1}{2 \pi} \log(k)(y) + \Gamma \Big) \, v_{m}(y) \, dy \\ 
&-& \omega^{2} \, \mu_{0} \, \tau_{m} \; \mathcal{O}\left( \int_{D_{m}} \vert x-y \vert \, \log\vert x-y \vert \, v_{m}(y) \, dy \right) \\ & + & \omega^{2} \, \mu_{0} \,  \int_{D_{m}} G_{k}(x,y) \, \int_{0}^{1} (y-z_{m}) \centerdot \nabla \varepsilon_{0}(z_{m}+t(y-z_{m})) \, dt \, v_{m}(y) \, dy  \\ & - & \omega^{2} \, \mu_{0} \, \sum_{j \neq m} \, \tau_{j} \, \int_{D_{j}} \Bigg[G_{k}(z_{m};z_{j}) + \int_{0}^{1} \nabla_{x} G_{k}(z_{m}+t(x-z_{m});z_{j})\centerdot(x-z_{m}) \; dt \\ &+& \int_{0}^{1} \nabla_{y}G_{k}(z_{m};z_{j}+t(y-z_{j}))\centerdot(y-z_{j}) \, dt \\ &+& \int_{0}^{1} \nabla_{x} \int_{0}^{1} \nabla_{y} G_{k}(z_{m}+t(x-z_{m});z_{j}+t(y-z_{j}))\centerdot(y-z_{j}) \, dt\centerdot \, (x-z_{m}) \, dt \Bigg]  \, v_{j}(y) \, dy \\ &+& \omega^{2} \, \mu_{0} \, \sum_{j \neq m} \, \int_{D_{j}} G_{k}(x,y) \, \int_{0}^{1} (y-z_{j}) \centerdot \nabla \varepsilon_{0}(z_{j}+t(y-z_{j})) \,dt \, v_{j}(y) \, dy \\ &=& u_{0}(z_{m}) + \int_{0}^{1} \nabla u_{0}(z_{m}+t(x-z_{m})) \centerdot (x-z_{m}) \, dt.
\end{eqnarray*}
We assumed that all nano-particles have the same electromagnetic properties, then $\tau_{j}$ is the same for every $j$ and let us denote it by $\tau$. Define
\begin{equation}\label{A}
w =  \omega^{2} \, \mu_{0} \, \tau \Big[I - \omega^{2} \, \mu_{0} \, \tau \, A_{0} \Big]^{-1}(1) =  \Big[\frac{1}{\omega^{2} \, \mu_{0} \, \tau} I - A_{0} \Big]^{-1}(1),
\end{equation}
and set the following notations
\begin{equation}\label{R}
\textbf{C}_{m} = \int_{D_{m}}  w  dx \quad \& \quad  \textbf{C}_{m}^{\star} = \textbf{C}_{m} \Bigg[1 - \Big(- \frac{1}{2 \pi} \log(k)(z_{m}) + \Gamma \Big)  \textbf{C}_{m}\Bigg]^{-1} \quad \& \quad   Q_{m} = \omega^{2} \, \mu_{0}  \tau  (\textbf{C}_{m}^{\star})^{-1} \, \int_{D_{m}} v_{m} dx.   
\end{equation}

Using the definition of $w$, and integrate $y$ over $D_{m}$,  the self adjointness of the operator $(\lambda I - A_{0})$ and we multiplying both sides of this equation by $\omega^{2} \, \mu_{0} \, \tau \, C_{m}^{-1}$, we obtain 
\begin{eqnarray*}
Q_{m} &-&  \sum_{j \neq m} G_{k}(z_{m};z_{j}) \; \textbf{C}_{j}^{\star} \; Q_{j} = u_{0}(z_{m}) + \omega^{2} \, \mu_{0} \; \tau \, \textbf{C}_{m}^{-1} \; \Bigg[ \\ 
&+&  \int_{D_{m}}\,w \, \int_{D_{m}} \vert x-y \vert \, \log\vert x-y \vert \, v_{m}(y) \, dy \, dx \\ 
& - & \tau^{-1} \, \int_{D_{m}} \, w \int_{D_{m}} G_{k}(x,y) \, \int_{0}^{1} (y-z_{m}) \centerdot \nabla \varepsilon_{0}(z_{m}+t(y-z_{m})) \,dt \; v_{m}(y) \, dy \, dx \\ 
& + &  \sum_{j \neq m} \, \int_{D_{m}} \, w \, \int_{0}^{1} \nabla_{x}G_{k}(z_{m}+t(x-z_{m});z_{j})\centerdot(x-z_{m})\,dt \,dx \, \int_{D_{j}} v_{j} \, dy \\ 
& + & \textbf{C}_{m}  \sum_{j \neq m} \, \int_{D_{j}} \, \int_{0}^{1} \, \nabla_{y} G_{k}(z_{m};z_{j}+t(y-z_{j}))\centerdot(y-z_{j})dt  \, v_{j}(y) \, dy\\ 
& + &  \sum_{j \neq m} \int_{D_{m}} w  \int_{D_{j}} \int_{0}^{1} \nabla \int_{0}^{1} \nabla G_{k}(z_{m}+t(x-z_{m});z_{j}+t(y-z_{j}))\centerdot(y-z_{j}) dt \centerdot (x-z_{m}) dt  v_{j}(y)  dy dx\\ 
&+&   \tau^{-1} \, \, \sum_{j \neq m} \, \int_{D_{m}} \, w \, \int_{D_{j}} G_{k}(x,y) \, \int_{0}^{1} (y-z_{j}) \centerdot \nabla \varepsilon_{0}(z_{j}+t(y-z_{j})) \,dt v_{j}(y) \, dy \, dx \\ 
&+& \big( \omega^{2} \, \mu_{0} \, \tau \big)^{-1} \, \int_{D_{m}} \, w \, \int_{0}^{1} (x-z_{m}) \centerdot \nabla u_{0}(z_{m}+t(x-z_{m})) dt \, dx + \mathcal{O}\Bigg(\textbf{C}_{m} \, a \, \int_{D_{m}} v_{m} dx  \Bigg)\Bigg].
\end{eqnarray*}
For the right side, we keep $u_{0}(z_{m})$ as a dominant term and estimate the other terms as an error. To achieve this goal, we need the following proposition.
\begin{proposition}\label{abc}
We have:
\begin{equation}\label{prioriest}
\Vert u \Vert_{\mathbb{L}^{2}(D)} \leq \vert \log(a) \vert^{h} \, \Vert u_{0} \Vert_{\mathbb{L}^{2}(D)},
\end{equation} 
and
\begin{equation*}
\textbf{C}_{m} = \mathcal{O}( \vert \log(a) \vert^{h-1}).
\end{equation*}
\end{proposition}
\begin{proof}
See Section \ref{appendixlemma}. 
\end{proof}
As the incident wave is smooth and independent on $a$, thanks to $(\ref{prioriest})$, we get
\begin{equation*}
\Vert w \Vert_{\mathbb{L}^{2}(D)} \leq a^{-1} \, \vert \log(a) \vert^{h-1}.
\end{equation*} 
We recall that  
\begin{equation*}
\tau \sim 1 / a^{2} \, \vert \log(a) \vert.
\end{equation*}
The error part contains eight terms. Next we define and estimate every term, then we sum them up. More precisely, we have
\begin{itemize}
\item \underline{Estimation of} 
$S_{1}  :=   \tau \; \textbf{C}_{m}^{-1} \int_{D_{m}}  \,w \,  \int_{D_{m}} \vert x-y \vert \log(\vert x-y \vert) v_{m}(y) dy  \; dx$ 
\begin{eqnarray*}
\vert S_{1} \vert & \precsim &  a^{-2} \, \vert \log(a) \vert^{-1} \; \vert \log(a) \vert^{1-h} \; \Vert w \Vert \, \Bigg[\int_{D_{m}} \Bigg\vert \int_{D_{m}} \vert x-y \vert \log(\vert x-y \vert) v_{m}(y) dy  \Bigg\vert^{2} \; dx \Bigg]^{\frac{1}{2}} \\
& \precsim &  a^{-2} \, \vert \log(a) \vert^{-h} \; a^{-1} \; \vert \log(a) \vert^{h-1} \; \Bigg[\int_{D_{m}} \Bigg( \int_{D_{m}}  \vert v_{m}\vert(y) dy  \Bigg)^{2} \; dx \Bigg]^{\frac{1}{2}} \, a \, \vert \log(a) \vert\\
& = & \mathcal{O}\Big( a^{-2} \, \vert \log(a) \vert^{-1} \, a \, \Vert v_{m} \Vert \, a \, \vert \log(a) \vert \Big),   
\end{eqnarray*}
and then
\begin{equation*}
S_{1} = \mathcal{O}\Big(a \, \vert \log(a) \vert^{h} \, M^{\frac{1}{2}} \Big).
\end{equation*}
\item \underline{Estimation of}
$S_{2} := \textbf{C}_{m}^{-1} \, \int_{D_{m}} w(x) \int_{D_{m}} G_{k}(x,y) \int_{0}^{1} (y-z_{m})\centerdot \nabla \varepsilon_{0}(z_{m}+t(y-z_{m})) dt \, v_{m}(y) \, dy \, dx $
\begin{equation*}
\vert S_{2} \vert  \lesssim  a^{-1} \, \Bigg[ \int_{D_{m}} \Bigg(\int_{D_{m}} \vert G_{k}\vert(x;y) \; \Big\vert \int_{0}^{1} (y-z_{m})\centerdot\nabla\varepsilon_{0}(z_{m}+t(y-z_{m})) dt \Big\vert \; \vert v_{m}\vert(y) dy \Bigg)^{2} dx \Bigg]^{\frac{1}{2}}. 
\end{equation*}
The smoothness of $\varepsilon_{0}$ implies 
$\left\vert \int_{0}^{1} (y-z_{m})\centerdot\nabla\varepsilon_{0}(z_{m}+t(y-z_{m})) dt \right\vert \lesssim \mathcal{O}(a),$ hence
\begin{eqnarray*}
\vert S_{2} \vert & \lesssim & \Vert v_{m} \Vert \; \Bigg[\int_{D_{m}} \int_{D_{m}} \vert G_{k} \vert^{2} (x;y) \; dy \; dx \Bigg]^{\frac{1}{2}} \lesssim  \Vert u \Vert \; a^{2} \; \vert \log(a) \vert, 
\end{eqnarray*}
and then
\begin{equation*}
S_{2} = \mathcal{O}\Big(a^{3} \; \vert \log(a) \vert^{1+h} \; M^{\frac{1}{2}}\Big).
\end{equation*}
\item \underline{Estimation of}
$ S_{3}   :=  \ \tau \, \textbf{C}_{m}^{-1} \, \underset{j \neq m}{\sum} \, \int_{D_{m}} \, w \, \int_{0}^{1} \nabla_{x}G_{k}(z_{m}+t(x-z_{m});z_{j})\centerdot(x-z_{m})\,dt \,dx \, \int_{D_{j}} v_{j} \, dy  $
\begin{equation*}
\vert S_{3} \vert  \lesssim  \frac{1}{a \, \vert \log(a) \vert^{h}} \sum_{j \neq m} \Vert w \Vert \; \Vert v_{j} \Vert \; \Bigg[\int_{D_{m}} \Bigg\vert \int_{0}^{1} \underset{x}{\nabla} G_{k}(z_{m}+t(x-z_{m});z_{j})\centerdot(x-z_{m}) \, dt\, \Bigg\vert^{2}  dx \Bigg]^{\frac{1}{2}}. 
\end{equation*}
Without difficulties, we can check that
\begin{equation*}
\Bigg[\int_{D_{m}}  \Bigg\vert \int_{0}^{1} \underset{x}{\nabla} G_{k}(z_{m}+t(x-z_{m});z_{j})\centerdot(x-z_{m}) \, dt \Bigg\vert^{2} \,  dx \Bigg]^{\frac{1}{2}} \lesssim \frac{a^{2}}{d_{mj}},
\end{equation*}
then we plug this on the previous equation and use Cauchy-Schwartz inequality, to get
\begin{equation*}
\vert S_{3} \vert  \lesssim  \vert \log(a) \vert^{-1} \; \Vert v \Vert \; \Bigg( \sum_{j \neq m} \frac{1}{d_{mj}^{2}} \Bigg)^{\frac{1}{2}}  \lesssim  \vert \log(a) \vert^{h-1} \; a \; M \; d^{-1}.
\end{equation*}
Set 
\begin{equation*}
S_{4} := \tau \, \, \sum_{j \neq m} \, \int_{D_{j}} \, \int_{0}^{1} \, \underset{y}{\nabla}  G_{k}(z_{m};z_{j}+t(y-z_{j}))\centerdot(y-z_{j})dt  \, v_{j}(y) \, dy, 
\end{equation*}
and remark that $S_{4}$ has a similar expression as $S_{3}$, then we obtain:
\begin{equation*}
S_{3} = \mathcal{O}(\vert \log(a) \vert^{h-1} \; a \; M \; d^{-1}) \quad \text{and} \quad S_{4} = \mathcal{O}(\vert \log(a) \vert^{h-1} \; a \; M \; d^{-1}).
\end{equation*} 
\item \underline{Estimation of} 
\begin{equation*}
S_{5} :=
\frac{\tau}{\textbf{C}_{m}} \, \sum_{j \neq m} \int_{D_{m}} w  \int_{D_{j}} \int_{0}^{1} \nabla \int_{0}^{1} \nabla G_{k}(z_{m}+t(x-z_{m});z_{j}+l(y-z_{j}))\centerdot(y-z_{j}) dl \centerdot (x-z_{m}) dt  v_{j}(y) dy  dx
\end{equation*} 
\begin{eqnarray*}
\vert S_{5} \vert & \lesssim & \frac{\Vert w \Vert}{a^{2} \, \vert \log(a) \vert^{h}} \sum_{j \neq m}  \, \Bigg\Vert \int_{D_{j}} \int_{0}^{1} \nabla \int_{0}^{1} \nabla G_{k}(z_{m}+t(\centerdot-z_{m});z_{j}+l(y-z_{j}))\centerdot(y-z_{j}) dl \centerdot (\centerdot-z_{m}) dt \, v_{j}(y) \, dy \Bigg\Vert \\
& \lesssim & \frac{a^{-3}}{\vert \log(a) \vert} \sum_{j \neq m} \Vert v_{j} \Vert \Bigg[\int_{D_{m}} \Bigg\vert \int_{D_{j}} \int_{0}^{1} \nabla \int_{0}^{1} \nabla G_{k}(z_{m}+t(x-z_{m});z_{j}+l(y-z_{j}))\centerdot(y-z_{j}) dl \centerdot (x-z_{m}) dt dy \Bigg\vert^{2} dx\Bigg]^{\frac{1}{2}}. 
\end{eqnarray*}
we have
\begin{equation*}
\int_{D_{m}} \Bigg\vert \int_{D_{j}} \int_{0}^{1} \nabla \int_{0}^{1} \nabla G_{k}(z_{m}+t(x-z_{m});z_{j}+l(y-z_{j}))\centerdot(y-z_{j}) dl \centerdot (x-z_{m}) dt dy \Bigg\vert^{2} dx \lesssim \mathcal{O}\Big( \frac{a^{8}}{d_{mj}^{4}} \Big),
\end{equation*}
hence
\begin{eqnarray*}
\vert S_{5} \vert & \lesssim & a \, \vert \log(a) \vert^{-1} \, \Vert u \Vert \, \Bigg(\sum_{j \neq m} \frac{1}{d_{mj}^{4}} \Bigg)^{\frac{1}{2}} \lesssim  a^{2} \; \vert \log(a) \vert^{h-1} \; M \; d^{-2}, 
\end{eqnarray*}
then
\begin{equation*}
S_{5} = \mathcal{O}(a^{2} \; \vert \log(a) \vert^{h-1} \; M \, d^{-2}).
\end{equation*}
\item \underline{Estimation of} $S_{6} := \textbf{C}_{m}^{-1} \, \underset{j \neq m}{\sum} \int_{D_{m}} w \int_{D_{j}} G_{k}(x,y) \int_{0}^{1} (y-z_{j})\centerdot \nabla \varepsilon_{0}(z_{j}+t(y-z_{j})) \, dt \, v_{j}(y) \, dy \, dx$ 
\begin{eqnarray*}
\vert S_{6} \vert & \lesssim & \sum_{j \neq m} \Vert v_{j} \Vert \; \Bigg[ \int_{D_{m}} \; \int_{D_{j}} \vert G_{k} \vert^{2}(x;y) \,  dy \; dx \, \Bigg]^{\frac{1}{2}}  \lesssim  a^{2} \; \Vert v \Vert \; M^{\frac{1}{2}} \leq a^{3} \; \vert \log(a) \vert^{h} \; M.
\end{eqnarray*}
Then 
\begin{equation*}
S_{6} = \mathcal{O}\Big(a^{3} \; \vert \log(a) \vert^{h} \; M\Big).
\end{equation*}
\item \underline{Estimation of} $S_{7} := \textbf{C}_{m}^{-1} \int_{D_{m}} w \int_{0}^{1} (x-z_{m})\centerdot \nabla u_{0}(z_{m}+t(x-z_{m})) \, dt \, dx$
\begin{eqnarray*}
\vert S_{7} \vert & \lesssim & \vert \log(a) \vert^{1-h} \; \Vert w \Vert \; \Bigg[ \int_{D_{m}} \Bigg\vert \int_{0}^{1} (x-z_{m}) \centerdot \nabla u_{0}(z_{m}+t(x-z_{m})) dt \Bigg\vert^{2}  dx \Bigg]^{\frac{1}{2}}.
\end{eqnarray*}
As $u_{0}$ is smooth, we have
\begin{equation*}
\Bigg[ \int_{D_{m}} \Bigg\vert \int_{0}^{1} (x-z_{m}) \centerdot \nabla u_{0}(z_{m}+t(x-z_{m})) dt \Bigg\vert^{2}  dx \Bigg]^{\frac{1}{2}} = \mathcal{O}(a). 
\end{equation*}
Hence
\begin{equation*}
S_{7} = \mathcal{O}(a).
\end{equation*}
\item \underline{Estimation of} 
$ S_{8} := a \, \tau \, \int_{D_{m}} v_{m}$
\begin{equation*}
\vert S_{8} \vert \leq \tau \, a \, \Vert 1 \vert_{\mathbb{L}^{2}(D_{m})} \,  \Vert v_{m} \Vert_{\mathbb{L}^{2}(D_{m})} = \mathcal{O}(a \, \vert \log(a) \vert^{h-1}). 
\end{equation*}
\end{itemize}
Finally, the error part is 
\begin{equation*}
Error = S_{1}+\cdots+S_{8} = \mathcal{O}(a \; d^{-1} \; \vert \log(a) \vert^{h-1} \; M ). 
\end{equation*}
\begin{proposition}\label{propositionas} 
The vector $\big( Q_{m} \big)_{m=1}^{M}$ satisfie the following algebraic system
\begin{equation}\label{as}
Q_{m} - \sum_{j \neq m} G_{k}(z_{m};z_{j}) \; \textbf{C}_{j}^{\star} \; Q_{j} = u_{0}(z_{m}) + \mathcal{O}(a \; d^{-1} \; \vert \log(a) \vert^{h-1} \; M ).
\end{equation} 
\end{proposition}
The algebraic system $\ref{as}$ can be written, in a matrix form,  as
\begin{equation}\label{per}
(I-B) \, Q = V + Err
\end{equation}
with $B=\Big(B_{mj}\Big)_{m,j=1}^{M}$ such that
 $B_{mj} := G_{k}(z_{m};z_{j}) \, \textbf{C}_{j}^{\star}$ and $V := (u_{0}(z_{1}), \cdots, u_{0}(z_{M}))^{\top}$. \\
In the next proposition, we give conditions under which the  linear system $(\ref{per})$
is invertible.
\begin{lemma}\label{d>exp}
The algebraic system $(\ref{per})$ is invertible if
\begin{equation}\label{Condinv}
d > exp\Big(- \vert \log(a) \vert^{1-h} \Big),
\end{equation}
where $d$ is the minimal distance between the particles. 
\end{lemma}
\begin{proof} of $\textbf{Lemma} (\ref{d>exp})$.
Let us evaluate the norm of $B$. For this we have: 
\begin{eqnarray*}
\Vert B \Vert & = & \max_{m} \, \sum_{j \neq m} \vert B_{mj} \vert \stackrel{\text{def}} = \max_{m} \sum_{j \neq m} \Bigg\vert G_{k}(z_{m};z_{j})  \Bigg[\textbf{C}^{-1}_{j} - \Big( \frac{-1}{2\pi} \log(k)(z_{j}) + \Gamma \Big)\Bigg]^{-1} \Bigg\vert \\
 & \leq &  \vert \log(a) \vert^{h-1} \sum_{j \neq m} \log\Bigg(\frac{1}{d_{mj}}\Bigg).  
\end{eqnarray*}
We need the following lemma
\begin{lemma}\label{P}
We have
\begin{equation}\label{lmjdmj}
\sum_{j \neq m} \log(1/d_{mj}) = \log(1/d). 
\end{equation}
\end{lemma}
\begin{proof}of \textbf{Lemma \ref{P}}\\
We set $  \log(1/d_{mj}) = 1/l_{mj}$ and $l = \displaystyle\min_{j \neq m} l_{mj}$. Then 
\begin{equation}\label{lmj}
\sum_{j \neq m} \log\Bigg(\frac{1}{d_{mj}}\Bigg) = \sum_{j \neq m} \frac{1}{l_{mj}} \stackrel{(\star)}= \frac{1}{l}.
\end{equation}
At first we assume that $(\star)$ is checked. Then we have
\begin{equation}\label{dmj}
l = \min_{j \neq m} \frac{1}{\log(1/d_{mj})} = \frac{1}{\log(\displaystyle\max_{j \neq m} (1/d_{mj}))} = \frac{1}{\log(1/(\displaystyle\min_{j \neq m} d_{mj}))} = \frac{1}{\log(1/d)}.
\end{equation}
Then $(\ref{lmj})$ combined with $(\ref{dmj})$ give a justification of $(\ref{lmjdmj})$.\\
Now, in order to prove $(\star)$ we modify to the two dimensional case the proof, done for three dimensional case, given in (\cite{ASCD}, page 13). We get
\begin{equation*}
\sum_{i=1 \atop i \neq j}^{M} \frac{1}{l^{k}_{ij}} = \left\{
\begin{array}{lll}
    \mathcal{O}(l^{-k}) + \mathcal{O}(l^{-2\alpha}) & if & k < 2\\
    \mathcal{O}(l^{-2}) + \mathcal{O}(l^{-2\alpha} \vert \log(l) \vert) & if  & k=2 \\
   \mathcal{O}(l^{-k}) + \mathcal{O}(l^{-\alpha k}) & if & k > 2.
    \end{array}
\right.
\end{equation*}
\end{proof}
Based on lemma $\ref{P}$ the condition $\Vert B \Vert < 1$, is fulfilled if
\begin{equation}\label{cdtsurd}
\log(1/d) < \vert \log(a) \vert^{1-h} \quad \text{or} \quad d > \exp\Big(-\vert \log(a) \vert^{1-h}\Big).
\end{equation} 
\end{proof}


\subsection{Inversion of the derived  Foldy-Lax algebraic system (\ref{as})}

\bigskip

Here, we deal with the case of two particles, i.e $M = 2$. In the equation $(\ref{as})$ we use the condition $d \backsim a^{\vert \log(a) \vert^{-h}}$, then we get  
\[
\left\{
\begin{array}{r c l}
Q_{1} - G_{k}(z_{1};z_{2}) \, \textbf{C}_{2}^{\star} Q_{2} &=& u_{0}(z_{1}) + \mathcal{O}(a^{1-\vert \log(a) \vert^{-h}} \, \, \vert \log(a) \vert^{h-1}),\\
&&\\
Q_{2} - G_{k}(z_{2};z_{1}) \, \textbf{C}_{1}^{\star} Q_{1} &=& u_{0}(z_{2}) + \mathcal{O}(a^{1-\vert \log(a) \vert^{-h}} \, \vert \log(a) \vert^{h-1}). 
\end{array}
\right.
\]
We check that the condition $d \backsim a^{\vert \log(a) \vert^{-h}}$ is sufficient for the invertibility of the last system. For this, we have from $(\ref{cdtsurd})$
\begin{equation*}
d > \exp\Big(-\vert \log(a) \vert^{1-h}\Big) = \Big(e^{-\vert \log(a) \vert} \Big)^{- \vert \log(a) \vert^{-h}} = a^{\vert \log(a) \vert^{-h}}.
\end{equation*}
Now, we assume that\footnote{Remember that we assumed that all nano-particles have the same electromagnetic properties.} $\textbf{C}_{1} = \textbf{C}_{2} = \textbf{C}$ and use the expansion of $G_{k}(z_{m};z_{j})$, see $(\ref{Gkexpansion})$, to obtain
\[
\left\{
\begin{array}{r c l}
Q_{1} - \Big[\Phi_{0}(z_{1};z_{2}) - \frac{1}{2\pi}\log(k)(z_{1})+\Gamma \Big] \, \textbf{C}_{2}^{\star} Q_{2} &=& u_{0}(z_{1}) + \mathcal{O}\bigg(\displaystyle\frac{a^{1-\vert \log(a) \vert^{-h}}}{ \, \vert \log(a) \vert^{1-h}}\bigg) + \mathcal{O}\bigg(d \, \vert \log(d) \vert \, \textbf{C}_{2}^{\star} Q_{2} \bigg), \\
&&\\
Q_{2} - \Big[\Phi_{0}(z_{2};z_{1}) - \frac{1}{2\pi}\log(k)(z_{2})+\Gamma \Big] \, \textbf{C}_{1}^{\star} Q_{1} &=& u_{0}(z_{2}) + \mathcal{O}\bigg(\displaystyle\frac{a^{1-\vert \log(a) \vert^{-h}}}{ \, \vert \log(a) \vert^{1-h}}\bigg) + \mathcal{O}\bigg(d \, \vert \log(d) \vert \, \textbf{C}_{1}^{\star} Q_{1} \bigg).
\end{array}
\right.
\]
We can estimate
\begin{equation}\label{Q}
d \, \vert \log(d) \vert \, \textbf{C}_{i}^{\star} Q_{i}  = \mathcal{O}(a^{\vert \log(a) \vert^{-h}}), \quad for \quad i=1,2,
\end{equation}
because, by the definition of $Q_{i}$, see $(\ref{R})$, we have 
\begin{equation*}
d \, \vert \log(d) \vert \, \textbf{C}_{i}^{\star} Q_{i} =  d \, \vert \log(d) \vert \, \textbf{C}_{i}^{\star} \, \omega^{2} \, \mu_{0} \, \tau \,  \big(\textbf{C}_{i}^{\star}\big)^{-1} \int_{D_{i}} v \, dx  \lesssim  d \, \vert \log(d) \vert \, \vert \tau \vert \, \Vert 1 \Vert \, \Vert u \Vert.
\end{equation*}
The value of $d$ and $\vert \tau \vert$ are known, and we have an a priori estimate about $\Vert u \Vert$ given by $(\ref{prioriest})$, then 
\begin{equation*}
\mathcal{O}(d \, \log(d)) \; \textbf{C}_{i}^{\star} \; Q_{i}  \lesssim  a^{\vert \log(a) \vert^{-h}} \, \vert \log(a) \vert^{1-h} \, a^{-2} \, \vert \log(a) \vert^{-1} \, a \,  \vert \log(a) \vert^{h} \, a = a^{\vert \log(a) \vert^{-h}}. 
\end{equation*}
This proves $(\ref{Q})$. \\
With these estimations the last system can be written as
\begin{equation}\label{ccequastar}
\left\{%
\begin{array}{lll}
    Q_{1} - \Big[\Phi_{0}(z_{1};z_{2}) - \frac{1}{2\pi}\log(k)(z_{1})+\Gamma \Big] \, \textbf{C}_{2}^{\star} Q_{2} &=& u_{0}(z_{1}) + \mathcal{O}(d), \\
&&\\
Q_{2} - \Big[\Phi_{0}(z_{2};z_{1}) - \frac{1}{2\pi}\log(k)(z_{2})+\Gamma \Big] \, \textbf{C}_{1}^{\star} Q_{1} &=& u_{0}(z_{2}) + \mathcal{O}(d). 
    \end{array}%
\right.
\end{equation}
We need the following lemma to simplify the last system. 
\begin{lemma}
Since $k$ is $\mathcal{C}^{1}$-smooth and $z_{1}$ is close to $z_{2}$ at a distance $d$, we obtain
\begin{equation*}
\log(k)(z_{2}) = \log(k)(z_{1}) + \mathcal{O}\big(d \big) 
\quad \text{and} \quad
\textbf{C}_{2}^{\star} = \textbf{C}_{1}^{\star} + \mathcal{O}\big(d \; \textbf{C}^{\,2}\big).
\end{equation*}
\end{lemma}
\begin{proof}
Use Taylor expansion of the function $k$ to get the first equality. Now the first one is proved, we use the definition of $\textbf{C}_{1,2}^{\star}$ and the fact that $\textbf{C}_{1} = \textbf{C}_{2}$  to obtain the second equality. 
\end{proof}
We use the last lemma to write the system $(\ref{ccequastar})$ as
\begin{equation*}
\left\{%
\begin{array}{lll}
    Q_{1} - \Big[\Phi_{0}(z_{1};z_{2}) - \frac{1}{2\pi}\log(k)(z_{1})+\Gamma \Big] \, \textbf{C}_{1}^{\star} Q_{2} &=& u_{0}(z_{1}) + \mathcal{O}(d), \\
&&\\
Q_{2} - \Big[\Phi_{0}(z_{2};z_{1}) - \frac{1}{2\pi}\log(k)(z_{1})+\Gamma \Big] \, \textbf{C}_{1}^{\star} Q_{1} &=& u_{0}(z_{2}) + \mathcal{O}(d). 
    \end{array}%
\right.
\end{equation*}
\begin{remark}
To simplify notations, we write $\Phi_{0}$ (respectively $\frac{-1}{2\pi}\log(k)$, $\textbf{C}^{\star}$) instead of $\Phi_{0}(z_{1};z_{2})$ (respectively $\frac{-1}{2 \pi} \, \log(k(z_{1}))$, $\textbf{C}_{1}^{\star}$).
\end{remark}
After resolution of this algebraic system, we obtain  
\begin{equation*}\label{usedpressure}
\left\{%
\begin{array}{lll}
    Q_{1} &=& \displaystyle\frac{u_{0}(z_{1})}{1-\big[\Phi_{0} +  \frac{-1}{2\pi}\log(k)+\Gamma \big] \, \textbf{C}^{\star}} + \mathcal{O}(d), \\
&&\\
Q_{2} &=& \displaystyle\frac{u_{0}(z_{2})}{1-\big[\Phi_{0} + \frac{-1}{2\pi}\log(k)+\Gamma \big] \, \textbf{C}^{\star}} + \mathcal{O}(d). 
    \end{array}%
\right.
\end{equation*}
We use the definition of $Q_{1,2}$, see $(\ref{R})$, to get
\begin{equation*}\label{equa53}
\int_{D_{1}} v \; dx = \displaystyle\frac{u_{0}(z_{1})}{\omega^{2} \, \mu_{0} \, \tau \big[(\textbf{C}^{\star})^{-1} - (\Phi_{0} + (\frac{-1}{2\pi}\log(k)+\Gamma))\big] \, } + \mathcal{O}(d \, a^{2} \, \vert \log(a) \vert^{h}).
\end{equation*}
Then
\begin{eqnarray*}
\int_{D_{1}} v \; dx &=& \displaystyle\frac{u_{0}(z_{2})}{\omega^{2} \, \mu_{0} \, \tau \big[(\textbf{C}^{\star})^{-1} - (\Phi_{0} + (\frac{-1}{2\pi}\log(k)+\Gamma))\big] \, } \\ &+& 
\displaystyle\frac{u_{0}(z_{1})-u_{0}(z_{2})}{\omega^{2} \, \mu_{0} \, \tau \big[(\textbf{C}^{\star})^{-1} - (\Phi_{0} + (\frac{-1}{2\pi}\log(k)+\Gamma))\big] \, } +
\mathcal{O}(d \, a^{2} \, \vert \log(a) \vert^{h}), 
\end{eqnarray*}
we estimate the term 
\begin{equation*}
\frac{u_{0}(z_{1})-u_{0}(z_{2})}{\omega^{2} \, \mu_{0} \, \tau \big[(\textbf{C}^{\star})^{-1} - (\Phi_{0} + (\frac{-1}{2\pi}\log(k)+\Gamma))\big] \, }
\end{equation*}
as $\mathcal{O}(d \, a^{2} \, \vert \log(a) \vert^{h})$, and use this to obtain
\begin{eqnarray*}
\int_{D_{1}} v \; dx &=& \displaystyle\frac{u_{0}(z_{2})}{\omega^{2} \, \mu_{0} \, \tau \big[(\textbf{C}^{\star})^{-1} - (\Phi_{0} + (\frac{-1}{2\pi}\log(k)+\Gamma))\big] \, } + 
\mathcal{O}(d \, a^{2} \, \vert \log(a) \vert^{h}) \\
&=& \int_{D_{2}} v \, dx + \mathcal{O}(d \, a^{2} \, \vert \log(a) \vert^{h}),
\end{eqnarray*} 
and finally
\begin{equation}\label{intv1intv2}
\int_{D_{1}} v \, dx = \int_{D_{2}} v \, dx + \mathcal{O}(d \, a^{2} \, \vert \log(a) \vert^{h}).
\end{equation}
By adding the two equations of system $(\ref{ccequastar})$, we get 
\begin{equation}\label{equastar}
\textbf{C}^{-1} - \Big( \Phi_{0}+2(-\log(k) / 2 \pi +\Gamma)\Big)  = \frac{u_{0}(z_{1}) +  u_{0}(z_{2})}{\omega^{2} \; \mu_{0} \; \tau \; \Bigg[ \displaystyle\int_{D_{1}} v \, dy +  \int_{D_{2}} v \, dy \Bigg]} +  \frac{\mathcal{O}\big(d \, \tau^{-1}\big)}{  \displaystyle \int_{D_{1}} v \, dy + \int_{D_{2}} v \, dy}.
\end{equation}
We use equation $(\ref{intv1intv2})$ to rewrite the denominator as
\begin{eqnarray*}
\omega^{2} \; \mu_{0} \; \tau \; \Bigg[ \int_{D_{1}} v \, dy + \int_{D_{2}} v \, dy \Bigg] & = & \omega^{2} \; \mu_{0} \; \tau \; \Bigg[2 \int_{D_{2}} v \, dy + \mathcal{O}(d \, a^{2} \, \vert \log(a) \vert^{h}) \Bigg] \\
& = & 2\, \omega^{2} \; \mu_{0} \; \tau \; \int_{D_{2}} v \, dy \, \Bigg[1 + \frac{\mathcal{O}(d \, a^{2} \, \vert \log(a) \vert^{h})}{\int_{D_{2}} v \, dy} \Bigg] \\
& = & 2\, \omega^{2} \; \mu_{0} \; \tau \; \int_{D_{2}} v \, dy \, \Big[1 + \mathcal{O}(d) \Big],
\end{eqnarray*}
then equation $(\ref{equastar})$ takes the form
\begin{equation*}
\textbf{C}^{-1} - \Big( \Phi_{0}+2(- \log(k) / 2 \pi +\Gamma)\Big) = \frac{u_{0}(z_{1}) +  u_{0}(z_{2})}{2\, \omega^{2} \; \mu_{0} \; \tau \; \displaystyle\int_{D_{2}} v \, dy } \Big[1 + \mathcal{O}(d) \Big] + \frac{\mathcal{O}\big(d \, \tau^{-1} \big)}{  \; \displaystyle \int_{D_{2}} v \, dy \, \Big[1 + \mathcal{O}(d) \Big]}, 
\end{equation*}
We manage the errors 
\begin{equation*}   
 \textbf{C}^{-1} - \Big( \Phi_{0}+2(- \log(k) / 2 \pi +\Gamma)\Big)  = \frac{u_{0}(z_{1}) +  u_{0}(z_{2})}{2\, \omega^{2} \; \mu_{0} \; \tau \; \displaystyle\int_{D_{2}} v \, dy } + \mathcal{O}(d \, \vert \log(a) \vert^{1-h}) 
\end{equation*}
\begin{eqnarray*} 
\phantom{Invisible text} \qquad \qquad &=& \frac{2\,u_{0}(z_{2})}{2\, \omega^{2} \; \mu_{0} \; \tau \; \displaystyle\int_{D_{2}} v \, dy } + \frac{\int_{0}^{1} (z_{1} -z_{2}) \centerdot \nabla u_{0}(z_{2}+t(z_{1}-z_{2})) dt}{2\, \omega^{2} \; \mu_{0} \; \tau \; \displaystyle\int_{D_{2}} v \, dy } +\mathcal{O}(d \, \vert \log(a) \vert^{1-h}) \\
 &=& \frac{u_{0}(z_{2})}{\omega^{2} \; \mu_{0} \; \tau \; \displaystyle\int_{D_{2}} v \, dy } +\mathcal{O}(d \, \vert \log(a) \vert^{1-h}), 
\end{eqnarray*}
and take the modulus, we derive the identity:
\begin{equation}\label{equa55}
\Bigg\vert \textbf{C}^{-1} - \Big( \Phi_{0}+2(- \log(k) / 2\pi +\Gamma)\Big) \Bigg\vert^{2} = \frac{\vert u_{0}(z_{2}) \vert^{2}}{\vert  \omega^{2} \; \mu_{0} \; \tau \vert^{2} \; \Big\vert \displaystyle\int_{D_{2}} v \, dy \Big\vert^{2}} +\mathcal{O}(d \, \vert \log(a) \vert^{2(1-h)}).
\end{equation}
Unfortunately, from the acoustic inversion, we get only data of the form $\int_{D_{1,2}} \vert v \vert^{2} dx$ and in the last equation we deal with $\vert \int_{D_{1,2}} v \, dx \vert^{2}$. The next lemma makes a link between  these two quantities. 
\begin{lemma}
We have 
\begin{equation}\label{equa56}
\Big\vert \int_{D_{i}} v \; dy \Big\vert^{2} = a^{2} \, \Big(\int_{B} \overline{e}_{n_{0}} \; dy \Big)^{2} \; \int_{D_{i}} \vert v \vert^{2} \; dy + \mathcal{O}(a^{4} \, \vert \log(a) \vert^{h}), \quad i=1,2.
\end{equation}
\end{lemma}
\begin{proof}
We split the proof into two steps.\\
Step 1: Estimation of $\vert \int_{D_{1}} v \, dy \vert^{2}$. \\
We use the same techniques as in the proof of the a priori estimation i.e proposition $\ref{abc}$. We have 
\begin{eqnarray*}
\int_{D_{1}} v \, dy &=& < v ; e^{(1)}_{n_{0}} > \; \int_{D} e_{n_{0}} \, dx + a^{2} \;  \sum_{n \neq n_{0} }  < \tilde{v} ; \overline{e}_{n} >  \; \int_{B} \overline{e}_{n} \, d\eta \\
& \stackrel{(\ref{W})}= & < v ; e^{(1)}_{n_{0}} >  \int_{D} e_{n_{0}} dx + \mathcal{O}(a^{2})  \sum_{n \neq n_{0}} \Bigg[\frac{< \tilde{u_{0}} ; \overline{e}_{n} >}{(1-\omega^{2}  \mu_{0}  \tau \lambda_{n})} + \mathcal{O}(\vert \log(a) \vert^{-h}) <1,\overline{e}_{n}>  \Bigg]  \int_{B} \overline{e}_{n}  d\eta. 
\end{eqnarray*} 
When the used frequency is not close to the resonance the following estimation holds
\begin{equation*}
\sum_{n \neq n_{0}} \Bigg[\frac{< \tilde{u_{0}} ; \overline{e}_{n} >}{(1-\omega^{2} \, \mu_{0} \, \tau \, \lambda_{n})} + \mathcal{O}(\vert \log(a) \vert^{-h}) <1,\overline{e}_{n}>  \Bigg] \; \int_{B} \overline{e}_{n} \, d\eta \sim \mathcal{O}(1),
\end{equation*}
and plug this in the previous equation to obtain
\begin{eqnarray*}
\int_{D_{1}} v \, dy & \stackrel{(\ref{equa825})}= & a^{2} \Bigg[ \frac{ < \tilde{u_{0}}; \overline{e}^{(1)}_{n_{0}} > }{(1-\omega^{2} \, \mu_{0} \, \tau \, \lambda_{n_{0}})-\omega^{2} \, \mu_{0} \,\tau \, a^{2} \, \Phi_{0} \Big(\int_{B} \overline{e}_{n_{0}}\Big)^{2}} + \mathcal{O}(1) \Bigg]  \; \int_{B} \overline{e}_{n_{0}} \, d\eta + \mathcal{O}(a^{2}). 
\end{eqnarray*}
Then
\begin{equation}\label{equa57}
\Big\vert \int_{D_{1}} v \, dy \Big\vert^{2} = a^{4}  \frac{ \vert < \tilde{u_{0}}; \overline{e}^{(1)}_{n_{0}} > \vert^{2}}{\Bigg\vert (1-\omega^{2} \, \mu_{0} \, \tau \, \lambda_{n_{0}})-\omega^{2} \, \mu_{0} \,\tau \, a^{2} \, \Phi_{0} \Big(\int_{B} \overline{e}_{n_{0}}\Big)^{2} \Bigg\vert^{2}}  \; \Big( \int_{B} \overline{e}_{n_{0}} \, d\eta  \Big)^{2} + \mathcal{O}(a^{4} \, \vert \log(a) \vert^{h}).
\end{equation}
Step 2: Estimation of $ \int_{D_{1}} \vert v \vert^{2} \, dy $. \\
We have
\begin{eqnarray*}
\int_{D_{1}} \vert v \vert^{2} \, dx &=& \sum_{n} \vert <v,e^{(1)}_{n}>  \vert^{2} = a^{2} \, \Big(\vert <\tilde{v_{1}},\overline{e}_{n_{0}}>  \vert^{2} + \sum_{n \neq n_{0}} \vert <\tilde{v_{1}},\overline{e}_{n}>  \vert^{2} \Big) \\
& = & a^{2} \, \vert <\tilde{v_{1}},\overline{e}_{n_{0}}>  \vert^{2} + \mathcal{O}(a^{2}) \\ 
& \stackrel{(\ref{equa825})}= & a^{2} \Bigg[\frac{ \vert < \tilde{u_{0}}; \overline{e}^{(1)}_{n_{0}} > \vert^{2}}{\Big\vert (1-\omega^{2} \, \mu_{0} \, \tau \, \lambda_{n_{0}})-\omega^{2} \, \mu_{0} \,\tau \, a^{2} \, \Phi_{0}(z_{1},z_{2}) \Big(\int_{B} \overline{e}_{n_{0}}\Big)^{2} \Big\vert^{2}} + \mathcal{O}(\vert \log(a) \vert^{h}) \Bigg] + \mathcal{O}(a^{2}).  
\end{eqnarray*}
Then
\begin{equation}\label{equa58}
\int_{D_{1}} \vert v \vert^{2} \, dx = a^{2} \frac{ \vert < \tilde{u_{0}}; \overline{e}^{(1)}_{n_{0}} > \vert^{2}}{\Bigg\vert (1-\omega^{2} \, \mu_{0} \, \tau \, \lambda_{n_{0}})-\omega^{2} \, \mu_{0} \,\tau \, a^{2} \, \Phi_{0}(z_{1},z_{2}) \Big(\int_{B} \overline{e}_{n_{0}}\Big)^{2} \Bigg\vert^{2}} + \mathcal{O}(a^{2} \, \vert \log(a) \vert^{h}). 
\end{equation}
Combining $(\ref{equa57})$ and $(\ref{equa58})$, we obtain
\begin{eqnarray*}
\Big\vert \int_{D_{1}} v \, dy \Big\vert^{2} &=& a^{4}  \frac{ \vert < \tilde{u_{0}}; \overline{e}^{(1)}_{n_{0}} > \vert^{2}}{\Bigg\vert (1-\omega^{2} \, \mu_{0} \, \tau \, \lambda_{n_{0}})-\omega^{2} \, \mu_{0} \,\tau \, a^{2} \, \Phi_{0}(z_{1},z_{2}) \Big(\int_{B} \overline{e}_{n_{0}}\Big)^{2} \Bigg\vert^{2}}  \; \Big( \int_{B} \overline{e}_{n_{0}} \, d\eta  \Big)^{2} + \mathcal{O}(a^{4} \, \vert \log(a) \vert^{h})  \\
&=& a^{2} \Bigg[ \int_{D_{1}} \vert v \vert^{2} \, dx +  \mathcal{O}(a^{2} \, \vert \log(a) \vert^{h}) \Bigg] \; \Big( \int_{B} \overline{e}_{n_{0}} \, d\eta  \Big)^{2}  + \mathcal{O}(a^{4} \, \vert \log(a) \vert^{h})  \\
&=& a^{2} \, \int_{D_{1}} \vert v \vert^{2} \, dx \; \Big( \int_{B} \overline{e}_{n_{0}} \, d\eta  \Big)^{2} + \mathcal{O}(a^{4} \, \vert \log(a) \vert^{h}),
\end{eqnarray*} 
which proves the formula $(\ref{equa56})$.
\end{proof}
We continue with equation $(\ref{equa55})$, then 
\begin{equation*}
\Bigg\vert \textbf{C}^{-1} - \Big( \Phi_{0}+2(- \log(k) / 2\pi +\Gamma)\Big) \Bigg\vert^{2} = \frac{\vert u_{0}(z_{2}) \vert^{2}}{\vert  \omega^{2} \; \mu_{0} \; \tau \vert^{2} \; \Big\vert \displaystyle\int_{D_{2}} v \, dy \Big\vert^{2}} +\mathcal{O}(d \, \vert \log(a) \vert^{2(1-h)}) 
\end{equation*}
\begin{eqnarray*}
\phantom{Invisible text} \quad \quad &\stackrel{(\ref{equa56})}=& \frac{\vert u_{0}(z_{2}) \vert^{2}}{\vert  \omega^{2} \; \mu_{0} \; \tau \vert^{2} \; \Big[ a^{2} \; \Big(\displaystyle \int_{B} \overline{e}_{n_{0}} dy \Big)^{2} \displaystyle\int_{D_{2}} \vert v \vert^{2} \, dy  + \mathcal{O}(a^{4} \vert \log(a) \vert^{h} ) \Big]} +\mathcal{O}(d \, \vert \log(a) \vert^{2(1-h)}) \\
&=& \frac{\vert u_{0}(z_{2}) \vert^{2}}{\vert \omega^{2} \; \mu_{0} \; \tau \vert^{2} \;  a^{2} \; \Big(\displaystyle \int_{B} \overline{e}_{n_{0}} dy \Big)^{2} \displaystyle\int_{D_{2}} \vert v \vert^{2} \, dy \Big[1 + \mathcal{O}(\vert \log(a) \vert^{-h} ) \Big]} +\mathcal{O}(d \, \vert \log(a) \vert^{2(1-h)}) \\
&=& \frac{\vert u_{0}(z_{2}) \vert^{2}}{\vert \omega^{2} \; \mu_{0} \; \tau \vert^{2} \;  a^{2} \; \Big(\displaystyle \int_{B} \overline{e}_{n_{0}} dy \Big)^{2} \displaystyle\int_{D_{2}} \vert v \vert^{2} \, dy } + \mathcal{O}(\vert \log(a) \vert^{2-3h}). 
\end{eqnarray*}
In the following proposition, we write an estimation of $\vert u_{0}(z_{2}) \vert$ in the case of one particle inside the domain.  
\begin{proposition}
We have
\begin{equation}\label{vVonelemma}
\vert u_{0}(z_{2}) \vert^{2} = \frac{\Big\vert 1-\omega^{2} \mu_{0} \tau \, \lambda_{n_{0}} -  \omega^{2} \mu_{0} \tau \, \Big( \frac{-1}{2\pi} \log(k)+\Gamma \Big) \,\Big( \int_{D} e_{n_{0}}  \Big)^{2} \Big\vert^{2}}{ \big(\int_{D} e_{n_{0}} \big)^{2}} \; \int_{D} \vert u_{1} \vert^{2} \, dx + \mathcal{O}\Big(\vert \log(a) \vert^{\max(-2h,-1)}\Big).
\end{equation}
\end{proposition} 
\begin{proof}
To fix notations recall L.S.E for one particle
\begin{equation*}
u_{1}(x) - \omega^{2} \, \mu_{0} \, \int_{D} G_{k}(x,y) \, (\varepsilon_{p}-\varepsilon_{0})(y) \, u_{1}(y) \, dy = u_{0}(x),  \qquad x \in D. 
\end{equation*}
With this notation the equation $(\ref{Z})$ takes the following form
\begin{equation}\label{equa511}
<u_{1};e_{n_{0}}> = \frac{ <u_{0};e_{n_{0}}>}{\Bigg[1-\omega^{2} \mu_{0} \tau \, \lambda_{n_{0}} -  \omega^{2} \mu_{0} \tau \, \Big( \frac{-1}{2\pi} \log(k)+\Gamma \Big) \,\Big( \int_{D} e_{n_{0}}  \Big)^{2}\Bigg]}  + \mathcal{O}(a\,\vert \log(a) \vert^{h-1}). 
\end{equation}
Next, 
\begin{eqnarray}\label{J}
\nonumber
\int_{D} \vert u_{1} \vert^{2} \, dx & = & \vert <u_{1} ;e_{n_{0}}> \vert^{2} +   a^{2} \, \sum_{n \neq n_{0} } \, \vert <\tilde{u}_{1} ; \overline{e}_{n}> \vert^{2} \\
& \stackrel{(\ref{equa511})}= &  \frac{\vert <u_{0};e_{n_{0}}> \vert^{2}}{\Big\vert 1-\omega^{2} \mu_{0} \tau \, \lambda_{n_{0}} -  \omega^{2} \mu_{0} \tau \, \Big( \frac{-1}{2\pi} \log(k)+\Gamma \Big) \,\Big( \int_{D} e_{n_{0}}  \Big)^{2} \Big\vert^{2}} \,  + \mathcal{O}\Big(a^{2} \, \vert \log(a) \vert^{\max(0,2h-1)}\Big). 
\end{eqnarray}
We develop $u_{0}$ near the point $z$ to obtain 
\begin{eqnarray*}
\int_{D} \vert u_{1} \vert^{2} \, dx & = & \frac{\Big[ \vert u_{0}(z_{2}) \vert^{2} \, \Big( \int_{D} e_{n_{0}} dx \Big)^{2} + \mathcal{O}(a^{3}) \Big]}{\Big\vert 1-\omega^{2} \mu_{0} \tau \, \lambda_{n_{0}} -  \omega^{2} \mu_{0} \tau \, \Big( \frac{-1}{2\pi} \log(k)+\Gamma \Big) \,\Big( \int_{D} e_{n_{0}}  \Big)^{2} \Big\vert^{2}} \,  +\mathcal{O}\Big(a^{2} \, \vert \log(a) \vert^{\max(0,2h-1)}\Big) \\
& = & \frac{\vert u_{0}(z_{2}) \vert^{2} \, \Big( \int_{D} e_{n_{0}} dx \Big)^{2}}{\Big\vert 1-\omega^{2} \mu_{0} \tau \, \lambda_{n_{0}} -  \omega^{2} \mu_{0} \tau \, \Big( \frac{-1}{2\pi} \log(k)+\Gamma \Big) \,\Big( \int_{D} e_{n_{0}}  \Big)^{2} \Big\vert^{2}} \,  + \mathcal{O}\Big(a^{2} \, \vert \log(a) \vert^{\max(0,2h-1)}\Big).
\end{eqnarray*}
This proves $(\ref{vVonelemma})$.
\end{proof}
In $(\ref{vVonelemma})$, we use the following notation
\begin{equation*}
\Psi :=  \Bigg\vert 1-\omega^{2} \mu_{0} \tau \, \lambda_{n_{0}} -  \omega^{2} \mu_{0} \tau \, \Big( \frac{-1}{2\pi} \log(k)+\Gamma \Big) \,\Big( \int_{D} e_{n_{0}}  \Big)^{2} \Bigg\vert^{2}.
\end{equation*}  
With this, we get
\begin{eqnarray*}
\left\vert \textbf{C}^{-1} - \bigg( \Phi_{0}+2 \bigg( \frac{-1}{2\pi}  \log(k) +\Gamma \bigg)\bigg) \right\vert^{2} &=& \frac{\vert u_{0}(z_{2}) \vert^{2}}{\vert \omega^{2} \; \mu_{0} \; \tau \vert^{2}  \; \Big(\displaystyle \int_{D} e_{n_{0}} dy \Big)^{2} \displaystyle\int_{D_{2}} \vert v \vert^{2} \, dy } + \mathcal{O}(\vert \log(a) \vert^{2-3h}) \\
&\stackrel{(\ref{vVonelemma})}=& \frac{\Psi \; \displaystyle \int_{D} \vert u_{1} \vert^{2} \, dx}{\vert \omega^{2} \; \mu_{0} \; \tau \vert^{2} \;  \Big(\displaystyle \int_{D} e_{n_{0}} dy \Big)^{4} \displaystyle\int_{D_{2}} \vert v \vert^{2} \, dy }  +   \mathcal{O}(\vert \log(a) \vert^{2-3h}).
\end{eqnarray*}
We set 
\begin{equation}\label{E}
B := \frac{\displaystyle \int_{D} \vert u_{1} \vert^{2} \, dx}{\vert \omega^{2} \; \mu_{0} \; \tau \vert^{2} \;  \Big(\displaystyle \int_{D} e_{n_{0}} dy \Big)^{4} \displaystyle\int_{D_{2}} \vert v \vert^{2} \, dy }.  
\end{equation}
Referring to $(\ref{phigamma})$, we set $\Gamma := \gamma + i/4 $. We develop the left side of the last equation as 
\begin{eqnarray*}
\left\vert \textbf{C}^{-1} - \Bigg[ \Phi_{0}+2 \Big( \frac{-1}{2 \pi} \log(k) +\Gamma\Big)\Bigg] \right\vert^{2} &=& \Big( \textbf{C}^{-1} - \Phi_{0}\Big)^{2} - 4 \, \Big( \textbf{C}^{-1} - \Phi_{0}\Big) \Bigg(\frac{-1}{2\pi} \log\vert k \vert + \gamma \Bigg) \\
& + & 4 \, \Bigg(\frac{-1}{2\pi} \log\vert k \vert + \gamma \Bigg)^{2} + 4 \, \Bigg(\frac{-1}{2\pi} Arg(k)  + \frac{1}{4} \Bigg)^{2},
\end{eqnarray*}
then, we have
\begin{eqnarray}\label{D}
\nonumber
\Big( \textbf{C}^{-1} - \Phi_{0}\Big)^{2} - 4 \, \Big( \textbf{C}^{-1} - \Phi_{0}\Big) \Bigg(\frac{-1}{2\pi} \log\vert k \vert + \gamma \Bigg) 
 &+& 4 \, \Bigg(\frac{-1}{2\pi} \log\vert k \vert + \gamma \Bigg)^{2} + 4 \, \Bigg(\frac{-1}{2\pi} Arg(k)  + \frac{1}{4} \Bigg)^{2} \\ &=& \Psi \, B + \mathcal{O}(\vert \log(a) \vert^{2-3h}).
\end{eqnarray}
Remark that $\Psi$ can be written as
\begin{eqnarray*}
\Psi &=& \Big\vert 1-\omega^{2} \mu_{0} \tau \, \lambda_{n_{0}}  \Big\vert^{2} + (\omega^{2} \mu_{0})^{2} \, \vert \tau \vert^{2} \, \Big(\int_{D} e_{n_{0}}\Big)^{4} \,\left[ \Bigg( \frac{-1}{2\pi} \log\vert k \vert + \gamma \Bigg)^{2} + \Bigg( \frac{-1}{2\pi} Arg(k) + \frac{1}{4} \Bigg)^{2} \right] \\ 
&-& 2 \omega^{2} \mu_{0} \Big(\int_{D} e_{n_{0}}\Big)^{2} \Bigg( \frac{-1}{2\pi} \log\vert k \vert + \gamma \Bigg) \, \Re\Big[\overline{\tau} \; (1-\omega^{2}\mu_{0}\tau \lambda_{n_{0}}) \Big]+\mathcal{O}(a^{2}). 
\end{eqnarray*}
Hence using $(\ref{C})$, we have
\begin{eqnarray*}
\Psi &=& \textbf{C}^{-2} \, (\omega^{2} \, \mu_{0})^{2} \,  \vert \tau \vert^{2} \, \Big( \int_{D} e_{n_{0}} \Big)^{4} + (\omega^{2} \mu_{0})^{2} \, \vert \tau \vert^{2} \, \Big(\int_{D} e_{n_{0}}\Big)^{4} \, \left[ \Bigg( \frac{-1}{2\pi} \log\vert k \vert + \gamma \Bigg)^{2} +  \, \Bigg( \frac{-1}{2\pi} Arg(k) + \frac{1}{4} \Bigg)^{2} \right] \\ 
&-&  2 \, \textbf{C}^{-1} \, (\omega^{2} \mu_{0})^{2} \, \vert \tau \vert^{2} \, \Big(\int_{D} e_{n_{0}}\Big)^{4} \Bigg( \frac{-1}{2\pi} \log\vert k \vert + \gamma \Bigg) + \mathcal{O}\big(\vert \log(a) \vert^{-3h}\big). 
\end{eqnarray*}
Replace $\Psi$ in $(\ref{D})$ and use the fact that $B = \mathcal{O}(\vert \log(a) \vert^{2})$ to cancel all the terms of order $\mathcal{O}(1)$. The formula $(\ref{D})$ will be
\begin{equation*}
\Big( \textbf{C}^{-1} - \Phi_{0}\Big)^{2} - 4 \, \Big( \textbf{C}^{-1} - \Phi_{0}\Big) \Bigg(\frac{-1}{2\pi} \log\vert k \vert + \gamma \Bigg)  =  - 2 \, \textbf{C}^{-1} \, (\omega^{2} \mu_{0})^{2} \, \vert \tau \vert^{2} \, \Big(\int_{D} e_{n_{0}}\Big)^{4} \Bigg( \frac{-1}{2\pi} \log\vert k \vert + \gamma \Bigg) \, B 
\end{equation*}
\begin{equation*}
\phantom{invisible text}   \qquad + \qquad  \textbf{C}^{-2} \, (\omega^{2} \, \mu_{0})^{2} \,  \vert \tau \vert^{2} \, \Big( \int_{D} e_{n_{0}} \Big)^{4} B + \mathcal{O}\big(\vert \log(a) \vert^{\max(0,2-3h)}\big).
\end{equation*}
Then
\begin{eqnarray*}
\Bigg(\frac{-1}{2\pi} \log\vert k \vert + \gamma \Bigg) \Bigg[ - 4 \, \Big( \textbf{C}^{-1} - \Phi_{0}\Big) +  2 \, \textbf{C}^{-1} \, (\omega^{2} \mu_{0})^{2} \, \vert \tau \vert^{2} \, \Big(\int_{D} e_{n_{0}}\Big)^{4}  \, B  \Bigg] & = &   \textbf{C}^{-2} \, (\omega^{2} \, \mu_{0})^{2} \,  \vert \tau \vert^{2} \, \Big( \int_{D} e_{n_{0}} \Big)^{4} B 
\end{eqnarray*}
\begin{eqnarray*}
\phantom{invisible text} \qquad \qquad \qquad \qquad \qquad \qquad \qquad \qquad \qquad &-& \Big( \textbf{C}^{-1} - \Phi_{0}\Big)^{2} + \mathcal{O}\big(\vert \log(a) \vert^{\max(0,2-3h)}\big).
\end{eqnarray*}
Using $(\ref{E})$, we get an explicit expression 
\begin{equation}\label{Kds}
\log\vert k \vert = 2 \pi \, \gamma - \frac{\pi}{\textbf{C}} \; \frac{\frac{\displaystyle \int_{D} \vert u_{1} \vert^{2} \, dx}{ \displaystyle\int_{D_{2}} \vert v \vert^{2} \, dx } - (1 - \textbf{C} \, \Phi_{0})^{2}}{ \frac{\displaystyle \int_{D} \vert u_{1} \vert^{2} \, dx}{ \displaystyle\int_{D_{2}} \vert v \vert^{2} \, dx } - 2 \, (1 - \textbf{C} \, \Phi_{0})} + \mathcal{O}\Big(\vert \log(a) \vert^{\max(h-1;1-2h)}\Big).
\end{equation}
\begin{remark}
To justify that $(\ref{Kds})$ is well defined, we use $(\ref{J}), (\ref{equa58})$ and $(\ref{C})$ to obtain the following relation 
\begin{equation*}
\frac{\displaystyle\int_{D} \vert u_{1} \vert^{2} \, dx}{\displaystyle \int_{D_{2}} \vert v \vert^{2} \, dx } = \frac{\Big( 1 - \textbf{C} \, \Phi_{0} \Big)^{2}}{\Bigg\vert 1 - \textbf{C} \, \Big(\frac{-1}{2\pi} \, \log(k) + \Gamma \Big) \Bigg\vert^{2}} + \; \mathcal{O}\Big(\vert \log(a) \vert^{\max(h-1;-h)}\Big).
\end{equation*}
Hence, 
\begin{equation*}
\frac{\frac{\displaystyle \int_{D} \vert u_{1} \vert^{2} \, dx}{ \displaystyle\int_{D_{2}} \vert v \vert^{2} \, dx } - (1 - \textbf{C} \, \Phi_{0})^{2}}{ \frac{\displaystyle \int_{D} \vert u_{1} \vert^{2} \, dx}{ \displaystyle\int_{D_{2}} \vert v \vert^{2} \, dx } - 2 \, (1 - \textbf{C} \, \Phi_{0})} = \frac{\Big(1 - \textbf{C} \, \Phi_{0} \Big) \, \textbf{C} \, \Bigg\lbrace 2 \, \Re\Big[ \frac{-1}{2\pi} \, \log(k) + \Gamma \Big] - \textbf{C} \, \Big\vert \frac{-1}{2\pi} \, \log(k) + \Gamma \Big\vert^{2} \Bigg\rbrace}{- 1 - \textbf{C} \, \Phi_{0} +2 \, \textbf{C} \, \Bigg\lbrace 2 \, \Re\Big[ \frac{-1}{2\pi} \, \log(k) + \Gamma \Big] - \textbf{C} \, \Big\vert \frac{-1}{2\pi} \, \log(k) + \Gamma \Big\vert^{2} \Bigg\rbrace} \thicksim \mathcal{O}(C).
\end{equation*}
\end{remark}
Therefore the error term in $(\ref{Kds})$ is indeed negligible as soon as $\frac{1}{2} < h <1$. \\  
Taking the exponential in both side of $(\ref{Kds})$ and using the smallness of $\mathcal{O}\Big(\vert \log(a) \vert^{\max(1-2h,h-1)}\Big)$, we write 
\begin{equation*}\label{equa516}
\vert k \vert  =  \exp\left\lbrace  2 \pi \, \gamma - \frac{\pi}{\textbf{C}} \; \frac{\frac{\displaystyle \int_{D} \vert u_{1} \vert^{2} \, dx}{ \displaystyle\int_{D_{2}} \vert v \vert^{2} \, dx } - (1 - \textbf{C} \, \Phi_{0})^{2}}{ \frac{\displaystyle \int_{D} \vert u_{1} \vert^{2} \, dx}{ \displaystyle\int_{D_{2}} \vert v \vert^{2} \, dx } - 2 \, (1 - \textbf{C} \, \Phi_{0})}  \right\rbrace +  \mathcal{O}\Big(\vert \log(a) \vert^{\max(h-1;1-2h)}\Big).
\end{equation*}

\bigskip
\section{Proof of Theorem \ref{Using-both-permittivity-and-conductivity-contrasts}}\label{Theo-using-both-contrasted-permit-conduct}
We recall the model problem for photo-acoustic imaging: 
\begin{equation}\label{equa91}
\left\{
\begin{array}{rll}
    \partial^{2}_{t} p(x,t) -  \Delta_{x} p(x,t) &=& 0 \qquad in \quad \mathbb{R}^{2} \times \mathbb{R}^{+},\\
    p(x,0) &=& \frac{\omega \, \beta_{0}}{c_{p}} \Im(\varepsilon)(x) \, \vert u \vert^{2}(x) \,\chi_{\Omega}  , \qquad in \quad \mathbb{R}^{2} \\ 
    \partial_{t}p(x,0) &=& 0 \qquad in \quad \mathbb{R}^{2}. 
    \end{array}
\right.
\end{equation}
\begin{remark}\label{FGH}
Next, when we solve the equation $(\ref{equa91})$, we omit the multiplicative term\footnote{The constant 2$\pi$ in the denominator comes from the Poisson formula.} $\frac{\omega \, \beta_{0}}{2\pi \, c_{p}}$. 
\end{remark}
\subsection{Photo-acoustic imaging using one particle. }\label{opsubsection}
Proof of $(\ref{pressure-to-v_0})$.
\
\\ The next lemma gives an estimation of the total field for $x \in \Omega \setminus D$.
\begin{lemma}
The total field behaves as
\begin{equation}\label{awayD}
\vert u_{1}(x) \vert^{2} = \mathcal{O}(1) + \mathcal{O}(\vert \log(a) \vert^{h-1} \, \vert \log(dist) \vert)  \quad  in  \quad \Omega \setminus D,
\end{equation}
where $dist = dist(x,D)$.
\end{lemma}
\begin{proof}
We use L.S.E
\begin{equation}\label{LSE1p}
u_{1}(x) = u_{0}(x) + \omega^{2} \, \mu_{0} \, \int_{D} (\varepsilon_{p}-\varepsilon_{0})(y) G_{k}(y,x) u_{1}(y) dy, \qquad \qquad x \in \mathbb{R}^{2}.
\end{equation}
Now, for $x$ away from $D$
\begin{eqnarray*}
\vert u_{1}(x) \vert & \leq & \vert u_{0}(x) \vert + \mathcal{O}\Bigg( \frac{1}{a^{2}\, \vert \log(a) \vert} \int_{D}  \vert G_{0} \vert (y,x) \; \vert u_{1}(y) \vert \; dy \Bigg) \\
& = & \mathcal{O}(1) + \mathcal{O}\Bigg( \frac{1}{a^{2}\, \vert \log(a) \vert} \; \Vert u_{1} \Vert \Big[ \int_{D}  \vert G_{0} \vert^{2} (y,x)  \; dy\Big]^{1/2} \Bigg)  =  \mathcal{O}(1) + \mathcal{O}\Big( \vert \log(a) \vert^{h-1} \, \vert \log(dist) \vert  \Big) 
\end{eqnarray*}
This proves $(\ref{awayD})$. 
\end{proof}
Let us recall from proposition $(\ref{Y})$, the following relation
\begin{equation}\label{ve=Ve}
<u_{1},e_{n_{0}}> = \frac{1}{[1-\omega^{2} \, \mu_{0} \, \tau \, \lambda_{n_{0}}]} \; <u_{0};e_{n_{0}}> + \mathcal{O}(a \, \vert \log(a) \vert^{2h-1}).
\end{equation}
We use Poisson's formula to solve the system $(\ref{equa91})$, see (\cite{PinchoverRubinstein}, Chapter 9), to represent the pressure as follows
\begin{eqnarray*}
p(t,x) &=& \partial_{t} \int_{\vert x-y \vert <t} \frac{(\Im(\varepsilon_{p}) \vert u_{1} \vert^{2})(y)}{\sqrt{t^{2}-\vert x-y \vert^{2}}} \chi_{D} dy + \partial_{t} \int_{\vert x-y \vert <t} \frac{(\Im(\varepsilon_{0}) \vert u_{1} \vert^{2})(y)}{\sqrt{t^{2}-\vert x-y \vert^{2}}} \chi_{\Omega \setminus D} dy \\
& = & \partial_{t} \int_{\vert x-y \vert <t} \frac{(\Im(\varepsilon_{p} - \varepsilon_{0}) \vert u_{1} \vert^{2})(y)}{\sqrt{t^{2}-\vert x-y \vert^{2}}} \chi_{D} dy + \partial_{t} \int_{\vert x-y \vert <t} \frac{(\Im(\varepsilon_{0}) \vert u_{1} \vert^{2})(y)}{\sqrt{t^{2}-\vert x-y \vert^{2}}} \chi_{\Omega} dy.
\end{eqnarray*}
Let $t > diam(\Omega)$. For $x \in \partial \Omega$, the representation above reduces to:  
\begin{equation*}
p(t,x)  =  \int_{D(z,a)} \partial_{t} \, \frac{(\Im(\varepsilon_{p} - \varepsilon_{0}) \vert u_{1} \vert^{2})(y)}{\sqrt{t^{2}-\vert x-y \vert^{2}}} dy + \int_{\Omega} \partial_{t} \frac{(\Im(\varepsilon_{0}) \vert u_{1} \vert^{2})(y)}{\sqrt{t^{2}-\vert x-y \vert^{2}}}  dy. 
\end{equation*}
Set $T_{4}$ to be 
\begin{equation*}
T_{4} := \int_{\Omega} \partial_{t} \frac{(\Im(\varepsilon_{0}) \vert u_{1} \vert^{2})(y)}{\sqrt{t^{2}-\vert x-y \vert^{2}}}  dy.
\end{equation*}
Recalling that $\tau := \varepsilon_{p} - \varepsilon_{0}(z)$, we have
\begin{equation*}
p(t,x)  =  -t \, \Im(\tau) \, \int_{D(z,a)} \frac{\, \vert u_{1} \vert^{2}(y)}{(t^{2}-\vert x-y \vert^{2})^{3/2}} dy +  T_{4} +    \int_{D(z,a)}\vert u_{1} \vert^{2}(y) \;\; \partial_{t} \frac{\Im\Big(\int_{0}^{1}(y-z) \centerdot \nabla \varepsilon_{0}(z+s(y-z))ds\Big)}{\sqrt{t^{2}-\vert x-y \vert^{2}}}  dy   
\end{equation*}
We estimate the remainder term as follows
\begin{equation}\label{lb} 
\Bigg\vert \int_{D(z,a)}\vert u_{1} \vert^{2}(y) \;\; \partial_{t} \frac{\Im\Big(\int_{0}^{1}(y-z) \centerdot \nabla \varepsilon_{0}(z+s(y-z))ds\Big)}{\sqrt{t^{2}-\vert x-y \vert^{2}}}  dy \Bigg\vert \\ 
 \leq  a \, \Vert u_{1} \Vert^{2}_{\mathbb{L}^{2}(D)} = \mathcal{O}(a^{3} \; \vert \log(a) \vert^{2h}), 
\end{equation}
then
\begin{equation*}
p(t,x)  =  -t \, \Im(\tau) \, \int_{D(z,a)} \frac{\, \vert u_{1} \vert^{2}(y)}{(t^{2}-\vert x-y \vert^{2})^{3/2}} dy + T_{4} + \mathcal{O}(a^{3} \; \vert \log(a) \vert^{2h}).
\end{equation*}
By Taylor expansion of the function $\Big(t^{2} - \vert x-\centerdot \vert^{2} \Big)^{-3/2}$ near $z$, we have
\begin{eqnarray*}
p(t,x)  &=&  \frac{-t \, \Im(\tau)}{(t^{2}-\vert x-z \vert^{2})^{3/2}}  \, \int_{D(z,a)} \vert u_{1} \vert^{2}(y) dy + T_{4} \\ 
&+& \mathcal{O}(a^{3} \; \vert \log(a) \vert^{2h}) + \mathcal{O}\Big[ \, \Im(\tau) \, \int_{D(z,a)} (\vert y-z \vert^{2} + 2 <x-z;z-y>) \vert u_{1} \vert^{2}(y) dy \Big].  
\end{eqnarray*}
We estimate the remainder term as 
\begin{equation}\label{hdd}
\bigg\vert \Im(\tau) \int_{D} (\vert y-z \vert^{2} + 2 <x-z;z-y>) \vert u_{1} \vert^{2}(y) dy \bigg\vert  \leq   \Im(\tau) \, a \, \Vert u_{1} \Vert^{2} = \mathcal{O}(\Im(\tau)  a^{3}  \vert \log(a) \vert^{2h}), 
\end{equation}
and then
\begin{equation*}
p(t,x) = \frac{-t \, \Im(\tau)}{(t^{2}-\vert x-z \vert^{2})^{3/2}}  \, \int_{D(z,a)} \vert u_{1} \vert^{2}(y) dy + T_{4} + \mathcal{O}\big(a^{3} \;  \Im(\tau) \; \vert \log(a) \vert^{2h}\big).
\end{equation*}
Writing $u_{1}$ as a Fourier series over the basis $\big(e_{n}\big)_{n \in \mathbb{N}}$, we obtain  
\begin{equation*}
p(t,x)  =  \frac{-t \, \Im(\tau) \, \vert <u_{1};e_{n_{0}}> \vert^{2}}{(t^{2}-\vert x-z \vert^{2})^{3/2}}  \,  - \frac{t \, \Im(\tau)}{(t^{2}-\vert x-z \vert^{2})^{3/2}}  \, \sum_{n \neq n_{0}} \vert <u_{1};e_{n}> \vert^{2} + T_{4} + \mathcal{O}\big(a^{3} \;  \Im(\tau) \; \vert \log(a) \vert^{2h}\big), 
\end{equation*}
since $n \neq n_{0}$ we estimate the series as
\begin{equation}\label{oyaya}
\mathcal{O}\big( \Im(\tau) \, \sum_{n \neq n_{0}} \vert <u_{1};e_{n}> \vert^{2} \big) \sim \mathcal{O}\big( \Im(\tau) \, \Vert u_{0} \vert_{\mathbb{L}^{2}(D)}^{2} \big) = \mathcal{O}\big( \Im(\tau) \, a^{2} \big).
\end{equation}
Next,
\begin{eqnarray*}
p(t,x) & \stackrel{(\ref{ve=Ve})}= & \frac{-t \, \Im(\tau)}{(t^{2}-\vert x-z \vert^{2})^{3/2}}  \, \Bigg[ \frac{\vert <u_{0};e_{n_{0}}> \vert^{2}}{\vert 1- \omega^{2} \mu_{0} \lambda_{n_{0}} \tau \vert^{2}} \,  + \mathcal{O}(a^{2} \, \vert \log(a) \vert^{3h-1}) \Bigg] + T_{4} + \mathcal{O}\big( \Im(\tau) \, a^{2} \big),
\end{eqnarray*}
hence
\begin{equation}\label{P(t,x)} 
p(t,x) = \frac{-t \, \Im(\tau)}{(t^{2}-\vert x-z \vert^{2})^{3/2}}  \, \frac{\vert <u_{0};e_{n_{0}}> \vert^{2}}{\vert 1- \omega^{2} \mu_{0} \lambda_{n_{0}} \tau \vert^{2}} \,   + T_{4} + \mathcal{O}(\Im(\tau) \, a^{2}) + \mathcal{O}(\Im(\tau) \, a^{2} \, \vert \log(a) \vert^{3h-1}).
\end{equation}
In order to calculate the terme $T_{4}$, we use L.S.E
\begin{equation*}
u_{1}(x) - \omega^{2} \, \mu_{0} \int_{D} (\varepsilon_{p}-\varepsilon_{0}(\eta)) G_{k}(x,\eta) \, u_{1}(\eta)  d\eta = u_{0}(x) \; \quad in \quad \Omega,
\end{equation*}
and define
\begin{equation}\label{p0np}
p_{0}(t,x) := \int_{\Omega} \, \partial_{t} \, \frac{1}{\sqrt{t^{2}-\vert x-y \vert^{2}}} \; \Im(\varepsilon_{0})(y) \; \vert u_{0} \vert^{2}(y) \,  dy.  
\end{equation} 
Observe that $p_{0}(t,x)$ is the measured pressure at point $x \in \partial \Omega$ and time $t$ when no particle is inside $\Omega$. \\
We set
\begin{equation*}
f := \partial_{t} \, \frac{1}{\sqrt{t^{2}-\vert x-y \vert^{2}}} \; \Im(\varepsilon_{0})(y).
\end{equation*} 
With this, we get 
\begin{equation*}
T_{4} = \int_{\Omega} \, \partial_{t} \, \frac{1}{\sqrt{t^{2}-\vert x-y \vert^{2}}} \; \Im(\varepsilon_{0})(y) \; \vert u_{1} \vert^{2}(y) \,  dy  =  \int_{\Omega \setminus D} \, f \, \vert u_{1} \vert^{2}(y) \,  dy + \int_{D} \, f \, \vert u_{1} \vert^{2}(y) \,  dy.
\end{equation*}
If we compare $(\ref{aprioriestimation})$ to $(\ref{awayD})$ we deduce that the term
$(\star) := \int_{D} \, f \, \vert u_{1} \vert^{2}(y) \,  dy$
is less dominant than the one given on $\Omega \setminus D$. Now,
since $f$ is smooth we can estimate $(\star)$, with help of a priori estimation, as 
\begin{equation*}
\vert (\star) \vert = \Big\vert \int_{D} \, f \, \vert u_{1} \vert^{2}(y) \,  dy \Big\vert \leq \Vert u_{1} \Vert^{2} = \mathcal{O}\big(a^{2} \, \vert \log(a) \vert^{2h} \big),
\end{equation*}
and, from L.S.E, see for instance $( \ref{LSE1p})$, we can rewrite $T_{4}$ as 
\begin{eqnarray*}\label{defS3}
T_{4} & = & p_{0}(t,x)  - \int_{D} f \vert u_{0} \vert^{2}(y) dy + (\omega^{2} \mu_{0} )^{2} \int_{\Omega \setminus D} f \Big\vert \int_{D}(\varepsilon_{p}-\varepsilon_{0}(\eta)) G_{k}(\eta,y)  u_{1}(\eta)  d\eta \Big\vert^{2} dy \\ 
&+& 2 \; \omega^{2}  \mu_{0}   \Re\Bigg[\int_{\Omega \setminus D} f \overline{u}_{0}(y)  \int_{D} (\varepsilon_{p}-\varepsilon_{0}(\eta)) G_{k}(\eta,y)  u_{1}(\eta)  d\eta dy \Bigg] + \mathcal{O}\big( a^{2} \, \vert \log(a) \vert^{2h} \big). 
\end{eqnarray*}
The smoothness of $u_{0}$ is enough to justify the following estimation 
\begin{equation*}
\Big\vert \int_{D} f \vert u_{0} \vert^{2}(y) dy \Big\vert  \sim \mathcal{O}\big( a^{2} \big).
\end{equation*}
To finish the estimation of $T_{4}$ we still have to deal with two terms. More exactly we set  
\begin{equation}\label{S3}
S_{3} :=   \int_{\Omega \setminus D} \, f \; \Big\vert \int_{D}(\varepsilon_{p}-\varepsilon_{0}(\eta)) G_{k}(\eta,y) \, u_{1}(\eta)  d\eta \Big\vert^{2} dy. 
\end{equation}
Expanding $(\varepsilon_{p}-\varepsilon_{0}(.))$ near $z$, we obtain 
\begin{eqnarray*}
\vert S_{3} \vert & \leq & \vert \tau \vert^{2} \,  \int_{\Omega \setminus D} \, \big\vert f \big\vert \; \bigg(  \int_{D} \big\vert G_{k}(\eta,y) \, u_{1}(\eta) \big\vert d\eta \bigg)^{2} dy \\ 
&+&   \int_{\Omega \setminus D} \, \big\vert f \big\vert \; \bigg( \int_{D} \bigg\vert \int_{0}^{1} (z - \eta )\centerdot \nabla \varepsilon_{0}(z+s(\eta-z))ds G_{k}(\eta,y) \, u_{1}(\eta) \bigg\vert d\eta \bigg)^{2} dy  \\
& + &  2  \;  \int_{\Omega \setminus D} \, \big\vert f \;\big\vert \; \bigg\vert \Re\bigg[\overline{\tau} \, \int_{D} \overline{G_{k}}(\eta,y) \, \overline{u_{1}}(\eta) \, d\eta \int_{D} \int_{0}^{1} (z-\eta)\centerdot \nabla \varepsilon_{0}(z+s(\eta - z))ds G_{k}(\eta,y) \, u_{1}(\eta)  d\eta \bigg]\bigg\vert dy, 
\end{eqnarray*}
then apply Cauchy Schwartz inequality and exchange the integration variables to obtain 
\begin{equation*}
\vert S_{3} \vert \leq  \vert \tau \vert^{2} \, \Vert u_{1} \Vert^{2} \int_{D} J(\eta) \; d\eta + \mathcal{O}(a^{2}) \; \Vert u_{1} \Vert^{2} \int_{D}  J(\eta) \; d\eta +  \mathcal{O}(a \; \tau) \; \Vert u_{1} \Vert^{2} \int_{D}  J(\eta) \; d\eta 
 \lesssim  \vert \tau \vert^{2} \, \Vert u_{1} \Vert^{2} \int_{D} J(\eta) \; d\eta ,
\end{equation*}
where $J$ is the function given by \quad
$J(\eta) := \int_{\Omega \setminus D} \big\vert f \big\vert \; \big\vert G_{k}(\eta,y) \big\vert^{2} \, \, dy.$ \\
Remark that $J$ is a smooth function because $f$ is a smooth and  $\eta$ and $y$ are in two disjoint domains. 
Then
\begin{equation*}
S_{3} = \mathcal{O}(\vert \log(a) \vert^{2h-2}).
\end{equation*}
The last term to estimate, that we set as $S_{4}$, is more delicate. We split it as:
\begin{eqnarray*}
S_{4} &:=& 2 \, \omega^{2} \mu_{0} \, \Re \,\Bigg[ \int_{\Omega \setminus D} \, f \; \; \overline{u}_{0}(y) \, \int_{D} (\varepsilon_{p}-\varepsilon_{0}(\eta)) G_{k}(\eta,y) \, u_{1}(\eta) \, d\eta\,  dy \Bigg] \\
 & = &  2 \, \omega^{2} \mu_{0} \, \sum_{n}  \Re \,\Bigg[ <u_{1};e_{n}> \, \tau \, \int_{\Omega \setminus D} \, f \;  \; \overline{u}_{0}(y) \, \int_{D}  G_{k}(\eta,y) \, e_{n}(\eta) \, d\eta\,  dy \Bigg] \\
 & - & 2 \, \omega^{2} \mu_{0} \,   \Re \,\Bigg[  \, \int_{\Omega \setminus D} \, f \; \; \overline{u}_{0}(y) \, \, \int_{D} \, \int_{0}^{1} (z-\eta ) \centerdot \nabla \varepsilon_{0}(z+s(\eta -z)) ds\, G_{k}(\eta,y) \, u_{1}(\eta) \, d\eta\,  dy \Bigg]. 
\end{eqnarray*}
The same techniques, as previously, allows to estimate the second term of $S_{4}$ as $ \mathcal{O}\Big(a^{4} \, \vert \log(a) \vert^{h}\Big)$. Then
\begin{eqnarray*} 
  S_{4}  & = &  2 \, \omega^{2} \mu_{0} \;  \Re \,\Bigg[ <u_{1};e_{n_{0}}> \, \tau \, \int_{\Omega \setminus D} \, f \;  \; \overline{u}_{0}(y) \, \int_{D}  G_{k}(\eta,y) \, e_{n_{0}}(\eta) \, d\eta\,  dy \Bigg] \\
 & + &  \mathcal{O}\left( \sum_{n \neq n_{0}}  \Re \,\Bigg[ <u_{1};e_{n}> \, \tau \, \int_{\Omega \setminus D} \, f \; \; \overline{u}_{0}(y) \, \int_{D}  G_{k}(\eta,y) \, e_{n}(\eta) \, d\eta\,  dy \Bigg]\right)  +   \mathcal{O}\Big(a^{4} \, \vert \log(a) \vert^{h}\Big). 
\end{eqnarray*}
We keep the term with index $n_{0}$ and estimate the series as
\begin{eqnarray}\label{mla}
\nonumber
\vert \mathcal{O}(\cdots) \vert & \lesssim & \vert \tau \vert \,  \,  \int_{\Omega \setminus D} \,\big\vert f \big\vert \; \big\vert \overline{u}_{0}(y)\big\vert \, \sum_{n \neq n_{0}}  \Big\vert <u_{1};e_{n}> \,\Big\vert \, \bigg\vert \int_{D}  G_{k}(\eta,y) \, e_{n}(\eta) \, d\eta\, \bigg\vert dy\\
& \lesssim & \vert \tau \vert \; \Vert u_{1} \Vert  \; \Vert f \, \overline{u}_{0} \Vert_{\mathbb{L}^{2}(\Omega \setminus D)} \; \Bigg( \int_{D} \int_{\Omega \setminus D}  \, \big\vert   G_{k}(\eta,y) \, \big\vert^{2} dy d\eta \bigg)^{\frac{1}{2}} =\mathcal{O}\Big(\vert \log(a) \vert^{-1}\Big).
\end{eqnarray}
Plug this in the last equation to obtain 
\begin{eqnarray*}
S_{4} & \stackrel{(\ref{ve=Ve})}= &  2 \, \omega^{2} \mu_{0} \;  \Re \,\Bigg[ \frac{<u_{0};e_{n_{0}}>}{[1-\omega^{2}\, \mu_{0} \, \tau \, \lambda_{n_{0}}]} \, \tau \, \int_{\Omega \setminus D} \, f \; \; \overline{u}_{0}(y) \, \int_{D}  G_{k}(\eta,y) \, e_{n_{0}}(\eta) \, d\eta\,  dy \Bigg] \\
 & + & \mathcal{O}\left(a \, \vert \log(a) \vert^{2h-1}  \;  \Re \,\Bigg[ \tau \, \int_{\Omega \setminus D} \, f \; \; \overline{u}_{0}(y) \, \int_{D}  G_{k}(\eta,y) \, e_{n_{0}}(\eta) \, d\eta\,  dy \Bigg]\right) + \mathcal{O}\big( \vert \log(a) \vert^{-1} \big), 
\end{eqnarray*}
the same technique, as previously again, see $(\ref{mla})$, allows us  to deduce that 
\begin{equation*}
a \, \vert \log(a) \vert^{2h-1}  \;  \Re \,\Bigg[ \tau \, \int_{\Omega \setminus D} \, f \; \; \overline{u}_{0}(y) \, \int_{D}  G_{k}(\eta,y) \, e_{n_{0}}(\eta) \, d\eta\,  dy \Bigg]  = \mathcal{O}\big( \vert \log(a) \vert^{2h-2} \big).
\end{equation*}
The last step is to use Taylor expansion to write $<u_{0};e_{n_{0}}>$ on function of the center $z$. We have    
\begin{eqnarray*}
 S_{4}  & = &  2 \, \omega^{2} \mu_{0} \; \int_{D} e_{n_{0}} \, dx \, \Re \,\Bigg[ \frac{u_{0}(z)}{[1-\omega^{2}\, \mu_{0} \, \tau \, \lambda_{n_{0}}]} \, \tau \, \int_{\Omega \setminus D} \, \; f \; \; \overline{u}_{0}(y) \, \int_{D}  G_{k}(\eta,y) \, e_{n_{0}}(\eta) \, d\eta\,  dy \Bigg] \\
& + & \mathcal{O}\left( \frac{\int_{D} \int_{0}^{1} (z-\eta)\centerdot \nabla u_{0}(z+s(\eta - z)) ds \, e_{n_{0}}(\eta) d\eta}{[1-\omega^{2}\, \mu_{0} \, \tau \, \lambda_{n_{0}}]} \, \tau \, \int_{\Omega \setminus D} \, \;  f \; \overline{u}_{0}(y)   \; \int_{D}  G_{k}(\eta,y) \, e_{n_{0}}(\eta) \, d\eta\,  dy \right) \\
 & + & \mathcal{O}\Big(\vert \log(a) \vert^{2h-2} \Big) + \mathcal{O}\Big(\vert \log(a) \vert^{-1} \Big),
\end{eqnarray*}
then we compute an estimation of the remainder term from Taylor expansion. More precisely, we have
\begin{eqnarray*}
\vert \mathcal{O}(\cdots) \vert & \lesssim &  a \, \vert \log(a) \vert^{h} \, \Bigg\vert \int_{D} e_{n_{0}} \, dx \Bigg\vert  \, \vert \tau \vert \, \Bigg\vert  \int_{\Omega \setminus D} \, \;  f \; \overline{u}_{0}(y)   \; \int_{D}  G_{k}(\eta,y) \, e_{n_{0}}(\eta) \, d\eta\,  dy \Bigg\vert \\
& \leq & \vert \log(a) \vert^{h-1} \;  \int_{D} \; \int_{\Omega \setminus D} \,  \big\vert  f \; \overline{u}_{0}(y)   \;   G_{k}(\eta,y)  \big\vert \, dy \, \big\vert e_{n_{0}}(\eta)  \big\vert \, d\eta\,   = \mathcal{O}\big(a \, \vert \log(a) \vert^{h-1}\big).
\end{eqnarray*}
Finally,
\begin{eqnarray*}
S_{4} &=&  2 \, \omega^{2} \mu_{0} \; \int_{D} e_{n_{0}} \, dx \, \Re \,\Bigg[ \frac{u_{0}(z)}{[1-\omega^{2}\, \mu_{0} \, \tau \, \lambda_{n_{0}}]} \, \tau \, \int_{\Omega \setminus D} \, \partial_{t} \, \frac{\, \Im(\varepsilon_{0})(y) \; \overline{u}_{0}(y) \,}{\sqrt{t^{2}-\vert x-y \vert^{2}}}  \int_{D}  G_{k}(\eta,y) \, e_{n_{0}}(\eta) \, d\eta\,  dy \Bigg] \\ 
&+& \mathcal{O}\Big(\vert \log(a) \vert^{-1}\Big) + \mathcal{O}\Big(\vert \log(a) \vert^{2h-2}\Big).
\end{eqnarray*}
Hence
\begin{eqnarray*}
T_{4} &=& p_{0}(t,x) + S_{3} + S_{4} \\ 
      &=& p_{0}(t,x) + 2 \, \omega^{2} \mu_{0} \; \int_{D} e_{n_{0}} \, dx \, \Re \,\Bigg[ \frac{u_{0}(z)}{[1-\omega^{2}\, \mu_{0} \, \tau \, \lambda_{n_{0}}]} \, \tau \, \int_{\Omega \setminus D} \, \partial_{t} \, \frac{\, \Im(\varepsilon_{0})(y) \; \overline{u}_{0}(y) \,}{\sqrt{t^{2}-\vert x-y \vert^{2}}}  \int_{D}  G_{k}(\eta,y) \, e_{n_{0}}(\eta) \, d\eta\,  dy \Bigg] \\ &+& \mathcal{O}\Big(\vert \log(a) \vert^{-1}\Big) + \mathcal{O}\Big(\vert \log(a) \vert^{2h-2}\Big).
\end{eqnarray*}
The equation $(\ref{P(t,x)})$ takes the form
\begin{eqnarray}\label{(p-p0)(t,x)} 
\nonumber
(p - p_{0})(t,x) &=& 2 \, \omega^{2} \mu_{0} \; \int_{D} e_{n_{0}} \, dx \, \Re \,\Bigg[ \frac{u_{0}(z)}{[1-\omega^{2}\, \mu_{0} \, \tau \, \lambda_{n_{0}}]} \, \tau \, \int_{\Omega \setminus D} \, \partial_{t} \, \frac{\, \Im(\varepsilon_{0})(y) \; \overline{u}_{0}(y) \,}{\sqrt{t^{2}-\vert x-y \vert^{2}}}  \int_{D}  G_{k}(\eta,y) \, e_{n_{0}}(\eta) \, d\eta\,  dy \Bigg] \\
&+&  \frac{-t \, \Im(\tau)}{(t^{2}-\vert x-z \vert^{2})^{3/2}}  \, \frac{\vert <u_{0};e_{n_{0}}> \vert^{2}}{\vert 1- \omega^{2} \mu_{0} \lambda_{n_{0}} \tau \vert^{2}} \, 
+ \mathcal{O}( \vert \log(a) \vert^{\max(2h-2,-1)}) + \mathcal{O}(\Im(\tau) \, a^{2} \, \vert \log(a) \vert^{\max(0,3h-1)}).
\end{eqnarray}
Recall that we take, 
\begin{equation*}
\Im(\tau) = \frac{1}{a^{2} \, \vert \log(a) \vert^{1+h+s}}
\end{equation*}
with 
\begin{equation}\label{cdt1s}
0 \leq s < \min(h,1-h).
\end{equation}
With this choice, the error part of $(\ref{(p-p0)(t,x)})$ will be of order  
$\mathcal{O}\big(\vert \log(a) \vert^{\max(-1,2h-2)}\big).$
Hence
\begin{eqnarray*}
(p - p_{0})(t,x) & = & 2 \omega^{2} \mu_{0}  \int_{D} e_{n_{0}}  dx  \Re \Bigg[ \frac{u_{0}(z)}{[1-\omega^{2} \mu_{0} \tau  \lambda_{n_{0}}]}  \tau  \int_{\Omega \setminus D}  \partial_{t}  \frac{\Im(\varepsilon_{0})(y)  \overline{u}_{0}(y)}{\sqrt{t^{2}-\vert x-y \vert^{2}}}  \int_{D}  G_{k}(\eta,y) e_{n_{0}}(\eta)  d\eta  dy \Bigg] \\
&+& \frac{-t \, \Im(\tau)}{(t^{2}-\vert x-z \vert^{2})^{3/2}} \, \frac{1}{\vert 1- \omega^{2} \, \mu_{0} \, \lambda_{n_{0}} \, \tau \vert^{2}} \, \vert <u_{0},e_{n_{0}}> \vert^{2} + \mathcal{O}\big(\vert \log(a) \vert^{\max(-1,2h-2)}\big).
\end{eqnarray*}
Using again the estimate
\begin{equation*}
\vert <u_{0},e_{n_{0}}> \vert^{2} =  \vert u_{0}(z) \vert^{2} \, \Big(\int_{D} e_{n_{0}} dx \Big)^{2} + \mathcal{O}(a^{3}),
\end{equation*}
we get 
\begin{eqnarray*}
\big(p - p_{0}\big)(t,x) & = & 2  \omega^{2} \mu_{0}  \int_{D} e_{n_{0}}  dx  \Re \Bigg[ \frac{u_{0}(z) \tau}{[1-\omega^{2} \mu_{0}  \tau \lambda_{n_{0}}]}   \int_{\Omega \setminus D}  \partial_{t} \frac{\Im(\varepsilon_{0})(y)  \overline{u}_{0}(y) }{\sqrt{t^{2}-\vert x-y \vert^{2}}}  \int_{D}  G_{k}(\eta,y)  e_{n_{0}}(\eta)  d\eta dy \Bigg] \\
&+& \frac{-t \, \Im(\tau)}{(t^{2}-\vert x-z \vert^{2})^{3/2}} \, \frac{\vert u_{0}(z) \vert^{2}}{\vert 1- \omega^{2} \, \mu_{0} \, \lambda_{n_{0}} \, \tau \vert^{2}} \,  \, \Big(\int_{D} e_{n_{0}} dx \Big)^{2} + \mathcal{O}\big(\vert \log(a) \vert^{\max(-1,2h-2)}\big).
\end{eqnarray*}
Now, if we take two frequencies $\omega^{2}_{\pm}$, such that $\omega^{2}_{\pm} = \omega^{2}_{n_{0}} \pm \vert \log(a) \vert^{-h}$, we obtain 
\begin{eqnarray*}
\big(p^{\pm} - p_{0}\big)(t,x) & = & 2 \, \omega_{\pm}^{2} \mu_{0} \; \int_{D} e_{n_{0}} \, dx \, \Re \,\Bigg[ \frac{u_{0}(z)}{[1-\omega_{\pm}^{2}\, \mu_{0} \, \tau \, \lambda_{n_{0}}]} \, \tau \, \int_{\Omega \setminus D} \, \partial_{t} \, \frac{\, \Im(\varepsilon_{0})(y) \; \overline{u}_{0}(y) \,}{\sqrt{t^{2}-\vert x-y \vert^{2}}}  \int_{D}  G_{k}(\eta,y) \, e_{n_{0}}(\eta) \, d\eta\,  dy \Bigg] \\
&+& \frac{-t \, \Im(\tau)}{(t^{2}-\vert x-z \vert^{2})^{3/2}} \, \frac{1}{\vert 1- \omega^{2}_{\pm} \, \mu_{0} \, \lambda_{n_{0}} \, \tau \vert^{2}} \, \vert u_{0}(z) \vert^{2} \, \Big(\int_{D} e_{n_{0}} dx \Big)^{2} + \mathcal{O}(\vert \log(a) \vert^{\max(-1,2h-2)}).
\end{eqnarray*}
We use $ 1-\omega^{2}_{n_{0}} \, \mu_{0} \lambda_{n_{0}} \, \tau = 0$ to deduce that $\vert 1-\omega^{2}_{\pm} \, \mu_{0} \lambda_{n_{0}} \, \tau \vert = \mathcal{O}( \vert \log(a) \vert^{-h})$. \\ 
After some simplifications we get
\begin{eqnarray*}
\big(p^{\pm} - p_{0}\big)(t,x) & = & 2 \, \omega_{n_{0}}^{2} \mu_{0} \; \int_{D} e_{n_{0}} \, dx \, \Re \,\Bigg[ \frac{u_{0}(z)}{[1-\omega_{\pm}^{2}\, \mu_{0} \, \tau \, \lambda_{n_{0}}]} \, \tau \, \int_{\Omega \setminus D} \, \partial_{t} \, \frac{\, \Im(\varepsilon_{0})(y) \; \overline{u}_{0}(y) \,}{\sqrt{t^{2}-\vert x-y \vert^{2}}}  \int_{D}  G_{k}(\eta,y) \, e_{n_{0}}(\eta) \, d\eta\,  dy \Bigg] \\
&+& \frac{-t \, \Im(\tau)}{(t^{2}-\vert x-z \vert^{2})^{3/2}} \, \frac{1}{\vert 1- \omega^{2} \, \mu_{0} \, \lambda_{n_{0}} \, \tau \vert^{2}}  \,  \vert u_{0}(z) \vert^{2} \, \Big(\int_{D} e_{n_{0}} dx \Big)^{2} + \mathcal{O}(\vert \log(a) \vert^{\max(-1,2h-2)}).
\end{eqnarray*}
Next, 
\begin{eqnarray*}
(p^{+} + p^{-} - 2p_{0})(t,x) & = & 4 \, \omega^{2}_{n_{0}} \mu_{0} \,  \int_{D} e_{n_{0}} dx\; \Re \,\Bigg[\frac{u_{0}(z) \; \tau \; (1-\omega^{2}_{n_{0}} \, \mu_{0} \lambda_{n_{0}} \, \tau)}{(1-\omega^{2}_{+} \, \mu_{0} \lambda_{n_{0}} \, \tau)\, (1-\omega^{2}_{-} \, \mu_{0} \lambda_{n_{0}} \, \tau)} \\&&  \int_{\Omega \setminus D} \, \partial_{t} \, \frac{ \Im(\varepsilon_{0})(y) \; \overline{u}_{0}(y)}{\sqrt{t^{2}-\vert x-y \vert^{2}}}  \int_{D}  G_{k}(\eta,y) \, e_{n_{0}}(\eta) \, d\eta\,  dy \Bigg] \\
& + & \frac{-2 \, t \, \Im(\tau)}{(t^{2}-\vert x-z \vert^{2})^{3/2}} \,  \frac{\vert u_{0}(z) \vert^{2}}{\vert 1- \omega^{2} \, \mu_{0} \, \lambda_{n_{0}} \, \tau \vert^{2}} \,   \, \Big(\int_{D} e_{n_{0}} dx \Big)^{2} +  \mathcal{O}(\vert \log(a) \vert^{\max(-1,2h-2)}),
\end{eqnarray*}
thanks to $(\ref{exactMieresonance})$, we know that $(1-\omega^{2}_{n_{0}} \, \mu_{0} \lambda_{n_{0}} \, \tau)=0$ , then the right term of this equation will be reduced to only the dominant term. Finally, we obtain 
\begin{equation}\label{valueofV}
(p^{+} + p^{-} - 2p_{0})(t,x)  =  \frac{-t \, }{(t^{2}-\vert x-z \vert^{2})^{3/2}} \, \, \frac{2 \; \Im(\tau) \; \, \vert u_{0}(z) \vert^{2}}{\vert 1- \omega^{2} \, \mu_{0} \, \lambda_{n_{0}} \, \tau \vert^{2}}  \Big(\int_{D} e_{n_{0}} dx \Big)^{2} + \mathcal{O}\big(\vert \log(a) \vert^{\max(-1,2h-2)}\big),
\end{equation}
or, with help of $(\ref{ve=Ve})$, 
\begin{equation}\label{abxyz}
(p^{+} + p^{-} - 2p_{0})(t,x)  =  \frac{-2 \, t \; \Im(\tau) \; \vert <u_{1};e_{n_{0}}> \vert^{2}}{(t^{2}-\vert x-z \vert^{2})^{3/2}}   + \mathcal{O}\big(\vert \log(a) \vert^{\max(-1,2h-2)}\big).
\end{equation}

\subsection{Photo-acoustic imaging using two close particles (Dimers)} 
Proof of $(\ref{pressure-tilde-expansion})$
\ 
\\To avoid using, in the proof, more notations we keep the same ones as in the case of one particle whenever this is possible. 
\begin{lemma}\label{lemmad1d2} 
We have
\begin{equation}\label{xd1d2}
u_{2}(x) =  \mathcal{O}(1) + \mathcal{O}(\vert \log(a) \vert^{h-1} \; \vert \log(dist(x,D_{1} \cup D_{2})) \vert), \qquad \qquad x \notin D_{1} \cup D_{2}. 
\end{equation}
\end{lemma}
\begin{proof}
We skip the proof since it is similar to that of one particle (see the proof of Lemma $\ref{awayD}$).  
\end{proof}
Now, from Poisson's formula, the solution can be written as
\begin{eqnarray*}
p(t,x) & = & \sum_{i=1}^{2} \partial_{t} \int_{\vert x-y \vert < t} \frac{(\Im(\varepsilon_{p}) \vert u_{2} \vert^{2})(y)}{\sqrt{t^{2}-\vert x-y \vert^{2}}} \chi_{D_{i}} \, dy + \partial_{t} \int_{\vert x-y \vert < t} \frac{(\Im(\varepsilon_{0}) \vert u_{2} \vert^{2})(y)}{\sqrt{t^{2}-\vert x-y \vert^{2}}} \chi_{\Omega \setminus D} \, dy \\
& = & \sum_{i=1}^{2} \partial_{t} \int_{\vert x-y \vert < t} \frac{(\Im(\varepsilon_{p}-\varepsilon_{0}) \vert u_{2} \vert^{2})(y)}{\sqrt{t^{2}-\vert x-y \vert^{2}}} \chi_{D_{i}} \, dy + \partial_{t} \int_{\vert x-y \vert < t} \frac{(\Im(\varepsilon_{0}) \vert u_{2} \vert^{2})(y)}{\sqrt{t^{2}-\vert x-y \vert^{2}}} \chi_{\Omega} \, dy. 
\end{eqnarray*}
For $t > diam (\Omega)$ we have 
\begin{equation*}
p(t,x)  =  \sum_{i=1}^{2} \partial_{t} \int_{D_{i}} \frac{(\Im(\varepsilon_{p}-\varepsilon_{0}) \vert u_{2} \vert^{2})(y)}{\sqrt{t^{2}-\vert x-y \vert^{2}}}  \, dy + \partial_{t} \int_{\Omega} \frac{(\Im(\varepsilon_{0}) \vert u_{2} \vert^{2})(y)}{\sqrt{t^{2}-\vert x-y \vert^{2}}} \, dy. 
\end{equation*}
As before set 
\begin{equation*}
T_{4}^{\star} := \partial_{t} \int_{\Omega} \frac{(\Im(\varepsilon_{0}) \vert u_{2} \vert^{2})(y)}{\sqrt{t^{2}-\vert x-y \vert^{2}}} \, dy.
\end{equation*}
Next, we assume that $\tau_{1} = \tau_{2} = \tau$ and we use Taylor expansion of $(\varepsilon_{p}-\varepsilon_{0})(\centerdot)$ and $\big(t^{2} - \vert x - \centerdot \vert^{2} \big)^{-3/2}$ near $z_{1,2}$ to obtain
\begin{equation*}
p(t,x)  =  -t \, \Im(\tau) \; \sum_{i=1}^{2}  \int_{D_{i}} \frac{ \vert u_{2} \vert^{2}(y)}{\big(t^{2}-\vert x-y \vert^{2}\big)^{3/2}}  \, dy +T_{4}^{\star} + \mathcal{O}\left(\sum_{i=1}^{2} \int_{D_{i}} \frac{ \int_{0}^{1} (y-z_{i})\centerdot \nabla \varepsilon_{0}(z_{i}+t(y-z_{i})) \, dt \, \vert u_{2} \vert^{2}(y)}{\big(t^{2}-\vert x-y \vert^{2}\big)^{3/2}}  \, dy \right).
\end{equation*}
The remainder term, as done in $(\ref{lb})$, is of order $\mathcal{O}\big(a^{3} \, \vert \log(a) \vert^{2h} \big)$. Then, as in the case of one particle, we have  
\begin{eqnarray*}
p(t,x) & = & \sum_{i=1}^{2} \frac{-t \Im(\tau)}{\big(t^{2}-\vert x-z_{i} \vert^{2}\big)^{3/2}}   \int_{D_{i}}  \vert u_{2} \vert^{2}  dy +  T_{4}^{\star} \\
& +& \mathcal{O}\left( \sum_{i=1}^{2} \Im(\tau)  \int_{D_{i}} (\vert y - z_{i} \vert^{2} + 2 <x-z_{i};z_{i}-y> )  \vert u_{2} \vert^{2}(y) dy  \right) + \mathcal{O}\big(a^{3} \, \vert \log(a) \vert^{2h} \big).
\end{eqnarray*}
We deduce as in $(\ref{hdd})$ that the remainder term can be estimated as $\mathcal{O}\big( \Im(\tau) \, a^{3} \, \vert \log(a) \vert^{2h} \big)$. Next, we develop $u_{2}$ over the basis  and we use $(\ref{oyaya})$ to estimate the remainder term to obtain 
\begin{eqnarray*}
p(t,x) & = & -t \Im(\tau) \sum_{i=1}^{2} \frac{\vert <u_{2};e^{(i)}_{n_{0}}> \vert^{2}}{\big(t^{2}-\vert x-z_{i} \vert^{2}\big)^{3/2}} +  T_{4}^{\star} + \mathcal{O}\left(\Im(\tau) \sum_{i=1 \atop n \neq n_{0}}^{2}  \vert <u_{2};e^{(i)}_{n}> \vert^{2}  \right) + \mathcal{O}\big(\Im(\tau) \, a^{3} \, \vert \log(a) \vert^{2h} \big).
\end{eqnarray*}
Then we get
\begin{equation}\label{T4p}
p(t,x)  =  -t \, \Im(\tau) \, \sum_{i=1}^{2} \, \frac{\vert <u_{2};e^{(i)}_{n_{0}}> \vert^{2}}{\big(t^{2}-\vert x-z_{i} \vert^{2}\big)^{3/2}} + T^{\star}_{4} + \mathcal{O}\big(\Im(\tau) \, a^{2} \big).
\end{equation}
Set $\Omega_{1,2} := \Omega \setminus \big( D_{1} \cup D_{2} \big)$ and write $T^{\star}_{4}$  as: 
\begin{equation*}
T^{\star}_{4}  =  \int_{\Omega} (\Im(\varepsilon_{0}) \vert u_{2} \vert^{2})(y) \; \partial_{t} \, \frac{1}{\sqrt{t^{2}-\vert x-y \vert^{2}}} \, dy  =  \int_{\Omega_{1,2}}  \vert u_{2} \vert^{2} \; f \; dy + \int_{D_{1} \cup D_{2}}  \vert u_{2} \vert^{2} \; f \; dy.
\end{equation*}
From the a priori estimate, see $(\ref{apmp})$, and lemma  $(\ref{lemmad1d2})$ we deduce that the first integral dominates the second one. Now, since $f$ is smooth, the a priori estimate allows to estimate the integral over $D_{1} \cup D_{2}$ as follows
\begin{equation*}
\Big\vert \int_{D_{1} \cup D_{2}}  \vert u_{2} \vert^{2} \; f \; dy \Big\vert \lesssim \Vert u_{2} \Vert^{2} = \mathcal{O}\big(a^{2} \vert \log(a) \vert^{2h}\big).
\end{equation*}
Then we use L.S.E to obtain 
\begin{eqnarray*}
T^{\star}_{4} & = & \int_{\Omega}  \, f \, \vert u_{0} \vert^{2} dy  - \int_{D_{1} \cup D_{2}}  \, f \, \vert u_{0} \vert^{2} dy +  2 \,\omega^{2} \, \mu_{0} \, \sum_{i=1}^{2} \,\Re\Bigg[ \int_{\Omega_{1,2}}  \; f \; \overline{u}_{0}(y) \, \int_{D_{i}} (\varepsilon_{p}-\varepsilon_{0})(\eta) \, G_{k}(\eta,y) \, u_{2}(\eta) \,  d\eta \, d y \Bigg] \\
& + & ( \omega^{2} \, \mu_{0} )^{2} \, \sum_{i=1}^{2} \, \int_{\Omega_{1,2}} f \,\Big\vert \int_{D_{i}} (\varepsilon_{p}-\varepsilon_{0})(\eta) \, G_{k}(\eta,y) \, u_{2}(\eta) \,  d\eta \Big\vert^{2} \, dy. \\
& + & 2 \, \Re\Bigg[  \int_{\Omega_{1,2}} f \; \int_{D_{1}} \overline{(\varepsilon_{p}-\varepsilon_{0})(\eta) \, G_{k}(\eta,y) \, u_{2}(\eta)} \,  d\eta \; \int_{D_{2}} (\varepsilon_{p}-\varepsilon_{0})(\eta) \, G_{k}(\eta,y) \, u_{2}(\eta) \,  d\eta  dy \Bigg] + \mathcal{O}\big(a^{2} \vert \log(a) \vert^{2h}\big).
\end{eqnarray*}
Clearly, by the smoothness of $f$ and $\vert u_{0} \vert$, we have
\begin{equation*}
\Big\vert \int_{D_{1} \cup D_{2}}  \, f \, \vert u_{0} \vert^{2} dy \Big\vert = \mathcal{O}\big( a^{2} \big).
\end{equation*}
Then, we obtain\footnote{For the definition of $p_{0}(t,x)$, see $(\ref{p0np})$.}
\begin{eqnarray*}
T^{\star}_{4} & = & p_{0}(t,x) +  2 \,\omega^{2} \, \mu_{0} \, \sum_{i=1}^{2} \,\Re\Bigg[ \int_{\Omega_{1,2}} f \; \overline{u}_{0}(y) \, \int_{D_{i}} (\varepsilon_{p}-\varepsilon_{0})(\eta) \, G_{k}(\eta,y) \, u_{2}(\eta) \,  d\eta \, d y \Bigg]\\
& + & 2 \, \Re\Bigg[  \int_{\Omega_{1,2}} f \; \int_{D_{1}} \overline{(\varepsilon_{p}-\varepsilon_{0})(\eta) \, G_{k}(\eta,y) \, u_{2}(\eta)} \,  d\eta \; \int_{D_{2}} (\varepsilon_{p}-\varepsilon_{0})(\eta) \, G_{k}(\eta,y) \, u_{2}(\eta) \,  d\eta  dy \Bigg] \\ 
&+& (\omega^{2} \, \mu_{0} )^{2} \, \sum_{i=1}^{2} \, \int_{\Omega_{1,2}} f\, \,\Big\vert \int_{D_{i}} (\varepsilon_{p}-\varepsilon_{0})(\eta) \, G_{k}(\eta,y) \, u_{2}(\eta) \,  d\eta \Big\vert^{2} \, dy + \mathcal{O}\big( a^{2} \, \vert \log(a) \vert^{2h} \big).
\end{eqnarray*}
We remark that 
\begin{equation*}
 (\omega^{2} \, \mu_{0} )^{2} \, \, \int_{\Omega_{1,2}} f\, \,\Big\vert \int_{D_{i}} (\varepsilon_{p}-\varepsilon_{0})(\eta) \, G_{k}(\eta,y) \, u_{2}(\eta) \,  d\eta \Big\vert^{2} \, dy \quad \text{for} \; i=1,2
\end{equation*}
have the same expression as $S_{3}$ given in section $(\ref{opsubsection})$ (more exactly see $\ref{S3}$). Then we estimate it as $\mathcal{O}\big( \vert \log(a) \vert^{2h-2} \big)$. 
Similarly, regardless of whether the position of $y$ is in $D_{1}$ or $D_{2}$, the same estimation holds for 
\begin{equation*}
2 \, \Re\Bigg[  \int_{\Omega_{1,2}} f \; \int_{D_{1}} \overline{(\varepsilon_{p}-\varepsilon_{0})(\eta) \, G_{k}(\eta,y) \, u_{2}(\eta)} \,  d\eta \; \int_{D_{2}} (\varepsilon_{p}-\varepsilon_{0})(\eta) \, G_{k}(\eta,y) \, u_{2}(\eta) \,  d\eta  dy \Bigg]
\end{equation*}  
We synthesize the above to get  
\begin{equation*}
T^{\star}_{4}  =  p_{0}(t,x) +  2 \,\omega^{2} \, \mu_{0} \, \sum_{i=1}^{2} \,\Re\Bigg[ \int_{\Omega_{1,2}} f \; \overline{u}_{0}(y) \, \int_{D_{i}} (\varepsilon_{p}-\varepsilon_{0})(\eta) \, G_{k}(\eta,y) \, u_{2}(\eta) \,  d\eta \, d y \Bigg]\\
 +   \mathcal{O}\big(  \vert \log(a) \vert^{2h-2} \big).
\end{equation*}
Next, we develop $u_{2}$ over the basis and use the Taylor expansion of $(\varepsilon_{p}-\varepsilon)(\centerdot)$ to obtain  
\begin{eqnarray*}
T^{\star}_{4} & = & p_{0}(t,x)  +  2 \,\omega^{2} \, \mu_{0} \, \sum_{i=1}^{2} \,\Re\Bigg[\tau \, <u_{2};e^{(i)}_{n_{0}}> \int_{\Omega_{1,2}} f \; \overline{u}_{0}(y) \, \int_{D_{i}} \, G_{k}(\eta,y) \, e^{(i)}_{n_{0}}(\eta) \,  d\eta \, d y \Bigg] \\
& - & 2 \,\omega^{2} \, \mu_{0} \, \sum_{i=1}^{2} \,\Re\Bigg[ \int_{\Omega_{1,2}}  \; f \; \overline{u}_{0}(y) \, \int_{D_{i}} \, \int_{0}^{1} (z_{i}- \eta) \centerdot \nabla \varepsilon_{0}(z_{i}+s(\eta - z_{i})) \, ds \, G_{k}(\eta,y) \, u_{2}(\eta) \,  d\eta \, d y \Bigg] \\
& + & 2 \,\omega^{2} \, \mu_{0}  \, \sum_{i=1,2 \atop n \neq n_{0}} \Re\Bigg[\tau \, <u_{2};e^{(i)}_{n}> \int_{\Omega_{1,2}} f \; \overline{u}_{0}(y) \, \int_{D_{i}} \, G_{k}(\eta,y) \, e^{(i)}_{n}(\eta) \,  d\eta \, d y \Bigg] + \mathcal{O}\big( \vert \log(a) \vert^{2h-2} \big) 
\end{eqnarray*}
To precise the value of the error we need to estimate 
\begin{eqnarray*}
&& \Bigg\vert \sum_{i=1}^{2} \,\Re\Bigg[ \int_{\Omega_{1,2}}  \; f \; \overline{u}_{0}(y) \, \int_{D_{i}} \, \int_{0}^{1} (z_{i}- \eta) \centerdot \nabla \varepsilon_{0}(z_{i}+s(\eta - z_{i})) \, ds \, G_{k}(\eta,y) \, u_{2}(\eta) \,  d\eta \, d y \Bigg] \Bigg\vert \\
& \leq & a \, \sum_{i=1}^{2} \, \int_{\Omega_{1,2}} \, \Big\vert    \; f \; \overline{u}_{0}(y) \, \Big\vert \,\Bigg\vert \int_{D_{i}}  \, G_{k}(\eta,y) \, u_{2}(\eta) \,  d\eta \, \Bigg\vert dy  \lesssim   a \,\sum_{i=1}^{2} \, \Bigg( \int_{\Omega_{1,2}} \,  \Bigg\vert \int_{D_{i}}  \, G_{k}(\eta,y) \, u_{2}(\eta) \,  d\eta \, \Bigg\vert^{2} dy \Bigg)^{\frac{1}{2}} \\ 
& \lesssim & a \, \Vert u_{2} \Vert \, \sum_{i=1}^{2} \, \Bigg(   \int_{D_{i}} \int_{\Omega_{1,2}}  \, \vert G_{k}\vert^{2} (\eta,y)  \,  dy \,  d\eta \Bigg)^{\frac{1}{2}} = \mathcal{O}\big( a^{3} \, \vert \log(a) \vert^{h} \big)
\end{eqnarray*}
and 
\begin{eqnarray*}
&& \Bigg\vert \sum_{i=1,2 \atop n \neq n_{0}} \Re\Bigg[\tau \, <u_{2};e^{(i)}_{n}> \int_{\Omega_{1,2}} f \; \overline{u}_{0}(y) \, \int_{D_{i}} \, G_{k}(\eta,y) \, e^{(i)}_{n}(\eta) \,  d\eta \, d y \Bigg] \Bigg\vert \\
& \leq & \vert \tau \vert \; \Big( \sum_{i=1,2 \atop n \neq n_{0}} \,\big\vert <u_{2};e^{(i)}_{n}> \, \big\vert^{2} \Big)^{\frac{1}{2}} \, \Bigg( \int_{\Omega_{1,2}} \Big\vert f \; \overline{u}_{0}(y) \,\Big\vert^{2} dy \Bigg)^{\frac{1}{2}} \, \Bigg(   \int_{\Omega_{1,2}} \sum_{i=1,2 \atop n \neq n_{0}} \Big\vert \int_{D_{i}} \, G_{k}(\eta,y) \, e^{(i)}_{n}(\eta) \,  d\eta \, \Big\vert^{2} \, d y  \Bigg)^{\frac{1}{2}} \\ 
& \leq & \vert \tau \vert \; \Vert u_{0} \Vert  \, \Bigg( \int_{D} \,  \int_{\Omega_{1,2}} \vert G_{k}\vert^{2}(\eta,y)   dy  \, d\eta  \Bigg)^{\frac{1}{2}} = \mathcal{O}\big( \vert \log(a) \vert^{-1} \big)
\end{eqnarray*}
We keep the dominant term and sum the others as an error to obtain 
\begin{equation*}
T^{\star}_{4}  =  p_{0}(t,x) + 2 \,\omega^{2} \, \mu_{0} \, \sum_{i=1}^{2} \,\Re\Bigg[\tau \, <u_{2};e^{(i)}_{n_{0}}> \int_{\Omega_{1,2}} f \; \overline{u}_{0}(y) \, \int_{D_{i}} \, G_{k}(\eta,y) \, e^{(i)}_{n_{0}}(\eta) \,  d\eta \, d y \Bigg]  + \mathcal{O}\big(\vert \log(a) \vert^{\max(-1,2h-2)}\big).
\end{equation*}

Use ($\ref{v&V}$) to obtain
\begin{eqnarray}\label{ved53}
\nonumber
T^{\star}_{4} & = &  2 \,\omega^{2} \, \mu_{0} \, \sum_{i=1}^{2} \,\Re\Bigg[ \tau \; \frac{ <u_{0};e^{(i)}_{n_{0}}>}{det^{\star}} \int_{\Omega_{1,2}}  f \;\, \overline{u}_{0}(y)  \, \int_{D_{i}} \, G_{k}(\eta;y) \, e^{(i)}_{n_{0}}(\eta) \,  d\eta \, d y \Bigg] + p_{0}(t,x) \\
&+& \mathcal{O}\Bigg(\tau \, a \,  \sum_{i=1}^{2} \,  \int_{\Omega_{1,2}} f \; \, \overline{u}_{0}(y) \, \int_{D_{i}} \, G_{k}(\eta,y) \, e^{(i)}_{n_{0}}(\eta) \,  d\eta \, d y  \Bigg) + \mathcal{O}\big(\vert \log(a) \vert^{\max(-1,2h-2)}\big),
\end{eqnarray}
where 
\begin{equation*}
det^{\star} := (1-\omega^{2}\,\mu_{0}\,\tau\,\lambda_{n_{0}})-\omega^{2}\, \mu_{0}\, \tau \, a^{2} \, \Phi_{0}(z_{1};z_{2}) \, \Big( \int_{B} \overline{e}_{n_{0}} \Big)^{2}, 
\end{equation*}
and 
\begin{equation*}
\tau  a \sum_{i=1}^{2}  \int_{\Omega_{1,2}} f   \overline{u}_{0}(y) \int_{D_{i}} G_{k}(\eta,y)  e^{(i)}_{n_{0}}(\eta) d\eta  d y = \tau  a  \sum_{i=1}^{2}  \int_{D_{i}}  \int_{\Omega_{1,2}} f   \overline{u}_{0}(y) G_{k}(\eta,y)  dy  e^{(i)}_{n_{0}}(\eta)  d \eta = \mathcal{O}\big( \vert \log(a) \vert^{-1} \big). 
\end{equation*}
The last equality is justified by the fact that we integrate a smooth function over $\Omega_{1,2}$ and we know that the integral over $D$ of an eigenfunction is of the order  $a$. \\  
Also we can write $(\ref{ved53})$ as
\begin{eqnarray*}
T^{\star}_{4} & = & 2 \,\omega^{2} \, \mu_{0} \, \int_{D} e_{n_{0}} dx \sum_{i=1}^{2} \,\Re\Bigg[\tau \; \frac{ u_{0}(z_{i}) \; }{det^{\star}} \int_{\Omega_{1,2}}  \partial_{t} \frac{\Im(\varepsilon_{0})(y)\, \overline{u}_{0}(y)}{\sqrt{t^{2}-\vert x-y \vert^{2}}}  \, \int_{D_{i}} \, G_{k}(\eta;y) \, e^{(i)}_{n_{0}}(\eta) \,  d\eta \, d y \Bigg] + p_{0}(t,x)\\
&+& \mathcal{O}\Bigg( \sum_{i=1}^{2} \, \frac{\tau \; a^{2} }{det^{\star}} \int_{\Omega_{1,2}}  f \, \overline{u}_{0}(y) \, \int_{D_{i}} \, G_{k}(\eta,y) \, e^{(i)}_{n_{0}}(\eta) \,  d\eta \, d y  \Bigg) + \mathcal{O}\big(\vert \log(a) \vert^{\max(-1,2h-2)}\big).
\end{eqnarray*}
We have 
\begin{equation*}
\mathcal{O}\Bigg( \sum_{i=1}^{2} \, \frac{\tau \; a^{2} }{det^{\star}} \int_{\Omega_{1,2}}  f \, \overline{u}_{0}(y) \, \int_{D_{i}} \, G_{k}(\eta,y) \, e^{(i)}_{n_{0}}(\eta) \,  d\eta \, d y  \Bigg) = \mathcal{O}(a \, \vert \log(a) \vert^{h-1})
\end{equation*}
since, if we compare it with the error term given in equation $(\ref{ved53})$ we deduce that they are different by a term of order $a / det^{\star}$. Finally: 
\begin{eqnarray*}
T^{\star}_{4} & = &   2 \,\omega^{2} \, \mu_{0} \, \int_{D} e_{n_{0}} dx \sum_{i=1}^{2} \,\Re\Bigg[ \displaystyle\frac{\tau \; u_{0}(z_{i}) }{det^{\star}} \; \int_{\Omega_{1,2}}  \partial_{t} \frac{\Im(\varepsilon_{0})(y)\, \overline{u}_{0}(y)}{\sqrt{t^{2}-\vert x-y \vert^{2}}}  \, \int_{D_{i}} \, G_{k}(\eta,y) \, e^{(i)}_{n_{0}}(\eta) \,  d\eta \, d y \Bigg]  +  p_{0}(t,x) \\  &+&  \mathcal{O}\big(\vert \log(a) \vert^{\max(-1,2h-2)}\big).
\end{eqnarray*}
We set $I_{i}$ to be 
\begin{equation*}
I_{i} := \int_{\Omega_{1,2}}  \; \partial_{t} \, \frac{\Im(\varepsilon_{0})(y)\, \overline{u}_{0}(y)}{\sqrt{t^{2}-\vert x-y \vert^{2}}}  \int_{D_{i}} \, G_{k}(\eta,y) \, e^{(i)}_{n_{0}}(\eta) \,  d\eta \, d y,
\end{equation*}
and use the estimation of $T^{\star}_{4}$ in the equation $(\ref{T4p})$ to obtain: 
\begin{eqnarray*}
(p-p_{0})(t,x)  &=& 2 \,\omega^{2} \, \mu_{0} \,\,\Big( \int_{D} e_{n_{0}} dx \Big)\, \sum_{i=1}^{2} \Re\Bigg[ \frac{\tau \, u_{0}(z_{i})}{det^{\star}} I_{i} \Bigg] - t \, \Im(\tau) \, \sum_{i=1}^{2} \, \frac{\vert <u_{2};e^{(i)}_{n_{0}}> \vert^{2}}{\big(t^{2}-\vert x-z_{i} \vert^{2}\big)^{3/2}} \\
&+& \mathcal{O}\big(\vert \log(a) \vert^{\max(-1,2h-2)}\big).
\end{eqnarray*} 
We use the next lemma to simplify the expression of $p(t,x)$
\begin{lemma}\label{lemma35}
We have
\begin{equation}\label{v1n0=v2n0}
<u_{2};e^{(1)}_{n_{0}}> = <u_{2};e^{(2)}_{n_{0}}> + \mathcal{O}(a).
\end{equation}
\end{lemma}
\begin{proof}
Remember, from $(\ref{intv1intv2})$, that we have:
\begin{equation*}
\int_{D_{1}} u_{2} \, dx = \int_{D_{2}} u_{2} \, dx + \mathcal{O}(d \, a^{2} \, \vert \log(a) \vert^{h}).
\end{equation*}
Write each integral over the basis: 
\begin{equation*}
<u_{2},e^{(1)}_{n_{0}}> \, \int_{D_{1}} e_{n_{0}} \, dx = <u_{2},e^{(2)}_{n_{0}}> \, \int_{D_{1}} e_{n_{0}} \, dx + \mathcal{O}\Bigg( \sum_{i=1 \atop n \neq n_{0}}^{2}  <u_{2},e^{(i)}_{n}> \, \int_{D_{i}} e^{(i)}_{n} \, dx \Bigg) + \mathcal{O}(d \, a^{2} \, \vert \log(a) \vert^{h})
\end{equation*}
Clearly, by Holder inequality, we have 
\begin{equation*}
\Bigg\vert \sum_{i=1 \atop n \neq n_{0}}^{2}  <u_{2},e^{(i)}_{n}> \, \int_{D_{i}} e^{(i)}_{n} \, dx \Bigg\vert \lesssim \Vert u_{0} \Vert \; \Vert 1 \Vert = \mathcal{O}\big( a^{2} \big)
\end{equation*}
and it follows that
\begin{equation*}
<u_{2};e^{(1)}_{n_{0}}>   =   <u_{2};e^{(2)}_{n_{0}}> + \mathcal{O}(a).
\end{equation*}
From $(\ref{v1n0=v2n0})$ we deduce: 
\begin{equation*}
\vert <u_{2};e^{(1)}_{n_{0}}> \vert^{2} = \vert <u_{2};e^{(2)}_{n_{0}}> \vert^{2} + \mathcal{O}(a^{2} \; \vert \log(a) \vert^{h}).  
\end{equation*}
\end{proof}
By lemma $\ref{lemma35}$, we have 
\begin{eqnarray*}
\nonumber
\big(p - p_{0}\big)(t,x) &=& 2 \,\omega^{2} \, \mu_{0} \, \Big[ \int_{D_{1}} e_{n_{0}} \Big]\sum_{i=1}^{2} \, \Re\Bigg[ \frac{\tau \, u_{0}(z_{i})}{det^{\star}} I_{i} \Bigg]- t \, \Im(\tau) \, \vert <u_{2};e^{(2)}_{n_{0}}> \vert^{2} \, \sum_{i=1}^{2} \, \frac{1}{\big(t^{2}-\vert x-z_{i} \vert^{2}\big)^{3/2}} \\ 
&+& \mathcal{O}\big(\vert \log(a) \vert^{\max(-1,2h-2)}\big).
\end{eqnarray*}
We have also: 
\begin{equation*}
\frac{1}{\big(t^{2}-\vert x-z_{1} \vert^{2}\big)^{3/2}} = \frac{1}{\big(t^{2}-\vert x-z_{2} \vert^{2}\big)^{3/2}} \Big( 1 + \mathcal{O}(d) \Big).
\end{equation*}
Then
\begin{equation*}
(p-p_{0})(t,x) = 2 \,\omega^{2} \, \mu_{0} \, \Big[ \int_{D_{1}} e_{n_{0}} \Big]\sum_{i=1}^{2} \, \Re\Bigg[ \frac{\tau \, u_{0}(z_{i})}{det^{\star}}  I_{i} \Bigg]  - 2 t \, \Im(\tau) \,  \, \frac{\vert <u_{2};e^{(2)}_{n_{0}}> \vert^{2}}{\big(t^{2}-\vert x-z_{2} \vert^{2}\big)^{3/2}} 
+ \mathcal{O}\big(\vert \log(a) \vert^{\max(-1,2h-2)}\big).
\end{equation*}
Next, we use the same technique as before by taking two frequencies $\omega_{\pm}^{2} = \omega_{n_{0}}^{2} \pm \vert \log(a) \vert^{-h}$, we get
\begin{eqnarray*}
(p^{\pm}-p_{0})(t,x) &=&  2 \,\omega_{n_{0}}^{2} \, \mu_{0} \, \Big[ \int_{D_{1}} e_{n_{0}} \Big]\sum_{i=1}^{2} \, \Re\Bigg[ \frac{\tau \, u_{0}(z_{i}) \; \; I_{i}}{(1-\omega_{\pm}^{2} \, \mu_{0} \tau \lambda_{n_{0}})-\omega_{\pm}^{2} \, \mu_{0} \tau a^{2} \Phi_{0} (\int_{B} \overline{e}_{n_{0}})^{2}}  \Bigg]  \\
&+& \mathcal{O}\Bigg( \vert \log(a) \vert^{-h} \, \Big[ \int_{D_{1}} e_{n_{0}} \Big]\sum_{i=1}^{2} \, \Re\Bigg[ \frac{\tau \, u_{0}(z_{i}) \; \; I_{i}}{(1-\omega_{\pm}^{2} \, \mu_{0} \tau \lambda_{n_{0}})-\omega_{\pm}^{2} \, \mu_{0} \tau a^{2} \Phi_{0} (\int_{B} \overline{e}_{n_{0}})^{2}}  \Bigg] \Bigg) \\
 &-&  2 t \, \Im(\tau) \, \vert <u_{2};e^{(2)}_{n_{0}}> \vert^{2}  \, \frac{1}{\big(t^{2}-\vert x-z_{2} \vert^{2}\big)^{3/2}} + \mathcal{O}\big(\vert \log(a) \vert^{\max(-1,2h-2)}\big).
\end{eqnarray*}
We estimate the error part as 
\begin{equation*}
 \vert \log(a) \vert^{-h} \, \Big[ \int_{D_{1}} e_{n_{0}} \Big]\sum_{i=1}^{2} \, \Re\Bigg[ \frac{\tau \, u_{0}(z_{i}) \; \; I_{i}}{(1-\omega_{\pm}^{2} \, \mu_{0} \tau \lambda_{n_{0}})-\omega_{\pm}^{2} \, \mu_{0} \tau a^{2} \Phi_{0} (\int_{B} \overline{e}_{n_{0}})^{2}}  \Bigg] \sim {\mathcal{O}(\vert \log(a) \vert^{-1})}.
\end{equation*}
Define $\tilde{p}(t,x)$ as
\begin{equation*}
\tilde{p}(t,x) := (p^{+}-p_{0})(t,x) + \frac{1-\omega_{n_{0}}^{2}}{1+\omega_{n_{0}}^{2}} (p^{-}-p_{0})(t,x), 
\end{equation*}
hence
\begin{eqnarray*}
\tilde{p}(t,x) &=&   \frac{-4}{1+\omega_{n_{0}}^{2}} t \, \Im(\tau) \, \frac{\vert <u_{2};e^{(2)}_{n_{0}}> \vert^{2}}{\big(t^{2}-\vert x-z_{2} \vert^{2}\big)^{3/2}}  + 2 \,\omega_{n_{0}}^{2} \, \mu_{0} \, \Big[ \int_{D_{1}} e_{n_{0}} \Big]\sum_{i=1}^{2} \, \Re\Bigg[\tau \, u_{0}(z_{i}) \; I_{i} \\&&  \Bigg( \frac{1}{(1-\omega_{+}^{2} \, \mu_{0} \tau \lambda_{n_{0}})-\omega_{+}^{2} \, \mu_{0} \tau a^{2} \Phi_{0} (\int_{B} \overline{e}_{n_{0}})^{2}} + \frac{1-\omega_{n_{0}}^{2}}{1+\omega_{n_{0}}^{2}} \frac{1}{(1-\omega_{-}^{2} \, \mu_{0} \tau \lambda_{n_{0}})-\omega_{-}^{2} \, \mu_{0} \tau a^{2} \Phi_{0} (\int_{B} \overline{e}_{n_{0}})^{2}}  \Bigg) \Bigg]  \\
&+&  \mathcal{O}\big(\vert \log(a) \vert^{\max(-1,2h-2)}\big).
\end{eqnarray*}
We compute the following quantity
\begin{eqnarray*}
J &:=&  \frac{1}{(1-\omega_{+}^{2} \, \mu_{0} \tau \lambda_{n_{0}})-\omega_{+}^{2} \, \mu_{0} \tau a^{2} \Phi_{0} (\int_{B} \overline{e}_{n_{0}})^{2}} + \frac{1-\omega_{n_{0}}^{2}}{1+\omega_{n_{0}}^{2}} \frac{1}{(1-\omega_{-}^{2} \, \mu_{0} \tau \lambda_{n_{0}})-\omega_{-}^{2} \, \mu_{0} \tau a^{2} \Phi_{0} (\int_{B} \overline{e}_{n_{0}})^{2}} \\
&=& \frac{2 \, \omega_{n_{0}}^{2} \, \mu_{0} \, \tau \, \vert \log(a) \vert^{-h} \, a^{2}\, \Big[ \tilde{\lambda_{n_{0}}} + \Phi_{0} \, \Big( \int_{B} \overline{e}_{n_{0}} \Big)^{2} \Big]}{(1+\omega_{n_{0}}^{2})^{2}\,\Big[(1-\omega_{+}^{2} \, \mu_{0} \tau \lambda_{n_{0}})-\omega_{+}^{2} \, \mu_{0} \tau a^{2} \Phi_{0} (\int_{B} \overline{e}_{n_{0}})^{2}\Big]\, \Big[(1-\omega_{-}^{2} \, \mu_{0} \tau \lambda_{n_{0}})-\omega_{-}^{2} \, \mu_{0} \tau a^{2} \Phi_{0} (\int_{B} \overline{e}_{n_{0}})^{2}\Big]}
\end{eqnarray*}
hence $J = \mathcal{O}(1)$. Going back to the formula of $\tilde{p}(t,x)$, we obtain: 
\begin{equation*}
\tilde{p}(t,x) =   \frac{-4 \, t \, \Im(\tau)}{1+\omega_{n_{0}}^{2}}  \frac{\vert <u_{2};e^{(2)}_{n_{0}}> \vert^{2}}{\big(t^{2}-\vert x-z_{2} \vert^{2}\big)^{3/2}}  + 2 \,\omega_{n_{0}}^{2} \, \mu_{0} \, \Big[ \int_{D_{1}} e_{n_{0}} \Big]\sum_{i=1}^{2} \, \Re\big[\tau \, u_{0}(z_{i}) \; I_{i} J \big]  +   \mathcal{O}\big(\vert \log(a) \vert^{\max(-1,2h-2)}\big),
\end{equation*}
and 
\begin{equation*}
\bigg\vert 2 \,\omega_{n_{0}}^{2} \, \mu_{0} \, \big[ \int_{D_{1}} e_{n_{0}} \big]\sum_{i=1}^{2} \, \Re\big[\tau \, u_{0}(z_{i}) \; I_{i} J \big] \bigg\vert \leq a \, \vert \tau \vert \, \vert I_{i} \vert = \mathcal{O}\big( \vert \log(a) \vert^{-1} \big)
\end{equation*}
Finally, we have the desired approximation formula
\begin{equation*}
\tilde{p}(t,x) =   \frac{-4}{1+\omega_{n_{0}}^{2}} t \, \Im(\tau) \, \vert <u_{2};e^{(2)}_{n_{0}}> \vert^{2}  \, \frac{1}{\big(t^{2}-\vert x-z_{2} \vert^{2}\big)^{3/2}} + \mathcal{O}\big(\vert \log(a) \vert^{max(-1,2h-2)}\big).
\end{equation*}

\bigskip
\section{A priori estimations}\label{appendixlemma}
\subsection{A priori estimates on the electric field}

\begin{proof}{of \bf{Proposition \ref{abc}}}\\
In order to prove the a priori estimation $(\ref{prioriest})$, we proceed in two steps. First we do it for one single particle and then for multiple particles.
\begin{itemize}
\item[Step 1/] \underline{Case of one particle:} \\
Remember that the eigenvalues and eigenfunctions of the logarithmic operator satisfy 
\begin{equation*}
\int_{D} \Phi_{0}(x,y) \, e_{n}(y) \, dy \,=\, \lambda_{n} \, e_{n}(x) \qquad \quad in \quad D,
\end{equation*}
and after scaling we get, with $\tilde{e}_{n}(\cdot) := e_{n}\left(\frac{\cdot-z}{\varepsilon}\right),$
\begin{equation}\label{eigevfB}
a^{2} \, \Bigg[ \int_{B} \Phi_{0}(\eta,\xi) \, \tilde{e}_{n}(\xi) \, d\xi - \, \frac{1}{2\pi} \, \log(a) \int_{B} \tilde{e}_{n}(\xi) d\xi \Bigg]\,=\, \lambda_{n} \, \tilde{e}_{n}(\eta) \qquad \quad in \quad B.
\end{equation}
Integrating the equation $(\ref{eigevfB})$ over $B$ we obtain 
\begin{equation*}
\int_{B} \int_{B} \Phi_{0}(\eta,\xi) \, \tilde{e_{n}}(\xi) \, d\eta d\xi = \Bigg[\frac{1}{2\pi} \log(a) \vert B \vert \,+\, \frac{\lambda_{n}}{a^{2}} \Bigg]  \,\int_{B} \tilde{e_{n}} d\eta.
\end{equation*} 
Multiplying $(\ref{eigevfB})$ by $\tilde{e_{m}}$ and integrating over $B$ we get: 
\begin{equation*}
\int_{B} \int_{B} \Phi_{0}(\eta,\xi) \, \tilde{e_{n}}(\xi) \,\tilde{e_{m}}(\eta) \, d\eta d\xi - \frac{1}{2\pi} \, \log(a) \int_{B} \tilde{e_{n}}(\xi) d\xi \,\int_{B} \tilde{e_{m}}(\eta) d\eta \, = \, \frac{\lambda_{n}}{a^{2}} \,  \int_{B} \tilde{e_{n}} \, \tilde{e_{m}} d\eta.
\end{equation*} 
Remark that when $m \neq n$, thanks to the fact that $\big\lbrace \tilde{e_{n}} \big\rbrace_{n \in \mathbb{N}} $ forms an orthogonal basis in $\mathbb{L}^{2}(B)$, we obtain  
\begin{equation}\label{whennneqm}
\int_{B} \int_{B} \Phi_{0}(\eta,\xi) \, \tilde{e_{n}}(\xi) \,\tilde{e_{m}}(\eta) \, d\eta d\xi = \frac{1}{2\pi} \, \log(a) \,  \int_{B} \tilde{e_{n}} d\xi \, \int_{B} \tilde{e_{m}} d\xi
\end{equation}
and when $m = n$, we get 
\begin{equation}\label{whenn=m}
\int_{B} \int_{B} \Phi_{0}(\eta,\xi) \, \tilde{e_{n}}(\xi) \,\tilde{e_{n}}(\eta) \, d\eta d\xi - \frac{1}{2\pi} \, \log(a) \Bigg( \int_{B} \tilde{e_{n}} d\xi \Bigg)^{2} \, = \, \frac{\lambda_{n}}{a^{2}} \,  \Vert \tilde{e_{n}} \Vert^{2}.
\end{equation}
After normalisation
\begin{equation}\label{ha}
\int_{B} \int_{B} \Phi_{0}(\eta,\xi) \, \frac{\tilde{e_{n}}(\xi)}{\Vert \tilde{e_{n}}\Vert} \, \frac{\tilde{e_{n}}(\eta)}{\Vert \tilde{e_{n}}\Vert} \, d\eta d\xi = \frac{1}{\Vert \tilde{e_{n}} \Vert^{2}} \Bigg[\frac{1}{2\pi} \, \log(a) \Bigg( \int_{B} \tilde{e_{n}} d\xi \Bigg)^{2} \, + \, \frac{\lambda_{n}}{a^{2}} \,  \Vert \tilde{e_{n}} \Vert^{2} \Bigg].
\end{equation}
We denote $\overline{e}_{n} := \tilde{e}_{n} / \Vert \tilde{e}_{n}\Vert$ the orthonormalized basis in $\mathbb{L}^{2}(B)$, and we set 
\begin{equation}\label{lambdatilde}
\tilde{\lambda_{n}} := \int_{B} \int_{B} \Phi_{0}(\eta,\xi) \, \overline{e}_{n}(\eta) \, \overline{e}_{n}(\xi) \,  \, d\eta \, d\xi,
\end{equation}
from $(\ref{ha})$ and $(\ref{lambdatilde})$ we deduce that
\begin{equation}\label{G}
\tilde{\lambda}_{n} = \frac{\lambda_{n}}{a^{2}} +\frac{1}{2\pi} \log(a) \Bigg(\int_{B} \overline{e}_{n} d\xi\Bigg)^{2}.
\end{equation}
Thanks to L.S.E and Green kernel expansion $(\ref{Gkexpansion})$, we have 
\begin{eqnarray*}
u_{1}(x) &-& \omega^{2} \mu_{0} \tau \, \int_{D} \Big(\Phi_{0}(x,y) -\frac{1}{2\pi} \log(k)(y)+\Gamma  \Big) \, u_{1}(y) \, dy \\ &=& u_{0}(x)  + \omega^{2} \mu_{0} \tau \, \mathcal{O}\left( \int_{D} \vert x-y \vert \, \log(\vert x-y \vert) \, u_{1}(y) \, dy \right) \quad in \quad D.
\end{eqnarray*}
After Taylor expansion of the function $\log(k)$ near the point $z$, we obtain 
\begin{eqnarray*}
u_{1}(x) &-& \omega^{2} \mu_{0} \tau \, \int_{D} \Big(\Phi_{0}(x,y) -\frac{1}{2\pi} \log(k)(z)+\Gamma  \Big) \, u_{1}(y) \, dy \\ &=& u_{0}(x) +   \mathcal{O}\left(\tau \, \int_{D} \vert x-y \vert \, \log(\vert x-y \vert) \, u_{1}(y) \, dy \right) \\ &+& \mathcal{O}\Big(\tau \, \int_{D} \int_{0}^{1} (y-z)\centerdot \nabla \log(k)(z+t(y-z)) dt \, u_{1}(y) \, dy \Big).
\end{eqnarray*}
Now scaling, we have
\begin{eqnarray*}
\tilde{u}_{1}(\eta) & - &  \omega^{2} \mu_{0} \tau \, a^{2} \, \int_{B} \Big(\Phi_{0}(\eta,\xi) - \frac{1}{2\pi} \log(k)(z) + \Gamma  \Big) \, \tilde{u}_{1}(\xi) \, d\xi + \omega^{2} \mu_{0} \tau \, a^{2} \, \frac{1}{2\pi} \; \log(a) \, \int_{B} \tilde{u}_{1} \; d\xi \\ & = &  \tilde{u_{0}}(\eta) + \mathcal{O}\bigg( \tau \, a^{3} \, \int_{B} \vert \eta -\xi \vert \, \log(\vert \eta - \xi \vert) \, \tilde{u}_{1}(\xi) \, d\xi 
 \bigg) + \mathcal{O}\bigg(  \tau \, a^{3} \, \log(a) \, \int_{B} \vert \eta -\xi \vert \, \tilde{u}_{1}(\xi) \, d\xi \bigg) \\
 & +& \mathcal{O}\Big(a^{3} \, \tau \, \int_{B} \tilde{u}_{1} \, d\xi \Big).
\end{eqnarray*}
Using the basis, we obtain
\label{technics}
\begin{eqnarray*}
<\tilde{u}_{1};\overline{e}_{n_{0}}> && \left[1 - \omega^{2} \mu_{0} \tau \, a^{2}  \, \int_{B} \, \int_{B} \Phi_{0}(\eta,\xi)  \, \overline{e}_{n_{0}}(\xi) \, d\xi \overline{e}_{n_{0}}(\eta) \, d\eta +   \frac{\omega^{2} \mu_{0} \tau \, a^{2}}{2\pi} \, \log(a) \,     \left[ \int_{B} \overline{e}_{n_{0}} d\xi \right]^{2} \right] \\ 
& = &  <\tilde{u_{0}};\overline{e}_{n_{0}}> +\mathcal{O}\bigg(  \tau \, a^{3} \, \int_{B} \,\overline{e}_{n_{0}}(\eta) \int_{B} \vert \eta -\xi \vert \, \log(\vert \eta - \xi \vert) \, \tilde{u}_{1}(\xi) \, d\xi d\eta \bigg) \\
 & + & \mathcal{O}\bigg(\tau \, a^{3} \, \log(a) \,\int_{B} \overline{e}_{n_{0}}(\eta) \int_{B} \vert \eta -\xi \vert \, \tilde{u}_{1}(\xi) \, d\xi \, d\eta \bigg) \\ 
 &-&   \omega^{2} \mu_{0} \tau \, a^{2} \, \frac{1}{2\pi}  \, \log(a) \, <1;\overline{e}_{n_{0}}> \, \sum_{n \neq n_{0}} <\tilde{u}_{1};\overline{e}_{n}> \, \int_{B} \overline{e}_{n} d\xi \\ &+& \omega^{2} \mu_{0} \tau \, a^{2} \Big(- \frac{1}{2\pi}\log(k)(z)+\Gamma \Big) \, \int_{B}\tilde{u}_{1} d\xi \,\int_{B} \overline{e}_{n_{0}} d\eta \\ 
& + &  \omega^{2} \mu_{0} \tau \, a^{2} \, \sum_{n \neq n_{0}} <\tilde{u}_{1};\overline{e}_{n}> \int_{B} \, \int_{B} \Phi_{0}(\eta,\xi)  \, \overline{e}_{n}(\xi) \, d\xi \overline{e}_{n_{0}}(\eta) \, d\eta + \mathcal{O}\Big(a^{3} \, \tau \, \int_{B} \tilde{u}_{1} \, d\xi \Big).
\end{eqnarray*}
After simplifications and using  $(\ref{whennneqm})$  and  $(\ref{ha})$   we get
\begin{eqnarray}\label{vVone}
\nonumber
<\tilde{u}_{1};\overline{e}_{n_{0}}>  &=& \frac{1}{\Big[1-\omega^{2} \mu_{0} \tau \, \lambda_{n_{0}} \Big]} \Bigg[ <\tilde{u_{0}};\overline{e}_{n_{0}}> + \omega^{2} \mu_{0} \tau \, a^{2} \Big( -\frac{1}{2\pi} \log(k)(z)+\Gamma \Big) \, \int_{B}\tilde{u}_{1} d\xi \,\int_{B} \overline{e}_{n_{0}}  d\eta \\ \nonumber & + & \mathcal{O}\bigg( \tau \, a^{3} \, \int_{B} \,\overline{e}_{n_{0}}(\eta) \int_{B} \vert \eta -\xi \vert \, \log(\vert \eta - \xi \vert) \, \tilde{u}_{1}(\xi) \, d\xi d\eta \bigg) \\
 & + & \mathcal{O}\bigg(\tau \, a^{3} \, \log(a) \,\int_{B} \overline{e}_{n_{0}}(\eta) \int_{B} \vert \eta -\xi \vert \, \tilde{u}_{1}(\xi) \, d\xi \, d\eta \bigg) +  \mathcal{O}\Big(a^{3} \, \tau \, \int_{B} \tilde{u}_{1} \, d\xi \Big) \Bigg].   
\end{eqnarray}
We take\footnote{The dielectric-resonance that we want to 
excite is $\omega_{n_{0}}$ given by
\begin{equation}\label{exactMieresonance}
\omega_{n_{0}}^{2} =  \frac{1}{\mu_{0} \, \lambda_{n_{0}} \, a^{-2} \, \vert \log(a) \vert^{-1}}.
\end{equation}
} $\tau$ and $\omega$ so that
\begin{equation}\label{H}
\tau \simeq \frac{1}{a^{2} \, \vert \log(a) \vert}  \quad \text{and} \quad \omega^{2} = \frac{\Big(1 \pm \vert \log(a) \vert^{-h}\Big)}{\mu_{0} \, \lambda_{n_{0}} \, a^{-2} \, \vert \log(a) \vert^{-1}}.
\end{equation}
With this choice we have the estimation
\begin{equation*}
\frac{1}{\Big\vert 1-\omega^{2} \mu_{0} \tau \, \lambda_{n_{0}} \,\Big\vert} = \mathcal{O}(\vert \log(a) \vert^{h}).
\end{equation*}
Then
\begin{eqnarray*}
\vert <\tilde{u}_{1};\overline{e}_{n_{0}}> \vert  & \leq & \vert \log(a) \vert^{h} \, \Bigg[ \vert <\tilde{u_{0}};\overline{e}_{n_{0}}> \vert + a \,\vert \log(a) \vert^{-1} \, \Bigg\vert \int_{B} \,\overline{e}_{n_{0}}(\eta) \int_{B} \vert \eta -\xi \vert \, \log(\vert \eta - \xi \vert) \, \tilde{u}_{1}(\xi) \, d\xi d\eta \Bigg\vert \\ & + & a \Bigg\vert \int_{B} \overline{e}_{n_{0}}(\eta) \int_{B} \vert \eta -\xi \vert \, \tilde{u}_{1}(\xi) \, d\xi \, d\eta \Bigg\vert + \vert \log(a) \vert^{-1}  \, \Bigg\vert \int_{B}\tilde{u}_{1} d\xi \,\Bigg\vert \, \Bigg\vert \int_{B} \overline{e}_{n_{0}} d\eta \Bigg\vert + a^{2} \, \vert \log(a) \vert^{-1} \; \Vert \tilde{u}_{1} \Vert  \Bigg]  
\end{eqnarray*}
Obviously the term $\vert <\tilde{u_{0}};\overline{e}_{n_{0}}> \vert$ dominates the others, but we need to check this mathematically by estimating the error part. This last one will be subdivided into three parts. We have  
\begin{itemize}
\item[] \underline{Estimation of} \qquad
$s_{1}  :=    \int_{B} \,\overline{e}_{n_{0}}(\eta) \int_{B} \vert \eta -\xi \vert \, \log(\vert \eta - \xi \vert) \, \tilde{u}_{1}(\xi) \, d\xi d\eta$. 
\begin{eqnarray*}
\vert s_{1} \vert & \leq &  \Vert \overline{e}_{n_{0}} \Vert \; \Bigg( \int_{B} \Big\vert\int_{B} \vert \eta -\xi \vert \, \log(\vert \eta - \xi \vert) \, \tilde{u}_{1}(\xi) \, d\xi \Big\vert^{2} d\eta \Bigg)^{\frac{1}{2}} = \mathcal{O}(\Vert \tilde{u}_{1} \Vert). 
\end{eqnarray*}
\item[] \underline{Estimation of} \quad \qquad $s_{2} := \int_{B} \overline{e}_{n_{0}}(\eta) \int_{B} \vert \eta -\xi \vert  \tilde{u}_{1}(\xi)  d\xi  d\eta$.
\begin{equation*}
\vert s_{2} \vert  \leq   \Vert \overline{e}_{n_{0}} \Vert  \Bigg( \int_{B} \Big\vert\int_{B} \vert \eta -\xi \vert  \tilde{u}_{1}(\xi)  d\xi \Big\vert^{2} d\eta \Bigg)^{\frac{1}{2}} = \mathcal{O}(\Vert \tilde{u}_{1} \Vert). 
\end{equation*}
\item[] \underline{Estimation of} \quad \qquad $s_{3}  :=  \int_{B}\tilde{u}_{1} d\xi \, \int_{B} \overline{e}_{n_{0}} d\eta$.
\begin{equation*}
\vert s_{3} \vert := \Bigg\vert \int_{B}\tilde{u}_{1} d\xi \,\Bigg\vert \, \Bigg\vert \int_{B} \overline{e}_{n_{0}} d\eta \Bigg\vert = \mathcal{O}(\Vert \tilde{u}_{1} \Vert).
\end{equation*}
\end{itemize}
Then
\begin{equation*}
\vert <\tilde{u}_{1};\overline{e}_{n_{0}}> \vert   \leq  \vert \log(a) \vert^{h} \, \Bigg[ \vert <\tilde{u_{0}};\overline{e}_{n_{0}}> \vert +  \Vert \tilde{u}_{1} \Vert\; \Big( a \,\vert \log(a) \vert^{-1}  +  a  +  \vert \log(a) \vert^{-1}+a^{2} \, \vert \log(a) \vert^{-1} \; \Big) \Bigg], 
\end{equation*}
and then
\begin{equation}\label{!n0}   
\vert <\tilde{u}_{1};\overline{e}_{n_{0}}> \vert^{2}   \leq  \vert \log(a) \vert^{2h} \, \Bigg[ \vert <\tilde{u_{0}};\overline{e}_{n_{0}}> \vert^{2} +  \Vert \tilde{u}_{1} \Vert^{2} \; \vert \log(a) \vert^{-2} \Bigg].  
\end{equation}
In what follows, we calculate an estimation of $\underset{n \neq n_{0}}{\sum} \vert <\tilde{u}_{1};\overline{e}_{n}> \vert^{2}$. We star with equation $(\ref{vVone})$, since the other steps are the same, to obtain:
\begin{eqnarray*}
<\tilde{u}_{1};\overline{e}_{n}>  &=& \frac{1}{\Big[1-\omega^{2} \mu_{0} \tau \, \lambda_{n} \Big]} \Bigg[ <\tilde{u_{0}};\overline{e}_{n}> + \mathcal{O}\bigg(\tau \, a^{3} \, \int_{B} \,\overline{e}_{n}(\eta) \int_{B} \vert \eta -\xi \vert \, \log(\vert \eta - \xi \vert) \, \tilde{u}_{1}(\xi) \, d\xi d\eta \bigg) \\ 
& + &  \mathcal{O}\bigg( \tau \, a^{3} \, \log(a) \,\int_{B} \overline{e}_{n}(\eta) \int_{B} \vert \eta -\xi \vert \, \tilde{u}_{1}(\xi) \, d\xi \, d\eta \bigg) + \omega^{2} \mu_{0} \tau \, a^{2} \Big(-\frac{1}{2\pi} \log(k)(z)+\Gamma \Big) \, \int_{B}\tilde{u}_{1} d\xi \,\int_{B} \overline{e}_{n} d\eta \\
&+&  \mathcal{O}\Big(a^{3} \, \tau \, \int_{B} \tilde{u}_{1} \; d\xi \int_{B} \overline{e}_{n} \; d\eta  \Big) \Bigg].   
\end{eqnarray*}
Then 
\begin{eqnarray*}
\sum_{n \neq n_{0}} \vert <\tilde{u}_{1};\overline{e}_{n}> \vert^{2}  &\leq& C^{te} \Bigg[ \sum_{n \neq n_{0}} \vert <\tilde{u_{0}};\overline{e}_{n}> \vert^{2} + a^{2} \, \vert \log(a) \vert^{-2} \sum_{n \neq n_{0}} \Bigg\vert \int_{B} \,\overline{e}_{n}(\eta) \int_{B} \vert \eta -\xi \vert \, \log(\vert \eta - \xi \vert) \, \tilde{u}_{1}(\xi) \, d\xi d\eta \Bigg\vert^{2} \\ 
& + &  a^{2} \,\sum_{n \neq n_{0}} \Bigg\vert \int_{B} \overline{e}_{n}(\eta) \int_{B} \vert \eta -\xi \vert \, \tilde{u}_{1}(\xi) \, d\xi \, d\eta \Bigg\vert^{2} + \vert \log(a) \vert^{-2} \, \Vert \tilde{u}_{1} \Vert^{2} \; \sum_{n \neq n_{0}}  \, \Bigg\vert \int_{B} \overline{e}_{n}\, d\eta \Bigg\vert^{2} \\ 
&+& a^{2} \; \vert \log(a) \vert^{-2} \, \Vert \tilde{u}_{1} \Vert^{2} \; \sum_{n \neq n_{0}}  \, \Bigg\vert \int_{B} \overline{e}_{n}\, d\eta \Bigg\vert^{2} \Bigg].   
\end{eqnarray*}
On the right side, except for the first term, we need to estimate the terms containing series. For this, we have
\begin{equation*}
\sum_{n \neq n_{0}} \Bigg\vert \int_{B} \,\overline{e}_{n}(\eta) \int_{B} \vert \eta -\xi \vert \, \log(\vert \eta - \xi \vert) \, \tilde{u}_{1}(\xi) \, d\xi d\eta \Bigg\vert^{2} \leq \int_{B} \Big\vert \int_{B} \vert \eta -\xi \vert \, \log(\vert \eta - \xi \vert) \, \tilde{u}_{1}(\xi) \, d\xi \Big\vert^{2} \, d\eta, 
\end{equation*}
since the function $\vert \centerdot \vert \, \log(\vert \centerdot \vert)$ is bounded on $B$ we get
\begin{equation*}
\int_{B} \Big\vert \int_{B} \vert \eta -\xi \vert \, \log(\vert \eta - \xi \vert) \, \tilde{u}_{1}(\xi) \, d\xi \Big\vert^{2} \, d\eta \leq  \mathcal{O}(\Vert \tilde{u}_{1} \Vert^{2}).
\end{equation*}
The same argument as before allows to deduce that 
\begin{equation*}
\sum_{n \neq n_{0}} \Bigg\vert \int_{B} \overline{e}_{n}(\eta) \int_{B} \vert \eta -\xi \vert \, \tilde{u}_{1}(\xi) \, d\xi \, d\eta \Bigg\vert^{2} \leq \mathcal{O}(\Vert \tilde{u}_{1} \Vert^{2}). 
\end{equation*}
Obviously we have also
\begin{equation*}
\sum_{n \neq n_{0}}  \, \Bigg\vert \int_{B} \overline{e}_{n}\, d\eta \Bigg\vert^{2} \leq \Vert 1 \Vert^{2}.
\end{equation*}
Hence
\begin{equation}\label{serienneqn0} 
\sum_{n \neq n_{0}} \vert <\tilde{u}_{1};\overline{e}_{n}> \vert^{2} \leq C^{te} \Bigg[ \sum_{n \neq n_{0}} \vert <\tilde{u_{0}};\overline{e}_{n}> \vert^{2} +  \vert \log(a) \vert^{-2} \Vert \tilde{u}_{1}  \Vert^{2} \Bigg]. 
\end{equation}
By adding $(\ref{!n0})$ and $(\ref{serienneqn0})$, we get
\begin{eqnarray*}
\Vert \tilde{u}_{1} \Vert^{2} & = & \vert <\tilde{u}_{1};\overline{e}_{n_{0}}> \vert^{2} + \sum_{n \neq n_{0}} \vert <\tilde{u}_{1};\overline{e}_{n}> \vert^{2} \\ 
 & \leq & \vert \log(a) \vert^{2h} \, \Bigg[ \vert <\tilde{u_{0}};\overline{e}_{n_{0}}> \vert^{2} +  \Vert \tilde{u}_{1} \Vert^{2} \; \vert \log(a) \vert^{-2} \Bigg] \\ 
&+& C^{te} \Bigg[ \sum_{n \neq n_{0}} \vert <\tilde{u_{0}};\overline{e}_{n}> \vert^{2} +  \vert \log(a) \vert^{-2} \Vert \tilde{u}_{1}  \Vert^{2} \Bigg] 
\end{eqnarray*}
hence
\begin{eqnarray*}
\Vert \tilde{u}_{1} \Vert^{2} & \leq & \vert \log(a) \vert^{2h} \, \Vert \tilde{u_{0}} \Vert^{2} + \vert \log(a) \vert^{2h-2} \, \Vert \tilde{u}_{1} \Vert^{2}, \\ 
\Vert \tilde{u}_{1} \Vert^{2} (1-\vert \log(a) \vert^{2h-2}) & \leq & \vert \log(a) \vert^{2h} \, \Vert \tilde{u_{0}} \Vert^{2} 
\end{eqnarray*}
and, as $h < 1$, 
\begin{equation*}
\Vert \tilde{u}_{1} \Vert^{2}  \leq  (1-\vert \log(a) \vert^{2h-2})^{-1} \, \vert \log(a) \vert^{2h} \, \Vert \tilde{u_{0}} \Vert^{2}  \leq  \vert \log(a) \vert^{2h} \, \Vert \tilde{u_{0}} \Vert^{2}, 
\end{equation*}
or
\begin{equation}\label{aprioriestimation}
\Vert u_{1} \Vert_{\mathbb{L}^{2}(D)}  \leq  \vert \log(a) \vert^{h} \, \Vert u_{0} \Vert_{\mathbb{L}^{2}(D)}. 
\end{equation}
The following proposition makes a link between the Fourier coefficient of the generated total field and that of the source field. 
\begin{proposition}\label{Y}
We have
\begin{equation*}\label{equa611}
<u_{1};e_{n_{0}}> = \frac{ <u_{0};e_{n_{0}}>}{(1-\omega^{2} \mu_{0} \tau \, \lambda_{n_{0}})}  + \mathcal{O}(a\,\vert \log(a) \vert^{2h-1}).
\end{equation*}
\end{proposition}
\begin{proof}
We write
\begin{equation*}
\int_{B} \tilde{u}_{1} d\xi = <\tilde{u}_{1};\overline{e}_{n_{0}}> \int_{B} \overline{e}_{n_{0}} d\xi + \sum_{n \neq n_{0}} <\tilde{u}_{1};\overline{e}_{n}> \int_{B} \overline{e}_{n} d\xi.
\end{equation*} 
Use this representation in $(\ref{vVone})$ and rearrange the equation to get 
\begin{eqnarray*}
<\tilde{u}_{1};\overline{e}_{n_{0}}>  &=& \frac{1}{\Bigg[1-\omega^{2} \mu_{0} \tau \, \lambda_{n_{0}} -  \omega^{2} \mu_{0} \tau \, a^{2} \Big(-\frac{1}{2\pi} \log(k)(z)+\Gamma \Big) \,\Big( \int_{B} \overline{e}_{n_{0}}  d\eta \Big)^{2}\Bigg]} \Bigg[ <\tilde{u_{0}};\overline{e}_{n_{0}}> \\
&+&\omega^{2} \mu_{0} \tau \, a^{2} \Big( \frac{-1}{2\pi} \log(k)+\Gamma \Big) \, \,\int_{B} \overline{e}_{n_{0}}  d\eta \; \sum_{n \neq n_{0}} <\tilde{u}_{1};\overline{e}_{n} > \, \int_{B} \overline{e}_{n} d\xi + \mathcal{O}\Big(a^{3} \, \tau \, \int_{B} \tilde{u}_{1} \, d\xi \Big) \\
&+& \omega^{2} \mu_{0} \tau \, a^{3} \, \int_{B} \,\overline{e}_{n_{0}}(\eta) \int_{B} \vert \eta -\xi \vert \, \log(\vert \eta - \xi \vert) \, \tilde{u}_{1}(\xi) \, d\xi d\eta \\ &+& \omega^{2} \mu_{0} \tau \, a^{3} \, \log(a) \,\int_{B} \overline{e}_{n_{0}}(\eta) \int_{B} \vert \eta -\xi \vert \, \tilde{u}_{1}(\xi) \, d\xi \, d\eta\Bigg].   
\end{eqnarray*}
We need to estimate the four last terms between brackets. We have
\begin{equation*}
\Big\vert \omega^{2} \mu_{0} \tau \, a^{2} \Big( \frac{-1}{2\pi} \log(k)+\Gamma \Big) \, \,\int_{B} \overline{e}_{n_{0}}  d\eta \; \sum_{n \neq n_{0}} <\tilde{u}_{1};\overline{e}_{n} > \, \int_{B} \overline{e}_{n} d\xi \Big\vert \lesssim  \tau \, a^{2} \; \Vert \tilde{u_{0}} \Vert \; \Vert 1 \Vert  = \mathcal{O}(\vert \log(a) \vert^{-1}).
\end{equation*}
Next, use Holder inequality and the a priori estimate to obtain  
\begin{equation}\label{up}
\Big\vert \omega^{2} \mu_{0} \tau \, a^{3} \, \int_{B} \,\overline{e}_{n_{0}}(\eta) \int_{B} \vert \eta -\xi \vert \, \log(\vert \eta - \xi \vert) \, \tilde{u}_{1}(\xi) \, d\xi d\eta \Big\vert \lesssim  \mathcal{O}(a \vert \log(a) \vert^{h-1}).
\end{equation}
Remark that the following term 
\begin{equation*}
\omega^{2} \mu_{0} \tau \, a^{3} \, \log(a) \,\int_{B} \overline{e}_{n_{0}}(\eta) \int_{B} \vert \eta -\xi \vert \, \tilde{u}_{1}(\xi) \, d\xi \, d\eta ,  
\end{equation*}
up to multiplicative constant $\vert \log(a) \vert$ behaves as $(\ref{up})$, then we estimate it as $\mathcal{O}(a \vert \log(a) \vert^{h})$, and obviously we have 
\begin{equation*}
a^{3} \, \tau \, \int_{B} \tilde{u}_{1} \, d\xi  \sim \mathcal{O}\big( a \, \vert \log(a) \vert^{h-1} \big).
\end{equation*}
Finally, we obtain
\begin{equation}\label{Z}
<u_{1};e_{n_{0}}> = \frac{ <u_{0};e_{n_{0}}>}{\Bigg[1-\omega^{2} \mu_{0} \tau \, \lambda_{n_{0}} -  \omega^{2} \mu_{0} \tau \, \Big( -\frac{1}{2\pi} \log(k)(z)+\Gamma \Big) \,\Big( \int_{D} e_{n_{0}}  d\eta \Big)^{2}\Bigg]}  + \mathcal{O}(a\,\vert \log(a) \vert^{h-1}), 
\end{equation}
or in the following form
\begin{eqnarray*}
<u_{1};e_{n_{0}}> &=& \frac{ <u_{0};e_{n_{0}}>}{(1-\omega^{2} \mu_{0} \tau \, \lambda_{n_{0}}) \, \Bigg[1 - \displaystyle\frac{\omega^{2} \mu_{0} \tau \, a^{2} \Big( -\frac{1}{2\pi} \log(k)(z)+\Gamma \Big) \,\Big( \int_{B} \overline{e}_{n_{0}}  d\eta \Big)^{2}}{(1-\omega^{2} \mu_{0} \tau \, \lambda_{n_{0}})}\Bigg]}  + \mathcal{O}(a\,\vert \log(a) \vert^{h-1})\\
&=& \frac{ <u_{0};e_{n_{0}}>}{(1-\omega^{2} \mu_{0} \tau \, \lambda_{n_{0}}) \, \Big[1 + \mathcal{O}(\vert \log(a) \vert^{h-1}) \Big]}  + \mathcal{O}(a\,\vert \log(a) \vert^{h-1}) \\ &=& \frac{ <u_{0};e_{n_{0}}>}{(1-\omega^{2} \mu_{0} \tau \, \lambda_{n_{0}})}  + \mathcal{O}(a\,\vert \log(a) \vert^{2h-1}) 
\end{eqnarray*}
which ends the proof.
\end{proof}
\item[Step 2/] \underline{Case of multiples particles:} 
Consider the L.S.E for multiple particles
\begin{equation}\label{LSED}
v_{i}(x) - \omega^{2} \, \mu_{0} \, \tau \int_{D_{i}} G_{k}(x;y) v_{i}(y) \, dy - \omega^{2} \, \mu_{0} \, \tau \, \sum_{m \neq i}^{M} \int_{D_{m}}  G_{k}(x;y) v_{m}(y) \, dy = u_{0}(x), \quad x \in D_{i}.
\end{equation}
We use the expansion formula $(\ref{Gkexpansion})$ of $G_{k}(x;y)$ to write
\begin{eqnarray*}
v_{i}(x) & - & \omega^{2} \, \mu_{0} \, \tau \int_{D_{i}} \Phi_{0}(x,y) v_{i}(y) \, dy = u_{0}(x) + \omega^{2} \, \mu_{0} \, \tau \,\Big(-\frac{1}{2\pi} \log(k)(z_{i})+\Gamma \Big) \int_{D_{i}} v_{i} dy + \mathcal{O}\Big( \tau \, a \, \int_{D_{i}} v \; dy \Big) \\ & + & \omega^{2} \, \mu_{0} \, \tau \, \int_{D_{i}} \vert x-y \vert \, \log
(\vert x-y \vert) \, v_{i}(y) dy + \omega^{2} \, \mu_{0} \, \tau \, \sum_{m \neq i}^{M} \int_{D_{m}} \Phi_{0}(x,y) v_{m}(y) \, dy + \mathcal{O}\Big( \tau \, a \, \sum_{m \neq i}^{M} \int_{D_{m}} v \; dy \Big) \\ &+& \omega^{2} \, \mu_{0} \, \tau \, \sum_{m \neq i}^{M} \Big(-\frac{1}{2\pi} \log(k)(z_{m})+\Gamma \Big) \, \int_{D_{m}} v_{m} dy + \omega^{2} \, \mu_{0} \, \tau \, \sum_{m \neq i}^{M} \int_{D_{m}} \vert x-y \vert \, \log(\vert x-y \vert) v_{m}(y) \, dy.
\end{eqnarray*}
Scaling, we obtain
\begin{eqnarray}\label{equa710}
\nonumber
\tilde{v}_{i}(\eta) & - & \omega^{2} \, \mu_{0} \, \tau \, a^{2}  \int_{B} \Phi_{0}(\eta,\xi) \tilde{v}_{i}(\xi) \, d\xi = u_{0}(z_{i}+a\, \eta) - \omega^{2} \, \mu_{0} \, \tau \, a^{2} \, \frac{1}{2\pi} \, \log(a) \, \int_{B} \tilde{v}_{i} \, d\xi + \mathcal{O}\Big(\tau \, a^{3} \, \int_{B} \tilde{v}_{i} \; d\xi \Big) \\ \nonumber
&+&    \omega^{2} \, \mu_{0} \, \tau \,\Big(-\frac{1}{2\pi} \log(k)(z_{i})+\Gamma \Big) \, a^{2} \, \int_{B} \tilde{v}_{i} d\xi  +  \omega^{2} \, \mu_{0} \, \tau \,a^{3} \, \int_{B} \vert \eta-\xi \vert \, \log
(\vert \eta - \xi \vert) \, \tilde{v}_{i}(\xi) d\xi \\ \nonumber 
& + &  \omega^{2} \, \mu_{0} \, \tau \,a^{3} \, \log(a) \int_{B} \vert \eta-\xi \vert \,\tilde{v}_{i}(\xi) d\xi - \frac{1}{2 \pi} \; \omega^{2} \, \mu_{0} \, \tau \,a^{2} \, \sum_{m \neq i}^{M} \int_{B}  \log \vert (z_{i}-z_{m})+a(\eta - \xi) \vert \, \tilde{v}_{m}(\xi) \, d\xi \\ \nonumber
&+& \omega^{2} \, \mu_{0} \, \tau \,\,a^{2} \, \sum_{m \neq i}^{M} \Big( \frac{-1}{2\pi}\log(k)(z_{m})+\Gamma \Big)  \int_{B} \tilde{v}_{m}\, d\xi + \mathcal{O}\Big( \tau \, a^{3} \, \sum_{m \neq i}^{M} \int_{B} \tilde{v}_{m} \; d\xi \Big)  \\ 
&+& \omega^{2} \, \mu_{0} \, \tau \,a^{2} \, \sum_{m \neq i}^{M} \int_{B} \vert (z_{i}-z_{m})+a(\eta - \xi)\vert \, \log(\vert (z_{i}-z_{m})+a(\eta - \xi)\vert) \tilde{v}_{m}(\xi) \, d\xi.
\end{eqnarray}
We recall that 
\begin{equation*}
A_{0} \, v(x) = \int_{B} \Phi_{0}(x,y) \; v(y) \, dy, 
\end{equation*}
and denote 
\begin{equation*}
T \, v(x) := \int_{B}  v(y) \, dy, \quad x \in B.
\end{equation*}
Then: 
\begin{equation*}
\Big[I - \omega^{2} \, \mu_{0} \, \tau \, a^{2} \, A_{0} + \omega^{2} \, \mu_{0} \, \tau \, a^{2} \, \frac{1}{2\pi} \, \log(a) \, T \Big] \tilde{v}_{i} = u_{0}(z_{i}+ a\, \cdot) + \omega^{2} \, \mu_{0} \, \tau \,\Big( -\frac{1}{2\pi} \log(k)(z_{i})+\Gamma \Big) \, a^{2} \, \int_{B} \tilde{v}_{i} d\xi  
\end{equation*}
\begin{eqnarray*}
\phantom \qquad \qquad \qquad \qquad &+&  \omega^{2} \, \mu_{0} \, \tau \,a^{3} \, \int_{B} \vert \eta-\xi \vert \, \log
(\vert \eta - \xi \vert) \, \tilde{v}_{i}(\xi) d\xi  +   \omega^{2} \, \mu_{0} \, \tau \,a^{3} \, \log(a) \int_{B} \vert \eta-\xi \vert \,\tilde{v}_{i}(\xi) d\xi  \\ 
&-&\frac{1}{2\pi} \, \omega^{2} \, \mu_{0} \, \tau \,a^{2} \, \sum_{m \neq i}^{M} \int_{B}  \log \vert (z_{i}-z_{m})+a(\eta - \xi) \vert \, \tilde{v}_{m}(\xi) \, d\xi 
\\ &+& \omega^{2} \, \mu_{0} \, \tau \,a^{2} \, \sum_{m \neq i}^{M} \int_{B} \vert (z_{i}-z_{m})+a(\eta - \xi)\vert \, \log(\vert (z_{i}-z_{m})+a(\eta - \xi)\vert) \tilde{v}_{m}(\xi) \, d\xi \\
&+& \omega^{2} \, \mu_{0} \, \tau \,a^{2} \, \sum_{m \neq i}^{M} \Big(-\frac{1}{2\pi}\log(k)(z_{m})+\Gamma\Big) \, \int_{B} \tilde{v}_{m}\, d\xi + \mathcal{O}\Big( \tau \, a^{3} \, \sum_{m = 1}^{M} \int_{B} \tilde{v}_{m} \; d\xi \Big). 
\end{eqnarray*}
Also, we note by 
\begin{equation*}
\mathfrak{Res}(A_{0};T) := [I - \omega^{2} \, \mu_{0} \, \tau \, a^{2} \, A_{0} + \omega^{2} \, \mu_{0} \, \tau \, a^{2} \, \frac{1}{2\pi} \, \log(a) \, T \Big]^{-1}.
\end{equation*}
\begin{remark}
In the definition of the operator $\mathfrak{Res}(A_{0};T)$ we cannot neglect the operator $T$ since it scales with the same order as $A_{0}$.
\end{remark}
Then
\begin{eqnarray*}
 \tilde{v}_{i} &=& \mathfrak{Res}(A_{0};T) \, (u_{0}(z_{i}+a \; \cdot))  +    \omega^{2} \, \mu_{0} \, \tau \,\Big( \frac{-1}{2\pi}\log(k)(z_{i})+\Gamma \Big) \, a^{2} \, \int_{B} \tilde{v}_{i} d\xi \; \mathfrak{Res}(A_{0};T)(1)\\
 &+&  \omega^{2} \, \mu_{0} \, \tau \,a^{3} \,\mathfrak{Res}(A_{0};T) \Big( \int_{B} \vert \eta-\xi \vert \, \log
(\vert \eta - \xi \vert) \, \tilde{v}_{i}(\xi) d\xi \Big) \\
&+& \omega^{2} \, \mu_{0} \, \tau \,a^{3} \, \log(a) \; \mathfrak{Res}(A_{0};T) \Big( \int_{B} \vert \eta-\xi \vert \,\tilde{v}_{i}(\xi) d\xi \Big) \\
&-& \frac{1}{2\pi} \omega^{2} \, \mu_{0} \, \tau \,a^{2} \, \sum_{m \neq i}^{M} \; \mathfrak{Res}(A_{0};T) \Big( \int_{B}  \log \vert (z_{i}-z_{m})+a(\eta - \xi) \vert \, \tilde{v}_{m}(\xi) \, d\xi \Big) \\ 
&+& \omega^{2} \, \mu_{0} \, \tau \,a^{2} \, \sum_{m \neq i}^{M}  \Big(-\frac{1}{2\pi}\log(k)(z_{m})+\Gamma\Big) \, \int_{B} \tilde{v}_{m}\, d\xi \;\; \mathfrak{Res}(A_{0};T)(1)  \\
 &+& \omega^{2} \, \mu_{0} \, \tau \,a^{2} \, \sum_{m \neq i}^{M} \; \mathfrak{Res}(A_{0};T) \Big( \int_{B} \vert (z_{i}-z_{m})+a(\eta - \xi)\vert \, \log(\vert (z_{i}-z_{m})+a(\eta - \xi)\vert) \tilde{v}_{m}(\xi) \, d\xi \Big) \\ 
&+& \mathcal{O}\Big( \tau \, a^{3} \, \sum_{m = 1}^{M} \int_{B} \tilde{v}_{m} \; d\xi \Big) \; \mathfrak{Res}(A_{0};T)(1).
\end{eqnarray*}
Using the a priori estimate $(\ref{aprioriestimation})$, we obtain 
\begin{eqnarray*}
\Vert \tilde{v}_{i} \Vert & \leq & \vert \log(a) \vert^{h} \Vert u_{0}(z_{i}+a \; \cdot) \Vert + a \, \vert \log(a) \vert^{h-1} \Vert \tilde{v}_{i} \Vert +  \, a^{2} \, \Vert \tilde{v}_{i} \Vert \; \vert \log(a) \vert^{h} \; \Vert 1 \Vert \\ &+&  \,a^{3} \, \vert \log(a) \vert^{h} \; \Big\Vert \int_{B} \vert \eta-\xi \vert \, \log(\vert \eta - \xi \vert) \, \tilde{v}_{i}(\xi) d\xi \Big\Vert  +  \,a^{3} \,\vert \log(a) \vert \; \vert \log(a) \vert^{h} \Big\Vert \int_{B} \vert \eta-\xi \vert \,\tilde{v}_{i}(\xi) d\xi \Big\Vert \\ &+& \,a^{2} \, \vert \log(a) \vert^{h} \; \sum_{m \neq i}^{M} \; \Big\Vert \int_{B}  \log \vert (z_{i}-z_{m})+a(\eta - \xi) \vert \, \tilde{v}_{m}(\xi) \, d\xi \Big\Vert + \,a^{2} \, \vert \log(a) \vert^{h} \; \Vert 1 \Vert \; \sum_{m \neq i}^{M}  \Vert \tilde{v}_{m}\, \Vert \\ &+&  \,a^{2} \,\vert \log(a) \vert^{h} \sum_{m \neq i}^{M} \Big\Vert \int_{B} \vert (z_{i}-z_{m})+a(\eta - \xi)\vert \, \log(\vert (z_{i}-z_{m})+a(\eta - \xi)\vert) \tilde{v}_{m}(\xi) \, d\xi \Big\Vert \\ 
&+& a \, \vert \log(a) \vert^{h-1} \, \sum_{m = 1}^{M} \Vert \tilde{v}_{m} \Vert ,
\end{eqnarray*}
and 
\begin{eqnarray*}
\Big\Vert \int_{B}  \log \vert (z_{i}-z_{m})+a(\eta - \xi) \vert \, \tilde{v}_{m}(\xi) \, d\xi \Big\Vert^{2} &=& \int_{B} \, \Big\vert \int_{B}  \log \vert (z_{i}-z_{m})+a(\eta - \xi) \vert \, \tilde{v}_{m}(\xi) \, d\xi \Big\vert^{2} \,  d\eta \\ 
& \leq & \int_{B} \,  \int_{B} \Big\vert \log \vert (z_{i}-z_{m})+a(\eta - \xi) \vert \,\Big\vert^{2} d\xi  \,  d\eta \, \; \Vert \tilde{v}_{m} \Vert^{2}. 
\end{eqnarray*}
Hence
\begin{equation*}
\Big\Vert \int_{B}  \log \vert (z_{i}-z_{m})+a(\eta - \xi) \vert \, \tilde{v}_{m}(\xi) \, d\xi \Big\Vert \lesssim \log(1/d_{im}) \; \Vert \tilde{v}_{m} \Vert.
\end{equation*}
The same calculus allows to obtain 
\begin{equation*}
\Big\Vert \int_{B} \vert (z_{i}-z_{m})+a(\eta - \xi)\vert \, \log(\vert (z_{i}-z_{m})+a(\eta - \xi)\vert) \tilde{v}_{m}(\xi) \, d\xi \Big\Vert \lesssim d_{im} \, \log(1/d_{im}) \; \Vert \tilde{v}_{m} \Vert.
\end{equation*}
Gathering these estimates, we have
\begin{eqnarray*}
\Vert \tilde{v}_{i} \Vert & \leq & \vert \log(a) \vert^{h} \Vert u_{0}(z_{i}+a \; \cdot) \Vert  +  \Big[ a^{3}  \; \vert \log(a) \vert^{h} \; +  \,a^{3} \, \vert \log(a) \vert^{h}  +  \,a^{3} \; \vert \log(a) \vert^{1+h}  +  \,a \; \vert \log(a) \vert^{h-1} \Big] \Vert \tilde{v}_{i} \Vert \\ 
&+& \Big[ a^{2} \; \vert \log(a) \vert^{h} \; \log(1/d)   + \,a^{2} \, \vert \log(a) \vert^{h} \; \Vert 1 \Vert \;   + a^{2} \, \vert \log(a) \vert^{h}  + a \, \vert \log(a) \vert^{h-1} \Big] \sum_{m \neq i}^{M}  \Vert \tilde{v}_{m}\, \Vert. 
\end{eqnarray*}
Then
\begin{equation*}
\Vert \tilde{v}_{i} \Vert  \leq  \vert \log(a) \vert^{h} \Vert u_{0}(z_{i}+a \; \cdot) \Vert + \,a \; \vert \log(a) \vert^{h-1} \; \Vert \tilde{v}_{i} \Vert + \,a \,\vert \log(a) \vert^{h-1}  \; \sum_{m \neq i}^{M} \Vert  \tilde{v}_{m} \Vert, 
\end{equation*}
or
\begin{equation*}
\Vert \tilde{v}_{i} \Vert \, (1-\,a \; \vert \log(a) \vert^{h-1})   \leq  \vert \log(a) \vert^{h} \Vert u_{0}(z_{i}+a \; \cdot) \Vert  + \,a \,\vert \log(a) \vert^{h-1}  \; \sum_{m \neq i}^{M} \Vert  \tilde{v}_{m} \Vert, 
\end{equation*}
hence
\begin{eqnarray*}
\Vert \tilde{v}_{i} \Vert_{\mathbb{L}^{2}(B)} & \leq & \vert \log(a) \vert^{h} \Vert u_{0}(z_{i}+a \; \cdot) \Vert_{\mathbb{L}^{2}(B)}  + \,a \,\vert \log(a) \vert^{h-1}  \; \sum_{m \neq i}^{M} \Vert  \tilde{v}_{m} \Vert_{\mathbb{L}^{2}(B)} \\
\Vert \tilde{v}_{i} \Vert_{\mathbb{L}^{2}(B)}^{2} & \leq & \vert \log(a) \vert^{2h} \Vert u_{0}(z_{i}+a \; \cdot) \Vert_{\mathbb{L}^{2}(B)}^{2}  + \,a^{2} \,\vert \log(a) \vert^{2h-2} \;  \;M \;  \sum_{m \neq i}^{M} \Vert  \tilde{v}_{m} \Vert_{\mathbb{L}^{2}(B)}^{2} \\
\Vert \tilde{v}_{i} \Vert_{\mathbb{L}^{2}(B)}^{2} & \leq & \vert \log(a) \vert^{2h} \Vert u_{0}(z_{i}+a \; \cdot) \Vert_{\mathbb{L}^{2}(B)}^{2}  + \,a^{2} \,\vert \log(a) \vert^{2h-2}   \; M \,  \Vert  \tilde{u} \Vert_{(\Pi \; \mathbb{L}^{2}(B))}^{2}, 
\end{eqnarray*}
we sum up to $M$, to obtain
\begin{eqnarray*}
\Vert \tilde{u} \Vert_{(\Pi \; \mathbb{L}^{2}(B))}^{2} & \leq & \vert \log(a) \vert^{2h} \Vert \tilde{u}_{0} \Vert_{(\Pi \; \mathbb{L}^{2}(B))}^{2}  + M^{2} \,a^{2} \,\vert \log(a) \vert^{2h-2}   \;   \Vert  \tilde{u} \Vert_{(\Pi \; \mathbb{L}^{2}(B))}^{2} \\
(1-M^{2} \,a^{2} \,\vert \log(a) \vert^{2h-2} ) \, \Vert \tilde{u} \Vert_{(\Pi \; \mathbb{L}^{2}(B))}^{2} & \leq & \vert \log(a) \vert^{2h} \Vert \tilde{u}_{0} \Vert_{(\Pi  \; \mathbb{L}^{2}(B))}^{2} \\
\Vert \tilde{u} \Vert_{(\Pi \; \mathbb{L}^{2}(B))}^{2}  & \leq & \vert \log(a) \vert^{2h} \Vert \tilde{u}_{0} \Vert_{(\Pi \; \mathbb{L}^{2}(B))}^{2}. 
\end{eqnarray*}
We obtain after scaling back
\begin{equation}\label{apmp}
\Vert u \Vert_{\mathbb{L}^{2}(D)} \leq  \vert \log(a) \vert^{h} \Vert u_{0} \Vert_{\mathbb{L}^{2}(D)}. 
\end{equation}
\end{itemize} 
\end{proof}
In the next proposition, which is analogous to proposition $(\ref{Y})$, we estimate the Fourier coefficient of the total field for dimer particles when $n \neq n_{0}$.
\begin{proposition}\label{X}
For $n \neq n_{0}$, we have
\begin{equation}\label{W}
<u_{2};e^{(i)}_{n}> = \frac{1}{(1 -  \omega^{2} \, \mu_{0} \, \tau  \, \lambda_{n})} \Big[ < u_{0},e^{(i)}_{n}> + \mathcal{O}(\vert \log(a) \vert^{-h}) \,\vert <1,e^{(i)}_{n}> \vert \Big], \; i=1,2.
\end{equation} 
\end{proposition}
\begin{proof}
First of all, recall that $v_{m} = u_{|_{D_{m}}}, m=1,2$ and let $n \neq n_{0}$. Take the scalar  product of $(\ref{equa710})$ with respect to $\overline{e}^{(i)}_{n}$, $i=1,2$,  to obtain 
\begin{eqnarray*}
<\tilde{v}_{1};\overline{e}_{n}> & - & \omega^{2} \, \mu_{0} \, \tau \, a^{2} \, \int_{B} \overline{e}_{n}(\eta) \int_{B} \Phi_{0}(\eta,\xi) \tilde{v}_{1}(\xi) \, d\xi \, d\eta \\ &=& < \tilde{u}_{0},\overline{e}_{n}> - \omega^{2} \, \mu_{0} \, \tau \, a^{2} \, \frac{1}{2\pi} \, \log(a) \, \int_{B} \tilde{v}_{1} \, d\xi \, \int_{B} \overline{e}_{n} d\eta \\
&+& \omega^{2} \, \mu_{0} \, \tau \,a^{2} \, \Bigg[ a \,\int_{B} \, \overline{e}_{n}(\eta) \int_{B} \vert \eta-\xi \vert \, \log(\vert \eta - \xi \vert) \, \tilde{v}_{1}(\xi) d\xi d\eta \\ 
 &+&  a \, \log(a) \int_{B} \, \overline{e}_{n}(\eta) \int_{B} \vert \eta-\xi \vert \,\tilde{v}_{1}(\xi) d\xi \, d\eta + \Big( \frac{-1}{2\pi}\log(k)+\Gamma \Big) \,  \int_{B} \big(\tilde{v}_{1}+\tilde{v}_{2} \big) \, d\xi \, \int_{B} \overline{e}_{n} d\eta \\ 
&-& \frac{1}{2 \pi}  \,\int_{B} \, \overline{e}_{n}(\eta)  \int_{B}  \log \vert (z_{1}-z_{2})+a(\eta - \xi) \vert \, \tilde{v}_{2}(\xi) \, d\xi \, d\eta + \mathcal{O}\Big(  a \, \int_{B} \big( \tilde{v}_{1} + \tilde{v}_{2} \big) \, d\xi \Big) \\
&+& \int_{B} \overline{e}_{n}(\eta) \int_{B} \vert (z_{1}-z_{2})+a(\eta - \xi)\vert \, \log(\vert (z_{1}-z_{2})+a(\eta - \xi)\vert) \tilde{v}_{2}(\xi) \, d\xi \, d\eta \Bigg]_{:=\textit{error}}.
\end{eqnarray*}
The $\textit{error}$ part, with the help of Taylor's formula, behaves as $\mathcal{O}\big( \vert \log(a) \vert^{1-h} \big) \vert <1,\overline{e}_{n}> \vert $ and we can write 
\begin{eqnarray*}
\int_{B} \overline{e}_{n} \int_{B} \Phi_{0} \tilde{v}_{1} \, d\xi \, d\eta &=& <\tilde{v}_{1};\overline{e}_{n}> \, \int_{B} \overline{e}_{n} \int_{B} \Phi_{0} \overline{e}_{n} \, d\xi \, d\eta + \sum_{j \neq n} <\tilde{v}_{1};\overline{e}_{j}> \, \int_{B} \overline{e}_{n} \int_{B} \Phi_{0} \overline{e}_{j}  d\xi \, d\eta 
\end{eqnarray*}
\begin{eqnarray*}
\phantom \quad  &\stackrel{(\ref{ha})}=& <\tilde{v}_{1};\overline{e}_{n}> \,\bigg[ \frac{\lambda_{n}}{a^{2}} + \boldsymbol{\frac{1}{2\pi} \log(a) \Big(\int_{B} \overline{e}_{n} d\eta\Big)^{2}} \bigg]   \stackrel{(\ref{whennneqm})} + \, \frac{1}{2\pi} \, \log(a) \, \int_{B} \overline{e}_{n} d\eta \sum_{j \neq n} <\tilde{v}_{1};\overline{e}_{j}>  \int_{B} \overline{e}_{j} d\eta.
\end{eqnarray*}
we plug all this in the previous equation to obtain 
\begin{eqnarray*}
<\tilde{v}_{1};\overline{e}_{n}> & - & \boldsymbol{\omega^{2} \, \mu_{0} \, \tau \, a^{2}} \, \Bigg[\boldsymbol{<\tilde{v}_{1};\overline{e}_{n}>} \, \int_{B} \overline{e}_{n}(\eta) \int_{B} \Phi_{0}(\eta,\xi) \overline{e}_{n}(\xi) \, d\xi \, d\eta  \\ 
&& \qquad \qquad \quad + \sum_{j \neq n} <\tilde{v}_{1};\overline{e}_{j}> \, \int_{B} \overline{e}_{n}(\eta) \int_{B} \Phi_{0}(\eta,\xi) \overline{e}_{j}(\xi) \, d\xi \, d\eta \Bigg] = < \tilde{u}_{0},\overline{e}_{n}> \\ &-&   \boldsymbol{\omega^{2} \, \mu_{0} \, \tau \, a^{2} \, \frac{1}{2\pi} \, \log(a)} \, \Bigg[\boldsymbol{<\tilde{v}_{1};\overline{e}_{n}> \,\int_{B} \overline{e}_{n} d\eta} + \sum_{j \neq n} <\tilde{v}_{1};\overline{e}_{j}> \, \int_{B}  \overline{e}_{j} \, d\xi \, \Bigg] \, \boldsymbol{\int_{B} \overline{e}_{n} d\eta} \\ &+& \mathcal{O}(\vert \log(a) \vert^{-h}) \, \vert <1,\overline{e}_{n}> \vert.
\end{eqnarray*}
Next, we cancel the two terms given by series and those written with \textbf{bold symbol} and scale back the obtained formula to get $(\ref{W})$.  
\end{proof}
The result in $(\ref{W})$ also applies to the case $n = n_{0}$ with an error term of order $\mathcal{O}\big(\vert \log(a) \vert^{-h}\big)$.\\
The next proposition improves the error term by improving the denominator term. 
\begin{proposition}
We have
\begin{equation}\label{v&V}
<u_{2};e^{(i)}_{n_{0}}> = \frac{<u_{0};e^{(i)}_{n_{0}}>}{(1-\omega^{2} \, \mu_{0} \, \tau \, \lambda_{n_{0}}) - \omega^{2} \, \mu_{0} \, \tau \, a^{2} \, \Phi_{0}(z_{1};z_{2}) \Big( \int_{B} \overline{e}_{n_{0}} \Big)^{2} } \;  + \mathcal{O}(a), \qquad i=1,2.
\end{equation}
\end{proposition}
\begin{proof}
In order to prove equality $(\ref{v&V})$ we take a scalar product with respect to $\overline{e}_{n_{0}}$ at the equation $(\ref{equa710})$, and after simplifications, we get:
\begin{equation}\label{algsystm}
\begin{bmatrix} 
(1-\omega^{2}\,\mu_{0}\,\lambda_{n_{0}}\, \tau) & -\omega^{2}\,\mu_{0}\,\tau\,a^{2} \, \Phi_{0}\, \Big( \int_{B} \overline{e}_{n_{0}} \Big)^{2} \\
-\omega^{2}\,\mu_{0}\,\tau\,a^{2} \, \Phi_{0} \, \Big( \int_{B} \overline{e}_{n_{0}} \Big)^{2} & (1-\omega^{2}\,\mu_{0}\,\lambda_{n_{0}}\, \tau) 
\end{bmatrix}
\begin{bmatrix} 
<\tilde{u}_{2};\overline{e}^{(1)}_{n_{0}}>  \\
\\
<\tilde{u}_{2};\overline{e}^{(2)}_{n_{0}}>  
\end{bmatrix}
=
\begin{bmatrix} 
<\tilde{u}_{0};\overline{e}^{(1)}_{n_{0}}> + \mathcal{O}(\vert \log(a) \vert^{-h}) \\
\\
<\tilde{u}_{0};\overline{e}^{(2)}_{n_{0}}>  + \mathcal{O}(\vert \log(a) \vert^{-h})
\end{bmatrix}
\end{equation}
We denote by $det$ the determinant of the last matrix, i.e 
\begin{equation}\label{defdet}
det = \Big(1-\omega^{2}\,\mu_{0}\,\lambda_{n_{0}}\, \tau\Big)^{2} - \Big( \omega^{2}\,\mu_{0}\,\tau\,a^{2} \, \Phi_{0} \, \Big( \int_{B} \overline{e}_{n_{0}} \Big)^{2} \Big)^{2}, \quad \text{where} \quad \Phi_{0} = \Phi_{0}(z_{1};z_{2})
\end{equation}
Next, we check that when we are close to the resonance the determinant $det \neq 0$. For this, and by construction of $\omega^{2}$, we have
\begin{equation*}
1 - \omega^{2} \mu_{0} \tau \lambda_{n_{0}} = \mp \vert \log(a) \vert^{-h},
\end{equation*}
and the fact that \; $d \sim a^{\vert \log(a) \vert^{-h}}$ implies that $ \tau a^{2} \Phi_{0}(z_{1},z_{2}) \sim \frac{1}{2\pi} \, \vert \log(a) \vert^{-h}$. Plug this in $(\ref{defdet})$  to obtain 
\begin{equation*}
det = \vert \log(a) \vert^{-2h} \, \Bigg[1 - \bigg( \omega^{2} \mu_{0} \, \frac{1}{2\pi}  \, \big( \int \overline{e}_{n_{0}} \big)^{2} \bigg)^{2} \Bigg] 
\stackrel{(\ref{H})}= \vert \log(a) \vert^{-2h} \, \left[1 -  \frac{\left( 1 \pm \vert \log(a) \vert^{-h} \right)}{\left( 1+\frac{\tilde{\lambda}_{n_{0}} \vert \log(a) \vert^{-1}}{\frac{1}{2\pi}  \, \big( \int \overline{e}_{n_{0}} \big)^{2}}\right)^{2}}  \,   \right]
\end{equation*}
from {\bf{Hypotheses}\ref{hyp}},  we deduce that  
\begin{equation*}
\left(\frac{\tilde{\lambda}_{n_{0}} \vert \log(a) \vert^{-1}}{\frac{1}{2\pi}  \, \big( \int \overline{e}_{n_{0}} \big)^{2}}\right) \sim \vert \log(a) \vert^{-1},
\end{equation*}
then  
\begin{equation*}
det = \vert \log(a) \vert^{-2h} \left[1 - \left( 1 \pm \vert \log(a) \vert^{-h} \right) \left( 1 + \vert \log(a) \vert^{-1} \right) \right] \sim \vert \log(a) \vert^{-3h}.
\end{equation*}
Since $det \neq 0$, the algebraic system $(\ref{algsystm})$ is invertible. We invert it and  
use the fact that
\begin{equation*}
<\tilde{u}_{0};\overline{e}^{(2)}_{n_{0}}>  = <\tilde{u}_{0};\overline{e}^{(1)}_{n_{0}}> + \mathcal{O}(d),
\end{equation*}
to obtain
\begin{equation}\label{equa825}
<\tilde{u}_{2};\overline{e}^{(1)}_{n_{0}}>  =  \frac{<\tilde{u}_{0};\overline{e}^{(1)}_{n_{0}}>}{(1-\omega^{2}\,\mu_{0}\,\tau\,\lambda_{n_{0}})-\omega^{2}\, \mu_{0}\, \tau \, a^{2} \, \Phi_{0}(z_{1};z_{2}) \, \Big( \int_{B} \overline{e}_{n_{0}} \Big)^{2}}
 +  \mathcal{O}\big(1\big),
\end{equation}
and, after scaling, we get $(\ref{v&V})$.
\end{proof}
\subsection{Estimation of the scattering coefficient $\textbf{C}$}\
\\From $(\ref{A})$ we have: 
\begin{equation*}
w =  \omega^{2} \, \mu_{0} \, \tau \Big[I - \omega^{2} \, \mu_{0} \, \tau \, A_{0} \Big]^{-1}(1)  \quad \text{or} \quad   \frac{1}{\omega^{2} \, \mu_{0} \, \tau} \; \Big[I - \omega^{2} \, \mu_{0} \, \tau \, A_{0} \Big](w) = 1.  
\end{equation*} 
Hence
\begin{equation*}
<1,e_{n}>  =   \frac{1}{\omega^{2} \, \mu_{0} \, \tau} <e_{n} ;  \big[I - \omega^{2} \, \mu_{0} \, \tau \, A_{0} \big](w)>   =  \frac{1}{\omega^{2} \, \mu_{0} \, \tau} \;  \Big[<e_{n},w> - \omega^{2} \, \mu_{0} \, \tau \; \lambda_{n} <e_{n},w> \Big]
\end{equation*}  
and then
\begin{equation}\label{A1}
<w,e_{n}> = \frac{\omega^{2} \, \mu_{0} \, \tau}{1-\omega^{2} \, \mu_{0} \, \tau \; \lambda_{n}} <1,e_{n}>.
\end{equation}
The next lemma uses $(\ref{A1})$ to gives a precision about the value of $\textbf{C}$. 
\begin{lemma} 
The coefficient $\textbf{C}$ can be approximated as
\begin{equation}\label{F}
\textbf{C} =   \frac{\omega^{2} \, \mu_{0} \, \tau}{(1-\omega^{2} \, \mu_{0} \, \tau \, \lambda_{n_{0}})} \, \Big( \int_{D} e_{n_{0}} \Big)^{2} + \mathcal{O}(\vert \log(a) \vert^{-1}). 
\end{equation}
\end{lemma} 
\begin{proof}
We use the definition of $\textbf{C}$, given by $(\ref{R})$, to write
\begin{equation*}
\textbf{C} := \int_{D} w \, dx = \sum_{n}  <w,e_{n}> \, <1,e_{n}>,
\end{equation*}
apply $(\ref{A1})$ to obtain
\begin{equation*}
\textbf{C}  =  \omega^{2} \, \mu_{0} \, \tau \, \Bigg[ \frac{1}{(1-\omega^{2} \, \mu_{0} \, \tau \, \lambda_{n_{0}})} \, \Big( \int_{D} e_{n_{0}} \Big)^{2} + \sum_{n \neq n_{0}}  \frac{1}{(1-\omega^{2} \, \mu_{0} \, \tau \, \lambda_{n})} \, \Big( \int_{D} e_{n} \Big)^{2}  \Bigg], 
\end{equation*}
and, since the frequency $\omega$ is near $\omega_{n_{0}}$, and hence away from the other resonances we have
\begin{equation*}
\Big\vert \sum_{n \neq n_{0}}  \frac{1}{(1-\omega^{2} \, \mu_{0} \, \tau \, \lambda_{n})} \, \Big( \int_{D} e_{n} \Big)^{2} \Big\vert \leq \sum_{n} \vert <1,e_{n}> \vert^{2} = \Vert 1 \Vert_{\mathbb{L}^{2}(D)} = \mathcal{O}\big(a^{2}\big). 
\end{equation*} 

\end{proof}
From $(\ref{F})$, we see that  
\begin{equation*}\label{B}
\textbf{C} \sim \vert \log(a) \vert^{h-1}.
\end{equation*}
We deduce also the following formula: 
\begin{equation}\label{C}
(1-\omega^{2} \, \mu_{0} \, \tau \, \lambda_{n_{0}}) = \textbf{C}^{-1} \, \omega^{2} \, \mu_{0} \, \tau \, \Big( \int_{D} e_{n_{0}} \Big)^{2} + \mathcal{O}(\vert \log(a) \vert^{-2h}).
\end{equation}

\section{Appendix}\label{the hypotheses-justification}
 To motivate the natural character of the hypotheses stated in {\bf{Hypotheses}} \ref{hyp}, let us make the following observations: 
\begin{enumerate}
\item[a)] We prove that the upper bound of $\lambda_{n}$ is of order $a^{2} \, \vert \log(a) \vert$. For this, 
recalling and rescaling $(\ref{U})$ we obtain, see section \ref{appendixlemma}, in particular (\ref{G}), for $a<<1,$
\begin{equation}\label{T}
\lambda_{n} = a^{2} \, \Big( {\tilde{\lambda}_{n}} + \frac{1}{2} \vert \log(a) \vert (\int_{B}\bar{e_n}(\xi) d\xi)^2 \Big), 
\end{equation}
where 
$${\tilde{\lambda}_{n}}:=\frac{1}{\Vert\tilde{e_n}\Vert_{\mathbb{L}^2(B)}^2}\int_B LP(\tilde{e_n})(\eta)~\tilde{e_n}(\eta) d\eta $$
 and $\tilde{e}_{n}$ is the scaled of any eigenfunction $e_n$ corresponding to $\lambda_n$.
Take the  absolute value in $(\ref{T})$ to obtain  
\begin{equation*}
\vert \lambda_{n} \vert \leq a^{2} \, \left( \vert {\tilde{\lambda}}_{n} \vert + \frac{1}{2} \vert \log(a) \vert \, \vert <1;\bar{e}_{n}> \vert^{2} \right).
\end{equation*}
From the definition of $\tilde{\lambda}_{n}$, see $(\ref{lambdatilde})$, we have $\vert \tilde{\lambda}_{n} \vert \leq \Vert \Phi_{0} \Vert_{\mathbb{L}^{2}(B \times B)} < \infty $ 
and we use the \\ Cauchy-Schwartz inequality to obtain
\begin{equation*}
\vert \lambda_{n} \vert \leq a^{2} \, \left( \Vert \Phi_{0} \Vert_{\mathbb{L}^{2}(B \times B)} + \frac{1}{2} \vert \log(a) \vert \, \vert B \vert^{2} \right) \lesssim a^{2} \, \vert \log(a) \vert.  
\end{equation*}
\item[b)] For the lower bound, the situation is less clear. Nevertheless, we have the following results:
\begin{enumerate}
\item[b.1)] When the shape is a disc of radius $a$, we refer to (Theorem 4.1, \cite{RG}) for the existence of a sequence of eigenvalues given by
\begin{equation*}
\lambda_{k,j} = a^{2} \, \left[\mu_{j}^{(k)}\right]^{-2}, \quad k=0,1,2,\cdots \; j=1,2,\cdots
\end{equation*} 
and the corresponding eigenfunctions given by
\begin{equation*}
u_{k,j}(r,\varphi) = \LARGE{\text{J}}_{k} \left( \mu_{j}^{(k)} \; r \; a^{-1} \right) \, e^{i \, k \, \varphi},
\end{equation*}
where $\LARGE{\text{J}}_{k}$ is the Bessel function of the first kind of order $k$ and $\mu_{j}^{(k)}$ are the roots of the following transcendental equation 
\begin{eqnarray}\label{our equa}
\nonumber
k \LARGE{\text{J}}_{k}\left( \mu_{j}^{(k)} \right) + \frac{\mu_{j}^{(k)}}{2} \left(  \LARGE{\text{J}}_{k-1}\left( \mu_{j}^{(k)} \right) - \LARGE{\text{J}}_{k+1}\left( \mu_{j}^{(k)} \right)\right) &=& 0, \qquad k=1,2,\cdots \\ 
\LARGE{\text{J}}_{0}\left( \mu_{j}^{(0)} \right) - \mu_{j}^{(0)} \, \log(a) \, \left( \LARGE{\text{J}}_{-1}\left( \mu_{j}^{(0)} \right)-\LARGE{\text{J}}_{1}\left( \mu_{j}^{(0)} \right) \right) &=& 0.  
\end{eqnarray}
We remark that (only) for $k = 0$, the associated eigenfunctions have a non zero average\footnote{We can compute
$\underset{D}{\int} u_{0,j} = \int_{0}^{2\pi} \int_{0}^{a} u_{0,j}(r,\varphi) \, r \, dr \, d\varphi = 2 \pi \, a^{2}  \, \LARGE{\text{J}}_{1}\left(\mu_{j}^{(0)}\right) / \mu_{j}^{(0)}.$}.\\
Next, in order to obtain a precision about the behaviour of $\{ \lambda_{0,j} \}_{j \geq 1}$ with respect to $a$, we need to investigate the behaviour of $\mu_{j}^{(0)}$  solutions of $(\ref{our equa})$. For this, we use the following properties of Bessel functions 
\begin{equation*}
\LARGE{\text{J}}_{-1}(x)-\LARGE{\text{J}}_{1}(x) = 2 \LARGE{\text{J}}^{\prime}_{0}(x) = - 2 \LARGE{\text{J}}_{1}(x),
\end{equation*}
to write $(\ref{our equa})$ as
\begin{equation*}
\LARGE{\text{J}}_{0} \Big(\mu_{j}^{(0)}\Big)  + 2 \, \log(a) \, \mu_{j}^{(0)} \LARGE{\text{J}}_{1} \Big(\mu_{j}^{(0)} \Big) = 0.
\end{equation*}
Set $\Psi(x) := \LARGE{\text{J}}_{0} ( x )  + 2 \, \log(a) \, x \, \LARGE{\text{J}}_{1} ( x )$ and use \emph{Dixon's} theorem, see \cite{Watson} page 480, to deduce that the roots of $\Psi$ are interlaced with those of  $\LARGE{\text{J}}_{0}$, noted by $\{ x_{0,j} \}_{j \geq 1}$, and those of $\LARGE{\text{J}}_{1}$, noted by $\{ x_{1,j} \}_{j \geq 1}$. At this stage, we distinguish two cases 
\begin{itemize}
\item[$\star$] The roots of $\Psi$ exceeding $x_{0,1}$: 
\end{itemize}
For this case, a direct application of \emph{Dixon's} theorem, allows to deduce that
\begin{equation*}
\forall j \geq 2,  \; x_{k,j-1} < \mu_{j}^{(0)} < x_{k,j}, \quad k=0,1
\end{equation*}
and
\begin{equation*}
\forall j \geq 2, \;  a^{2} \, x^{-2}_{k,j} < \lambda_{0,j} < a^{2} \, x^{-2}_{k,j-1}, \quad k=0,1,
\end{equation*}
since $\big\lbrace x_{k,j} \big\rbrace_{j \geq 1 \atop k=0,1}$ are independent of \, $a$ \, we deduce that $\lambda_{0,j}$ behaves as \, $a^{2}$. 
\begin{itemize}
\item[$\star$] The root of $\Psi$ less than $x_{0,1}$:
\end{itemize}
The analysis of this case is more delicate. First, we observe that if, for a certain $x$, $\Psi(x) = 0$, then $\LARGE{\text{J}}_{0}(x) \neq 0$. Otherwise, we would have also $\LARGE{\text{J}}_{1}(x) = 0$ which is impossible as the zeros of $\LARGE{\text{J}}_{0}$ and $\LARGE{\text{J}}_{1}$ are disjoint, see \emph{Bourget's Hypothesis}, page 484, section 15.28 in \cite{Watson}. Hence the equation $\Psi(x) = 0$ can be rewritten as 
\begin{equation}\label{log=F0}
\frac{1}{2 \, \log(a)} = \frac{-x \, \LARGE{\text{J}}_{1} ( x )}{\LARGE{\text{J}}_{0} ( x )} := \textbf{F}_{0}(x).
\end{equation}
Clearly, $\textbf{F}_{0}$ is a smooth function on each interval not containing a zero of $\LARGE{\text{J}}_{0}$ and from \cite{landau}, see equation 27, we deduce that it is also a decreasing function, (see  figure  $\ref{F0}$, for a schematic picture).
\begin{figure}[H]
  \centering
  \begin{minipage}[b]{0.4\textwidth}
 \includegraphics[scale=0.4]{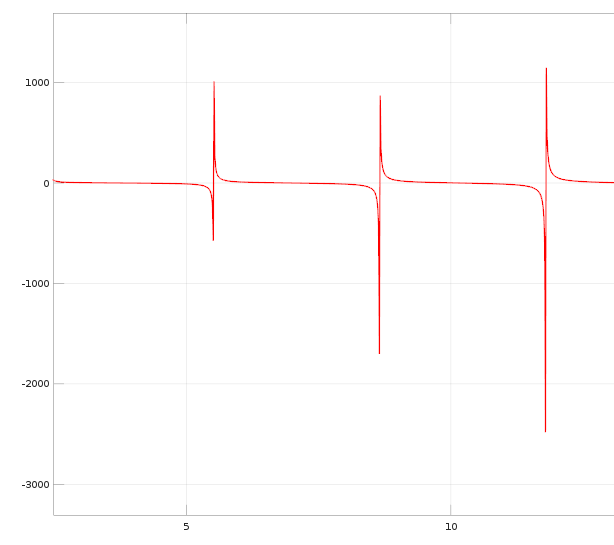}
    \caption{Graphic of $\textbf{F}_{0}$.}
    \label{F0}
  \end{minipage}
  \hfill
  \begin{minipage}[b]{0.5\textwidth}
 \includegraphics[scale=0.4]{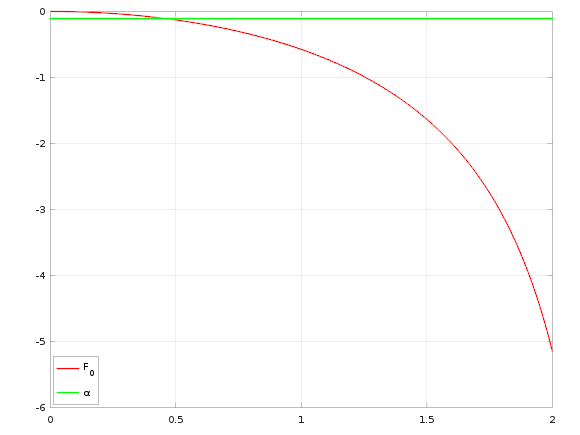}

    \caption{Solving, for $x \in (0,\nu)$,  $\textbf{F}_{0}(x)=1/(-4 \, \log(10))$.}
    \label{F0=alpha} 
  \end{minipage}
\end{figure}
 So, if we restrict our study to $(0 , \nu)$ with $\nu < x_{0,1} $ we deduce that $\textbf{F}_{0}^{-1}$ exists and is continuous, then the equation $(\ref{log=F0})$ is solvable and the solution that we obtain is also small, (see figure $(\ref{F0=alpha})$, for numerical demonstration).    \\ 
Now, since $x$ is small we use the asymptotic behaviour of $\textbf{F}_{0}$, see for instance (equation 25 in \cite{landau}), $\textbf{F}_{0}(x) \sim -x^{2}/2$ to write $(\ref{log=F0})$ as 
\begin{equation*}
\frac{1}{2 \, \log(a)} \sim \frac{-x^{2}}{2},
\end{equation*} 
and this implies that $ x \sim \Big( \log(1/a) \Big)^{\frac{-1}{2}}$. Finally 
\begin{equation}
\lambda_{0,1} \sim a^{2} \, \vert \log(a) \vert
\end{equation}
\item[b.2)] For the case of an arbitrary shape $D$, with  $\vert D \vert = \vert B_a \vert$ where $B_a$ is the disc of radius $a$, and referring to (Theorem 2.5, \cite{MRDS}) 
we have $\Vert LP_{D} \Vert \leq \Vert LP_{B_a} \Vert$. From the definition of $\Vert LP_{D} \Vert$, we write this inequality as a Faber-Krahn type inequality 
\begin{equation*}
\frac{1}{\lambda_{0,1}^2(D)} = \Vert LP_{\Omega} \Vert \leq \Vert LP_{D} \Vert = \frac{1}{\lambda_1^2(B_a)} \quad \text{or equivalently}  \quad \lambda_{1}(B_a) \geq \lambda_{0,1}(D).
\end{equation*}
We deduce the lower bound, and hence the behaviour, of the first eigenvalue 
\begin{equation}
\lambda_{0,1}(D) ~ \sim a^2  \vert \log(a)\vert, \quad \forall a <<1.
\end{equation}
In addition from $(\ref{T})$, we see that 
\begin{equation*}
\left( \int_{D} e_{1} \right)^{2} = \frac{\lambda_{0,1}}{a^{2} \, \vert \log(a) \vert} + \mathcal{O}\big(\vert \log(a) \vert^{-1}\big)
\end{equation*}
and hence 
\begin{equation}
\left( \int_{D} e_{1} \right)^{2} \sim 1 \;\; for \;\; 	a<<1.
\end{equation}
\end{enumerate}
\end{enumerate}

\end{document}